\documentclass[11pt]{amsart}        
\usepackage{baskervald}
\usepackage[swedish, english]{babel}
\usepackage{hyperref}
\usepackage{xcolor}
\usepackage[latin1]{inputenc}
\usepackage{accents}
\usepackage{adjustbox}
\usepackage{url}
\usepackage[hmargin=3.3cm,vmargin=3.5cm]{geometry} 
\usepackage{graphicx}                         
\usepackage{amsmath}
\graphicspath{ {./images/} }
\usepackage{tabularx}
\usepackage{multirow}
\usepackage{makecell}
\usepackage[noabbrev,capitalize]{cleveref}
\hypersetup{
	colorlinks,
	linkcolor={red!50!black},
	citecolor={blue!50!black},
	urlcolor={blue!80!black}
}

\usepackage{amsthm}                      
\usepackage{amssymb} 
\usepackage{mathrsfs}
\usepackage{stmaryrd}
\usepackage{mathtools}
\usepackage[shortlabels]{enumitem}
\usepackage{tikz-cd}
\usetikzlibrary{babel}
\DeclareFontFamily{U}{wncy}{}
\DeclareFontShape{U}{wncy}{m}{n}{<->wncyr10}{}
\DeclareSymbolFont{mcy}{U}{wncy}{m}{n}
\DeclareMathSymbol{\sha}{\mathord}{mcy}{"58} 
\usepackage{euscript}
\usepackage[normalem]{ulem}
\newcommand{\Eu}[1]{\EuScript{#1}}
\swapnumbers
\theoremstyle{plain}                                              \makeatletter
\def\swappedhead#1#2#3{%
	\thmnumber{\@upn{\the\thm@headfont#2\@ifnotempty{#1}{.~}}}%
	\thmname{#1}%
	\thmnote{ {\the\thm@notefont(#3)}}}
\makeatother
\newtheorem{thm}{Theorem}[subsection]
\newtheorem*{theorem*}{Theorem}
\newtheorem{state}{Theorem}

\newtheorem{lem}[thm]{Lemma}
\newtheorem{prop}[thm]{Proposition}
\newtheorem{cor}[thm]{Corollary}

\newtheorem*{conjecture*}{Conjecture}

\theoremstyle{definition}

\newtheorem{rem}[thm]{Remark}
\newtheorem{defn}[thm]{Definition}

\newtheorem{notation}[thm]{Notation}
\newtheorem{para}[thm]{}

\newcommand{\F}{\mathbf{F}}
\newcommand{\Q}{\mathbf{Q}}
\newcommand{\C}{\mathbf{C}}
\newcommand{\R}{\mathbf{R}}

\newcommand{\Z}{\mathbf{Z}}

\newcommand{\GL}{\mathrm{GL}}

\def\res{\operatorname*{Res}}

\renewcommand{\aa}{\alpha}

\newcommand{\tr}{\operatorname{Tr}}
\newcommand{\cond}{\operatorname{Cond}}
\newcommand{\ord}{\operatorname{ord}}

\def\aux{\mathrm{aux}}
\def\fin{\mathrm{fin}}
\def\sb{\mathbf{s}}
\def\db{\mathbf{d}}
\def\eb{\mathbf{e}}
\def\sss{\scriptscriptstyle}
\def\ZZ{\mathbb{Z}}
\def\dd{\delta}
\def\eps{\epsilon}
\def\psym#1#2{\Big(\frac{#1}{#2}\Big)}
\def\lam{\lambda}
\def\dlog{\operatorname{dlog}}
\def\rad{\operatorname{rad}}
\def\Sc{\Eu{S}}
\def\Rc{\Eu{R}}

\def\OO{\mathscr{O}}
\def\EE{\mathscr{E}}
\def\ai{\mathfrak{a}}
\def\bi{\mathfrak{b}}
\def\ci{\mathfrak{c}}
\def\di{\mathfrak{d}}
\def\ei{\mathfrak{e}}
\def\ff{\mathfrak{f}}
\def\gi{\mathfrak{g}}
\def\hi{\mathfrak{h}}

\def\mi{\mathfrak{m}}
\def\ni{\mathfrak{n}}
\def\qi{\mathfrak{q}}
\def\ri{\mathfrak{r}}
\def\ti{\mathfrak{t}}
\def\pp{\mathfrak{p}}
\def\Ci{\mathfrak{C}}

\def\ji{\mathfrak{j}}

\def\sf{\mathrm{sf}}
\def\full{\mathrm{full}}

\def\vp{\varphi}
\def\g{\gamma}
\def\xx{\mathbf{x}}
\def\wZ{\widetilde{Z}}
\def\wW{\widehat{W}}
\def\wC{\widetilde{C}}

\def\wD{\widetilde{D}}

\def\wg{\widetilde{g}}
\def\wf{\widetilde{f}}
\def\wdd{\widetilde{\delta}}

\newcommand{\comment}[1]{}
\newcommand{\ov}[1]{\overline{#1}}

\def\be{\begin{equation}}  \def\ee{\end{equation}}
\def\lle{\ll_{\epsilon}|\hi|^\eps}

\def\Ac{\Eu{A}}
\def\Bc{\Eu{B}}

\renewcommand{\b}{\mathfrak{b}}
\newcommand{\bt}{\tilde{\b}}
\newcommand{\q}{\mathfrak{q}}

\newcommand{\h}{\mathfrak{h}}

\newcommand{\m}{\mathfrak{m}}

\DeclareMathOperator\supp{supp}

\title[]{On the second moment and non-vanishing of central values
of Hecke $L$-functions of $r$-th order characters}
\author{Adrian Diaconu, Bogdan Ion, Vicen\c tiu Pa\c sol, Alexandru A. Popa}
\newcommand{\Addresses}{{\footnotesize
\bigskip
\noindent\textsc{Adrian Diaconu} \\
 University of Minnesota\\
 School of Mathematics, 127 Vincent Hall\\
 Minneapolis, MN 55455, USA;\\
 Institute of Mathematics of the Romanian Academy\\
  Calea Grivi\c tei 21, Bucharest, Romania\\
 Email: \texttt{cad@umn.edu}

\smallskip
\noindent\textsc{Bogdan Ion} \\
University of Pittsburgh\\
Department of Mathematics, 301 Thackeray Hall\\
Pittsburgh, PA 15260, USA\\
Email: \texttt{bion@pitt.edu}

\smallskip
\noindent\textsc{Vicen\c tiu Pa\c sol}\\
Institute of Mathematics of the Romanian Academy\\
Calea Grivi\c tei 21, Bucharest, Romania\\
Email: \texttt{vicentiu.pasol@gmail.com}

\smallskip
\noindent\textsc{Alexandru A. Popa}\\
Institute of Mathematics of the Romanian Academy\\
Calea Grivi\c tei 21, Bucharest, Romania\\
Email: \texttt{aapopa@gmail.com}
}}

\begin{document}

	\begin{abstract}
In this paper, we establish asymptotic formulas for the first and second twisted moments of $r$-th order Hecke $L$-functions over global fields that contain the $2r$-th roots of unity, for $r\ge 3$.
We focus primarily on algebraic number fields. As a consequence, we establish a positive proportion of non-vanishing central values for these $L$-functions, specifically for families of both square-free and $r$-th power-free ideals. Our approach is based on the machinery of multiple Dirichlet series.
\end{abstract}
	\maketitle

\section{Introduction}

\subsection{Summary of results}
We establish asymptotics for the twisted second moment of $r$-th
order Hecke $L$-functions over any global field. As a direct
consequence of our evaluation of the first and second mollified
moments, we prove that a positive proportion of the central values of
Hecke $L$-functions associated with primitive $r$-th order residue
symbols are nonzero over any number field containing the $2r$-th roots
of unity,  following the mollification approach of \cite{S00}, \cite{DFDS24}, and \cite{CFD26}.

Our approach to obtaining an asymptotic formula for the twisted second moment is completely different from the methods employed in \cite{S00} (for $r=2$), \cite{DFDS24} (for $r=3$),  and \cite{CFD26} (for $r=4$), and applies uniformly over any global field. It relies on the Multiple Dirichlet Series (MDS) machinery \cite{Di04} and the techniques developed in \cite{Di19} and in the joint work of the first author with Whitehead \cite{Di-Wh21} to study the cubic moment in the quadratic case.

Although the results we are about to describe hold in the general setting of any global field, we shall focus our exposition exclusively on the number field case.

\begin{para}
To be specific, let $r \ge 3$ be an integer and let $F$ be a number field containing the group $\mu_{2r}$ of $2r$-th roots of unity. Let $\mathscr{O}$ denote the ring of integers of $F$, and let $S_f$ be a finite set of nonarchimedean places containing all places which are ramified over $\Q$ (in particular, those dividing $r$) and such that the ring of $S_f$-integers $\mathscr{O}_{S_f}$ is a principal ideal domain. Let $S_{\infty}$ denote the set of archimedean places, and let $S = S_f \cup S_{\infty}$. For each place $v$, let $F_v$ denote the completion of $F$ at $v$. For $v$ nonarchimedean, let $\mathfrak{p}_v$ denote the corresponding prime ideal of $\mathscr{O}$, and let $q_v = |\mathfrak{p}_v| := |\mathscr{O}/\mathfrak{p}_v|$. Let $\mathfrak{c} = \prod_{v \in S_f} \mathfrak{p}_{v}^{n_v}$, with $n_v = 1$ if $\mathrm{ord}_{v}(r) = 0$, and where $n_{v}$ is sufficiently large so that if $a \in F_{v}^{\times}$ satisfies $\mathrm{ord}_{v}(a - 1) \ge n_v$, then $a \in (F_{v}^{\times})^{r}$.

Let $I(S)$ denote the set of nonzero ideals of $\mathscr{O}$ coprime to $S_f$. For $\mathfrak{a}, \mathfrak{b} \in I(S)$ coprime, let
$
\chi_{\mathfrak{a}}(\mathfrak{b}) = (\frac{\mathfrak{a}}{\mathfrak{b}})
$
denote the Fisher--Friedberg extension of the $r$-th power residue symbol \cite{FHL03, FF04}. This is a Hecke character modulo $\mathfrak{c}\mathfrak{a}$ of conductor
$
\mathfrak{c}_{\mathfrak{a}}\prod_{\substack{\mathfrak{p}_{v} \mid \mathfrak{a}\\ r \nmid \mathrm{ord}_{v}(\mathfrak{a})}}
\mathfrak{p}_{v}
$,
for some divisor $\mathfrak{c}_{\mathfrak{a}}$ of $\mathfrak{c}$.  Given a Hecke character $\psi$ unramified outside $S$, we define the partial $L$-function for $\Re(s) > 1$ by
$$
L^{S}(s, \psi \chi_{\mathfrak{a}})
= \prod_{v \notin S } (1 -
(\psi_{v} \chi_{\mathfrak{a}, v})(\varpi_v)q_{v}^{-s})^{-1}
$$
where $\varpi_v$ is a local uniformizer at $v$. We will denote by $\zeta_F^S(s)$ the partial Dedekind zeta-function of
$F$.
\end{para}
\begin{para}
Our first main result is an asymptotic formula for a twisted second moment (\cref{thm:twisted_2ndmom}); for simplicity, we state it here in its untwisted form.

\begin{state}\label{Main-Thm1-alt} Let $W(u)$ be a Schwartz function with compact support on $(1,2)$ satisfying $0\le W(u) \le 1$,  and let $A_r=1/2$ if $r>3$, and $A_3=1/3$.
 Then, for every $\varepsilon > 0$, we have
$$
\sum_{\ai \in I(S)} \mu^2 (\ai) \big|L^S(\tfrac{1}{2}, \chi_{\ai})\big|^2 W\big(\tfrac{|\ai|}{X}\big)= X R_W^{(1)}(\log X) + O\Big(X^{\frac{r+A_r}{r+1}+\varepsilon}\Big),
$$
for some explicit linear polynomial $R_W^{(1)}$ with leading coefficient
$$
\widehat{W}(1) \zeta_F^S(r/2)\Big[\res_{w=1} \zeta_F^S(w)\Big]^2 \prod_{v\notin S} \big(1-q_v^{-1}\big) P_r^{(1)} \big(q_v^{-1/2}\big),
$$
where $\widehat{W}(s)$ denotes the Mellin transform of $W(u)$ and
$P_r^{(1)}(x)=1+x^2-x^4+x^r(1-x^2+x^4)$.
\end{state}
For $r=3$ and $F=\Q(\sqrt{-3})$, this recovers the  David--de Faveri--Dunn--Stucky \cite{DFDS24} asymptotic formula with the same error term. For $r=4$ and $F=\Q(\sqrt{-1})$, this recovers the Castillo--de Faveri--Dunn \cite{CFD26} asymptotic formula with a better  power-saving error term (see Remark~\ref{rem:Qi}). For $r\ge 5$, \cref{Main-Thm1-alt} is the first asymptotic formula for the second moment in the square-free family.

\begin{rem}\label{cubic-case-assum-GRH-alt} As pointed out in \cite{DFDS24}, the asymptotic formula in \cref{Main-Thm1-alt} for $r=3$ and $F=\Q(\sqrt{-3})$ falls just short of capturing the second-order term $X^{5/6}Q_W^{(1)}(\log X)$ conjectured in \cite[Conjecture 4.5]{Di04}. It is interesting to note that our method shows that any improvement of the large sieve bound in \cite[Proposition 4.2]{DFDS24} for the second moment sum
\begin{equation} \label{mom-sum-cubic-alt}
S(Q_1, Q_2) \; = \sum_{\substack{a, b \in \Z[\omega] \\ a, b \equiv 1 \!\!\pmod 3 \\ |a| \asymp Q_1,\, |b| \asymp Q_2}} \mu^2 (ab)\big|L(\tfrac{1}{2}, \chi_{ab^2})\big|^2
\end{equation}
would improve the error term in \cref{Main-Thm1-alt}, and hence give an asymptotic formula capturing the conjectural second-order term. Note that by assuming the Generalized Riemann Hypothesis for $L(s, \chi_{ab^2})$,
we have
\[
S(Q_1, Q_2) \ll_{\varepsilon} (Q_1 Q_2)^{1 + \varepsilon} \quad \text{(for any $\varepsilon > 0$)}
\]
resulting in an error term $\ll_{\varepsilon} X^{\frac{3}{4}+\varepsilon}$ in the asymptotic formula.
The same barrier $X^{5/6}$ appears in the main term of Patterson's conjecture on the distribution of cubic Gauss sums, proved by Dunn and Radziwi\l\l ~\cite{DR} under the same GRH assumption.
The main obstacle to sharpening the bound on the sum $S(Q_1, Q_2)$ is that
the current approach depends on Heath-Brown's cubic large-sieve estimate  \cite[Theorem 2]{HB}. Indeed, very recently, de Faveri, Dunn, and Hoffstein \cite{FDH26} proved that  the term $(AB)^{2/3}$ in the cubic large sieve inequality cannot be removed, and they make a refined prediction  for the $r$-th order large sieve inequality for $r\geq 4$.
\end{rem}
\end{para}

\begin{para}
Our next result establishes an asymptotic formula for a twisted first moment (\cref{thm:twisted_first_moment}). As with the second moment, for the sake of simplicity, we state the theorem here in its untwisted form.

\begin{state}\label{Thm2-first-moment-alt} Let the notation be as in \cref{Main-Thm1-alt}. Then, for every $\varepsilon > 0$, we have
$$
\sum_{\ai \in I(S)} \mu^2 (\ai) L^S(\tfrac{1}{2}, \chi_{\ai}) W\big(\tfrac{|\ai|}{X}\big)=C_{1}\widehat{W}(1)X +F_1\widehat{W}(\tfrac{1}{2} + \tfrac{1}{r})X^{\frac{1}{2}+\frac{1}{r}} +  O\Big(X^{\frac{r+1 + A_r}{r+3}+\varepsilon}\Big),
$$
where
$$
C_{1} = \zeta_F^S(r/2)\res_{w=1}\zeta_F^S(w)
\prod_{v\notin S} \big(1-q_v^{-1}\big)\big(1+q_v^{-1}-q_v^{-(r+2)/2}\big).
$$
\end{state}
Note that our asymptotic formula captures the second-order term only for $r = 3$. The constant $F_1$ can be computed as a series involving the Whittaker--Fourier coefficients of theta functions on the $r$-fold cover of $\mathrm{GL}_2$. It is well known that for $r \ge 4$, these coefficients are far from being well understood (see \cite{Hof93, KP84, Kub69} for the general case, and \cite{EP92, Hof93, Su82} for the biquadratic case). In the cubic case where $r = 3$ and $F = \Q(\sqrt{-3})$, a complete description of the cubic theta function was obtained by Patterson \cite{Pat1}. Prior to our result, the only known asymptotic formulas for the first moment were those of Luo \cite{Luo} and the very recent work of Hamdar \cite{Ham}, both of which are restricted to the cubic case over $F = \Q(\sqrt{-3})$.
\end{para}

\begin{para} Using mollified first and second moments, we prove, as in \cite{S00}, \cite{DFDS24}, and \cite{CFD26},
the following:
\begin{state}\label{Nonvanishing-sq-free-alt} Let the notation be as above. Then, for a positive proportion of square-free ideals $\ai \in I(S)$, we have
	$L(\tfrac{1}{2}, \chi_{\ai}) \ne 0$. More precisely, for all sufficiently large $X$ and any fixed $\varepsilon > 0$, we have
	$$
	\sum_{\substack{\ai \in I(S) \\  |\ai| \le X\\ L(1/2, \chi_{\ai}) \ne 0}} \mu^2 (\ai)
	\ge
	\begin{cases}
		\left(\dfrac{1}{12}-\varepsilon\right)
		\displaystyle\sum_{\substack{\ai \in I(S)\\ |\ai| \le X}}\mu^2 (\ai)
		& \text{if } r=3,\\[1.2em]
		\left(\dfrac{1}{4r+2}-\varepsilon\right)
		\displaystyle\sum_{\substack{\ai \in I(S)\\ |\ai| \le X}} \mu^2 (\ai)
		& \text{if } r\ge4.
	\end{cases}
	$$
	\end{state}
Before the present work, unconditional positive-proportion nonvanishing results have only been established in the quadratic \cite{S00}, cubic \cite{DFDS24}, and quartic \cite{CFD26} cases. Additionally, non-vanishing results for infinite density-zero subsets within related families of Hecke $L$-functions have been obtained in \cite{BY10, BGL14, Gu, Luo}.
We note that the positive proportion obtained in \cref{Nonvanishing-sq-free-alt} for $r=3$ is smaller than the corresponding proportion established in \cite{DFDS24}. This discrepancy arises from the fact that we do not optimize our estimates over the mollifier parameter $Y$ by employing a truncated M\"obius sum as in \cite{S00, DFDS24}. Implementing such an optimization would considerably increase the technical complexity of the exposition without  introducing any significant new ideas.
\end{para}
\begin{para}
We also present analogues of the preceding theorems for the family of Hecke $L$-functions associated with $r$-th power-free ideals, as our method treats both families uniformly. To the best of our knowledge, such results are new even in the cubic and quartic cases. The resulting asymptotic formulas have sharper error terms than those for the square-free family and include secondary terms for small $r$.

\begin{state}\label{Main-Thm2-r-th-power-free-alt} Let the notation be as in \cref{Main-Thm1-alt}, and let $I(S)_{0}$
denote the set of nonzero $r$-th power-free ideals of $\mathscr{O}$ coprime to $S_f$. Then, for every $\varepsilon > 0$, we have
$$
\sum_{\ai \in I(S)_{0}} \big|L^S(\tfrac{1}{2}, \chi_{\ai})\big|^2 Q_\ai W\big(\tfrac{|\ai|}{X}\big)=X R_W^{(0)}(\log X) +X^{\frac{1}{2}+\frac{1}{r}} S_W^{(0)}(\log X)+O\Big(X^{\frac{3r - 2}{4r-2}+\varepsilon}\Big),
$$
for explicit linear polynomials $R_W^{(0)}, S_W^{(0)}$ and
$Q_\ai = \prod_{\pp^{n_\pp}\|\ai} n_\pp$. The leading coefficient of $R_W^{(0)}$ is
$$
\widehat{W}(1) \zeta_F^S(r/2)\Big[\res_{w=1} \zeta_F^S(w)\Big]^2 \prod_{v\notin S} \big(1-q_v^{-1}\big)
P_r^{(0)} \big(q_v^{-1/2}\big),
$$
where  $\displaystyle P_r^{(0)}(x)=(1-x^{r})\Big(\frac{1-x^{2r}}{1-x^2}-(r-1)x^{2r}\Big)+2x^r$.
\end{state}
Note that this is an asymptotic formula with a second-order term for $r = 3, 4$. Note also that for $r=3$,
if we write  $\ai=\bi\ci^2$ with $\bi,\ci$ square-free, then the factor $Q_{\ai}=d(\ci)$, the number of divisors of $\ci$.

Just as in the square-free setting, we obtain a companion asymptotic formula for the first moment restricted to $r$-th power-free ideals, which captures the second-order term for $r \le 6$.

\begin{state}\label{Thm5-first-moment-r-th-power-free-alt} Let the notation be as above. Then, for every $\varepsilon > 0$, we have
$$
\sum_{\ai \in I(S)_{0}} L^S(\tfrac{1}{2}, \chi_{\ai}) W\big(\tfrac{|\ai|}{X}\big)=C_{0}\widehat{W}(1)X +F_0\widehat{W}(\tfrac{1}{2} + \tfrac{1}{r})X^{\frac{1}{2}+\frac{1}{r}}+O\Big(X^{\frac{2r-1}{3r-1}+\varepsilon}\Big),
$$
where
$$
C_0=\zeta_F^S(r/2)\res_{w=1}\zeta_F^S(w)\prod_{v\notin S}\big(1-q_v^{-r}-q_v^{-(r+2)/2}+q_v^{-3r/2}\big).
$$
\end{state}
We establish twisted versions of the theorems above uniformly for both
the square-free and the $r$-th power-free families; see \cref{thm:twisted_2ndmom} and \cref{thm:twisted_first_moment}.

Finally, we establish a positive-proportion nonvanishing result for the Hecke $L$-functions associated with characters $\chi_{\ai}$ with $\ai \in I(S)_{0}$, where the family is ordered by the \emph{ideal norm} $|\ai| \le X$ rather than by the conductor of $\chi_{\ai}$.

\begin{state}\label{Nonvanishing-r-th-power-free-alt} Let the notation be as above. Then, for a positive proportion of ideals $\ai \in I(S)_{0}$, we have $L(\tfrac{1}{2}, \chi_{\ai}) \ne 0$. More precisely, for all sufficiently large $X$ and any fixed $\varepsilon > 0$, we have
	$$
	\sum_{\substack{\ai \in I(S)_{0} \\  |\ai| \le X\\ L(1/2, \chi_{\ai}) \ne 0}} 1
	\ge
	\begin{cases}
		\left(\alpha_r\dfrac{r-2}{2r^2-2r}-\varepsilon\right)
		\displaystyle\sum_{\substack{\ai \in I(S)_{0}\\ |\ai| \le X}}1
		& \text{if } r=3,4\\[1.2em]
		\left(\alpha_r\dfrac{r}{2r^2+4r-2}-\varepsilon\right)
		\displaystyle\sum_{\substack{\ai \in I(S)_{0}\\ |\ai| \le X}}1
		& \text{if } r>4,
	\end{cases}
	$$
	where $\alpha_r$ is an explicit Euler product given by
	\[\alpha_r= \frac{\zeta_F^S(r)}{\zeta_F^S(2)}\cdot \prod_{v\notin S} (1-q_v^{-3} u_v(r) )
	\]
	with $u_v(3)<0$ and $0<u_v(r)<1$ for $r\ge 4$. The precise value of $\alpha_r$ is given in \S\ref{par:non-vanish0}.
\end{state}
The proof of \cref{Nonvanishing-sq-free-alt} and \cref{Nonvanishing-r-th-power-free-alt} is given in \S\ref{sec:non-vanishing}.
\begin{rem}\label{rem:non-vanishing-alt}
Since the density of square-free ideals among $r$-th power-free ideals in $I(S)$
is $\zeta_F^S(r)/\zeta_F^S(2)$, \cref{Nonvanishing-sq-free-alt} already gives a positive proportion
of non-vanishing for $r$-th power-free ideals. For $F=\Q(\mu_{2r})$,
we show in \S\ref{sec:non-vanishing} that the product in
the definition of $\alpha_r$ is at least $1-\zeta(3)/r^2$ if $r\ge 4$ (clearly it is at least 1 if $r=3$), which shows that
the proportion of non-vanishing in \cref{Nonvanishing-r-th-power-free-alt} is greater than what one obtains trivially from \cref{Nonvanishing-sq-free-alt}.

\end{rem}
\end{para}

\subsection{Prior and related developments}
The fundamental problem of establishing asymptotic formulas for integral moments of $L$-functions on the critical line has a long history, going back to the 1916 work of Hardy and Littlewood. General conjectures for the leading-order asymptotics of the moments of $L$-functions on the critical line within classical families, based on random matrix calculations for the characteristic polynomials of matrices from the classical groups, were developed by   Keating and Snaith \cite{KeSn00b, KeSn00a}; see also Conrey and Farmer \cite{CF00}. These predictions were grounded in Katz and Sarnak's influential discovery \cite{KS99a}, stemming from their analysis in the function field setting, that the local distribution of zeros within families of $L$-functions should be governed by the scaling limits of eigenvalue distributions of classical compact groups attached to the families. The asymptotic predictions for the moments of $L$-functions were further refined  to include the full main terms in the asymptotics by the first author in joint work with Goldfeld and Hoffstein \cite{DGH03}, as well as by Conrey, Farmer, Keating, Rubinstein, and Snaith \cite{CFKRS}.

The symplectic family of real primitive Dirichlet characters has  drawn considerable attention over the past forty years. The full moment conjecture of Conrey, Farmer, Keating, Rubinstein, and Snaith \cite{CFKRS} predicts

\[
\sum_{d < X}L(\tfrac{1}{2}, \chi_d)^{k} \sim XQ_k(\log X) \qquad \text{(as $X \to \infty$)}
\]
for an explicit polynomial $Q_k$
 of degree $k(k+1)/2$, whose leading coefficient was specified by Keating and Snaith \cite{KeSn00b}. As with the zeta-function, this is known unconditionally only for small $k$: Jutila \cite{Jut81} (first and second moment), Soundararajan \cite{S00} (third moment), with improved error terms later obtained in \cite{DGH03,Young13}; Shen and Stucky \cite{Sh-St24}  obtained an asymptotic formula with four main terms for the fourth moment. The existence of extra terms for $k=3$, conjectured in \cite{DGH03}, was confirmed in \cite{Di-Wh21},  where an asymptotic formula with an additional secondary term was established for a smoothed version of the third moment.

While the moments of the family of quadratic Dirichlet $L$-functions have seen significant structural advancements, computing the moments for families of $L$-functions associated with higher-order residue symbols (cubic, quartic, and general $r$-th order characters) presents fundamentally different analytic difficulties.
Unlike in the quadratic case, the study of unitary families of higher-order characters is complicated by the presence of Gauss sums.
 Using the theory of multiple Dirichlet series, Friedberg, Hoffstein, and Lieman \cite{FHL03} computed the first moment \emph{with certain correction factors} for the family of Hecke $L$-functions  attached to  $r$-th order characters over a number field containing the $r$-th roots of unity. In \cite{Di04}, the first named author used the meromorphic continuation of the multiple Dirichlet series resulting from the Rankin--Selberg convolution of a metaplectic Eisenstein series on the $r$-fold cover of $\mathrm{GL}_2$ with itself to compute the second moment of the same family, \emph{with similar correction factors}. In \cite{BGL14}, Blomer, Goldmakher, and Louvel established a large sieve inequality for $r$-th order residue symbols and used it to obtain a zero-density result and a lower bound for the non-vanishing of the central values of the associated Hecke $L$-functions. Assuming the Generalized Riemann Hypothesis for the $L$-functions attached to the cubic residue symbols $\chi_a$ over the Eisenstein quadratic number field $\Q(\omega)$ ($\omega = e^{2 \pi i /3}$), David and G\"{u}lo\u{g}lu \cite{DG22} used the one-level density to show that $L(\tfrac{1}{2}, \chi_a) \neq 0$ for a positive proportion (at least $2/13$) of the characters $\chi_a$ with $a$ square-free and $a \equiv 1 \pmod 9$; see also \cite{GY} for the family of Hecke $L$-functions associated to primitive cubic characters over $\Q(\omega)$.

More recently, David, de Faveri, Dunn, and Stucky \cite{DFDS24} and Castillo, de Faveri, and Dunn \cite{CFD26} used the method of first and second mollified moments, as previously employed by Soundararajan \cite{S00} in the quadratic case, to obtain unconditional positive proportions of non-vanishing at the central point for families of cubic and quartic Hecke $L$-functions, respectively. Specifically, these works rely on the asymptotic evaluation of the second mollified moment with a power-saving error term for cubic residue symbols over the Eisenstein field $\Q(\omega)$ and quartic residue symbols over the Gaussian field $\Q(\sqrt{-1})$.

\subsection{Function field analogues}
An important feature of families of $L$-functions associated with any fixed higher-order residue symbol is that their moments admit perfect function field analogues. For instance, following the ``recipe'' developed in \cite{CFKRS}, Andrade and Keating \cite{AK14} obtained the corresponding conjectural asymptotics for all moments of the quadratic family; see also \cite{Ru-Wu15}.

Very recently, Bergstr\"om, Petersen, Westerland, and the first author \cite{BDPW23} computed the stable homology of the braid group with coefficients in any Schur functor applied to the integral reduced Burau representation. Combined with a recent homological stability theorem of Miller--Patzt--Petersen--Randal-Williams \cite{MPPRW24}, these homological results confirmed the Conrey--Farmer--Keating--Rubinstein--Snaith and Andrade--Keating predictions for all sufficiently large prime powers $q$. Prior to \cite{BDPW23} and \cite{MPPRW24}, asymptotics in the quadratic case were known only for the first four moments, largely due to the work of Florea \cite{Flo17a, Flo17b, Flo17c}; previously, Hoffstein and Rosen \cite{Ho-Ro92} obtained an asymptotic formula for the first moment, albeit with a weaker error term. By using one-level density, Bui and Florea \cite{BF} established that at least 94\% of quadratic Dirichlet $L$-functions over rational function fields do not vanish at the central point. A refined asymptotic formula for a weighted version of the fourth moment with infinitely many secondary terms was established in \cite{DPP23}, and an asymptotic formula for the third moment with a secondary term, analogous to that in \cite{Di-Wh21}, was proved in \cite{Di19}.

The study of asymptotics for the mean values of $L$-functions associated with families of higher-order residue symbols over rational function fields, as well as the non-vanishing of their central values, has garnered significant attention recently. Much of the existing research has concentrated solely on the first moment, one-level density,
and non-vanishing results; see \cite{DFL21, DFL22, DM23, DFL25, ELS}. Recently, Goel and Ray
\cite{GR} established an asymptotic formula for the second moment of cubic $L$-functions over the rational
function field $\F_{q}(T)$ when $q \equiv 2 \pmod 3$; see also \cite{HZ25} for a twisted version of the second moment. As far as the authors know, this is the only existing asymptotic formula for the second moment of $L$-functions associated with higher-order residue symbols in the function field setting. In \cite{DMPPW}, certain results of \cite{BDPW23} and \cite{MPPRW24} will be generalized to higher-degree cyclic extensions, establishing asymptotics for arbitrary moments of $L$-functions for any fixed higher-order residue symbol.

As alluded to above, all the results described admit direct function field counterparts, which can be established using the same framework developed in this paper. A significant advantage in the function field context is the availability of the Riemann Hypothesis, which can be exploited to quantitatively sharpen these results. For example, in light of the discussion in \cref{cubic-case-assum-GRH-alt}, the asymptotic formula for the second moment in \cref{Main-Thm1-alt} can be upgraded to capture the second-order term for cubic residue symbols over rational function fields $\F_{q}(T)$ with $q \equiv 1 \pmod 3$ (the Eisenstein case). Furthermore, the positive proportions in our non-vanishing theorems could be structurally improved in this setting by using level density estimates. Nevertheless, to keep the present manuscript at a reasonable length and to maintain our primary focus on the number field setting, we shall refrain from developing these function field extensions here.

\subsection{The overall strategy of the proof}\label{sec:strategy-alt}
To illustrate our method, we outline the proof of the twisted version of the asymptotic formula presented in \cref{Main-Thm1-alt}.

Let the notation be as above. Let $\psi_{1},  \psi_{2}$, and $\rho$ be Hecke characters unramified outside $S$. For $\mathbf{s} = (s_1, s_2, s_3) \in \C^3$ with real parts larger than $1$, consider
$$
Z_{1}^S(\mathbf{s}; \psi_{1},  \psi_{2}, \rho) =
\sum_{\ai \in I(S)}\frac{\mu^2 (\ai)L^{S}(s_1, \psi_{1}  \chi_{\ai})L^{S}(s_2, \overline{\psi}_{2}\overline{\chi}_{\ai})\rho(\ai)}{|\ai|^{s_3}}.
$$
We will deduce the theorem from the analytic properties of this function, namely its meromorphic continuation, the location of its poles, and the associated residues. For our purposes, it will suffice to restrict our attention to the region where $\Re(s_1), \Re(s_2) \ge \frac{1}{2}$. To accomplish this, we follow the strategy developed in \cite{Di19} and \cite{Di-Wh21}, expressing the function $Z_{1}^S$ in terms of Weyl group multiple Dirichlet series, whose analytic properties are much better understood.

To describe this in the present context, for an ideal $\ai \in I(S)$, let $\ai_0$ denote its $r$-th power-free part, and let $\ai_1 = \prod_{\mathfrak{p}_{v} \mid \ai_{0}} \mathfrak{p}_{v}$ denote the square-free kernel of $\ai_0$. We also set $S_\ai = \{v : \mathrm{ord}_{v}(\ai) > 0\}$. For a square-free ideal $\mathfrak{h} \in I(S)$, define
$$
Z^S(\mathbf{s}; \psi_{1},  \psi_{2}, \rho; \mathfrak{h})
\, = \sum_{\substack{\ai \in I(S)\\ \mathfrak{h} \mid  \ai \ai_{1}^{-1}}}
L^{S}(s_{1}, \psi_{1}  \chi_{\ai_{0}})L^{S}(s_{2}, \overline{\psi}_{2}\overline{\chi}_{\ai_{0}})
Q_\ai(s_1, s_2; \psi_{1}, \psi_{2})\rho(\ai)|\ai|^{- s_{3}},
$$
where $Q_\ai(s_1, s_2; \psi_{1}, \psi_{2})$ is the Dirichlet polynomial defined in subsection \ref{sec:wZ}. Since these polynomials equal $1$ whenever $\ai$ is square-free, it follows by M\"obius inversion that
$$
Z_{1}^S(\mathbf{s}; \psi_{1},  \psi_{2}, \rho) = \sum_{\mathfrak{h} \in I(S)}\mu(\mathfrak{h})Z^S(\mathbf{s}; \psi_{1},  \psi_{2}, \rho; \mathfrak{h}).
$$
This equality holds for $\Re(s_i) > 1$ ($i = 1,2,3$). Note that for the purpose of establishing this identity, the precise definition of the polynomials $Q_\ai$ on non-square-free ideals is immaterial.

For a fixed square-free ideal $\hi$, each element of the family $(Z^S(\sb; \psi_1,  \psi_2, \rho; \hi))_{\psi_1, \psi_2, \rho}$ can be expressed, as in \cite{Di19, Di-Wh21}, as a finite linear combination of multiple Dirichlet series of the form $
\wZ^{S \cup S_\hi}(\sb; \psi_1\chi_\ff, \psi_2\chi_\ff, \rho \chi_\gi)
$,
for coprime $r$-th power-free ideals $\ff, \gi \in I(S)$ whose prime factors divide $\hi$ (i.e., $S_\ff, S_\gi \subseteq S_\hi$). The precise explicit formula for this decomposition is provided in \cref{prop: sieving}. The family
$
\Eu{F} = (\wZ^{T}(\sb; \psi_1, \psi_2, \rho))_{\psi_1, \psi_2, \rho}
$,
parametrized by Hecke characters $\psi_1, \psi_2, \rho$ unramified outside a finite set of places $T$ containing $S$,
is a family of Weyl group multiple Dirichlet series constructed from twisted Kubota series
$$
D^T(s, \ai, \psi) = \sum_{\bi\in I(T)} \psi(\bi)G(\ai, \bi) |\bi|^{-s} \quad \text{($\ai \in I(T)$),}
$$
where $G(\ai, \bi)$ are certain (normalized) generalized $r$-th order Gauss sums (see \cite{Di04, FHL03, BGL14}). This construction was carried out in \cite{Di04} and is based on studying the Rankin--Selberg convolution series
\[
	Z_\aux^T({\sb}; \psi_1, \psi_2, \rho)  =\sum_{\ai\in I(T)}	\frac{D^{T}(s_1, \ai, \psi_1 \rho) \overline{D^{T}(\overline{s_2}, \ai, \psi_2 \overline{\rho})}\, \overline{\rho}(\ai)}{|\ai|^{s_3}}.
\]
It turns out that each element of $\Eu{F}$ satisfies a finite group of functional equations involving all elements of the family, and admits a meromorphic continuation to $\C^3$. One essential ingredient in this process is the meromorphic continuation to all of $\C$ and the functional equation of the function
$$
\widetilde{D}^{T}(s, \ai, \psi) = \zeta_{F}^{T}(rs - r/2  + 1)D^T(s, \ai, \psi)
\quad \text{($\psi^r = 1$),}
$$
established in \cite{BB06}; the function
$
\big(s - \tfrac{1}{2} - \tfrac{1}{r}\big)\big(s - \tfrac{1}{2} + \tfrac{1}{r}\big)\widetilde{D}^{T}(s, \ai, \psi)
$
is entire of order $1$, and hence it satisfies a convexity bound. Working out the explicit bound in the critical strip is somewhat subtle; it is achieved here via a recursive argument starting from the set of places $S$ in combination with the functional equation of the function $\widetilde{D}^{S}(s, \ai, \psi)$.

\begin{rem}\label{rem:Qi}
Our restriction to algebraic number fields containing the group $\mu_{2r}$ of $2r$-th roots of unity is made, for simplicity, to apply directly the results in \cite{BB06}. In principle, this restriction can be removed with additional technical work. For example, the case $r=4$ and $F=\Q(\sqrt{-1})$ has been explicitly worked out in \cite{Di04}, allowing twisting by general Hecke characters. In  \cite{CFD26}, Castillo, de Faveri, and Dunn also treated the case of Hecke $L$-functions associated to quartic twists of the congruent number curve $y^2=x^3-x$, and proved that in this family a positive proportion of the quartic twists have Mordell--Weil rank 0 over $\Q(\sqrt{-1})$.
\end{rem}

By applying a sequence of functional equations and convexity bounds at each step, we have the estimate (see \cref{prop:convexity})
\[\begin{aligned}
(s_3-1)^2 &(s_3-1/2-1/r)^2\wZ^{S\cup S_\hi}\big(\tfrac 12, \tfrac 12, s_3; \psi_1 \chi_{\ff}, \psi_2 \chi_{\ff}, \rho\chi_\gi \big)\\
&\ll_\eps |\hi|^{(r-1)(1-\sigma)+\eps}
(1+|s_3|)^{4+rd(1-\sigma)+\eps}\max_{\psi\in \widehat{R}_\ci}\sum_{\ti|\hi^{r-1}}
\sum_{\substack{\ai \in I(S)_{0}\\ (\ai,\ti)=1}}\frac{|L(1/2, \psi\chi_{\ai\ti})|^2}{|\ai|^{1+\eps}}
\end{aligned}
\]
for $\sigma=\Re (s_3)\in(0,1)$, where $d=[F:\Q]$. Here $R_{\ci} =H_{\ci} \otimes \Z/r\Z$, where $H_{\ci}$ is the ray class group modulo the ideal $\ci$ introduced above. Thus, each function $Z^S(\sb; \psi_1,  \psi_2, \rho; \hi)$ admits a meromorphic continuation to $\C^3$. Moreover, the corresponding estimate for these functions for $\sigma \in (0,1)$ can be easily obtained.

Using the analytic properties of the functions $Z^S(\sb; \psi_{1},  \psi_{2}, \rho; \hi)$ ($\hi \in I(S)$ square-free), we seek to determine the region of absolute convergence of the series
$$
\sum_{\mathfrak{h} \in I(S)}\mu(\mathfrak{h})Z^S\big(\tfrac 12, \tfrac 12, s_3; \psi_{1},  \psi_{2}, \rho; \mathfrak{h}\big)
$$
for $\sigma > \frac{1}{2}$, away from the poles at $s_3 = \frac{1}{2} + \frac{1}{r}, 1$. It is precisely at this point that we need the large sieve inequality for higher-order residue symbols \cite[Theorem 1.3]{BGL14} to obtain the estimate
$$
\sum_{\substack{\ai\in I(S)_0\\ (\ai,\ti)=1}}\frac{|L(1/2, \psi\chi_{\ai\ti})|^2}{|\ai|^{1+\eps}}\ll |\ti_1|^{\frac{1}{2}+\eps},
$$
where $\psi$ is a Hecke character unramified outside $S$ of order dividing $r$, and $\ti \in I(S)$ is $r$-th power-free (see \cref{prop:estimate at1}). This bound yields the continuation of $Z_{1}^S(\tfrac 12, \tfrac 12, s_3; \psi_{1},  \psi_{2}, \rho)$ to the half-plane $\sigma>\frac{2r+1}{2r+2}$ with one double pole at $s_3 = 1$. When $r = 3$, we use \cite[Proposition 4.2]{DFDS24} to show that the series
$$
\sum_{\substack{\ai,\ei \in I(S)_0\\ (\ai,\ei)=1}}|\ei|^{2-4\sigma}
\frac{|L(1/2, \psi\chi_{\ai\ei^2})|^2}{|\ai|^{1+\eps}}
$$
converges for $\sigma>5/6$. This pushes the continuation of $Z_{1}^S(\tfrac 12, \tfrac 12, s_3; \psi_{1}, \psi_{2}, \rho)$ to the half-plane $\sigma>5/6$.

With the analytic continuation established, the asymptotic formula for the twisted second moment (summed over square-free ideals $\ai \in I(S)$) is derived via a standard contour integration argument. The argument relies on the fact that, away from the double pole at $s_3=1$, the function $Z_{1}^S(\tfrac 12, \tfrac 12, s_3; \psi_{1}, \psi_{2}, \rho)$ exhibits polynomial growth in any vertical strip $\frac{2r+1}{2r+2} < \sigma < C$ (and $5/6 < \sigma < C$ when $r = 3$). The main term of the asymptotic formula is then dictated by the principal part of this function, which is extracted by explicitly computing the residues at the simple poles of the constituent Weyl group multiple Dirichlet series $\wZ^{S \cup S_\hi}(\sb; \psi_1\chi_\ff, \psi_2\chi_\ff, \rho \chi_\gi)$.

Finally, once the smoothed asymptotics for the first two twisted moments are established -- with the first moment being treated by an entirely analogous method -- the non-vanishing theorems can be deduced. Following the method in \cite{S00}, \cite{DFDS24}, and \cite{CFD26}, these non-vanishing results are obtained using a similarly constructed mollifier. Furthermore, to optimize the error terms in the case of $r$-th power-free ideals (\cref{Main-Thm2-r-th-power-free-alt} and \cref{Thm5-first-moment-r-th-power-free-alt}), we must introduce an additional technical refinement. Specifically, we implement a delicate inductive argument that allows us to sharpen the convexity bounds with respect to one of the underlying parameters, thereby yielding the better error thresholds asserted in our results.

\begin{rem}[Final remark]
Compared to previous approaches, the main advantage of our method lies in the study of the fundamental building blocks $\wZ^{S \cup S_\hi}(\sb; \psi_1\chi_\ff, \psi_2\chi_\ff, \rho \chi_\gi)$. By analyzing these constituent series, we capture deeper analytic information that is ``invisible'' at the level of the fully sieved object $Z_{1}^S(\mathbf{s}; \psi_{1},  \psi_{2}, \rho)$. This phenomenon is well illustrated by the study of smoothed asymptotics for the cubic moment of quadratic Dirichlet $L$-functions. For that problem, traditional methods yielded an asymptotic formula with an error term of $\ll_{\varepsilon} X^{3/4 + \varepsilon}$ \cite{Young13}. In contrast, the approach utilizing Weyl group multiple Dirichlet series extracted an asymptotic formula with a second-order main term and the sharper error term of $\ll_{\varepsilon} X^{2/3 + \varepsilon}$ \cite{Di-Wh21}. At present, this multiple Dirichlet series framework remains the only known method capable of detecting this second-order term.
\end{rem}

 \subsection{Acknowledgments} We thank Alexander Dunn for some helpful comments
 and for bringing several relevant references to our attention.
 A.D., V.P. and A.P. were supported by the project CF159/31.07.2023, ``Group schemes, root systems, and related representations'', funded by the European Union - NextGenerationEU through the PNRR call no. PNRR-III-C9-2023-I8. B.I. was partially supported by  the  Simons Foundation grant 420882.

\section{Notation and definitions}\label{sec:notation}
\subsection{Notation related to ideals}
Let $F$ be a number field containing the $2r$-th order roots of unity, and let $\OO$ be its ring of integers. Fix a set $S$ of places containing the archimedean places, the ramified places, the places dividing $r$, and such that the ring of $S_f$-integers $\OO_S$ has class number 1,
with $S_f$ the set of finite places in $S$.
For an integral ideal $\bi\subset \OO$, we denote by $S_\bi$ the set of finite places $v$ with $\ord_v(\bi)>0$.
We denote by $I(S)$ the set of integral ideals $\bi$ such that $S_\bi\cap S=\emptyset$.

For an ideal $\mathfrak{a} \in I(S)$, let $\ai_0$ denote the $r$-th power-free part
  of $\ai$ and let $\ai_1$ denote the square-free kernel of $\ai_0$, namely $\ai_1=\prod_{\pp_v|\ai_0} \pp_v$.
  We denote by $I(S)_0$ and $I(S)_1$, respectively, the subsets of $I(S)$ consisting of
  $r$-th power-free and square-free ideals.

 For an integral ideal $\ai$, we denote by $\rad \ai$ the product of the prime ideals dividing it.
 If $\ai$ is $(r+1)$-th power-free, we set $\ov{\ai}=(\rad\ai)^r/\ai$.

 For $\ai,\bi$ integral ideals, we denote by $\ai^{(\bi)}=\prod_{v\notin S_\bi} \pp_v^{\ord_v(\ai)}$ the part of $\ai$ coprime to $\bi$.

\subsection{Character data}\label{sec:char_data}
The definition of the characters $\chi_\mi$ for ideals $\mi\in I(S)$ depends on some choices that we fix throughout,
following \cite[p.318]{FHL03}.

Let $\ci=\prod_{v\in S_f} \pp_v^{n_v}$, where $n_v=1$ if $\pp_v\nmid r$, and
for $\pp_v\mid r$ the exponents $n_v$ are large enough such that,
if $a\in F_v$ has $\ord_v(a-1)\ge n_v$, then $a\in (F_v^\times)^{r}$.

Let $R_{\ci} =H_{\ci} \otimes \Z/r\Z$, where $H_{\ci}$ is the ray class group modulo the ideal $\ci$ fixed above.
Let $\EE\subset I(S)$ be a full set of representatives for $R_\ci$, and for $E\in \EE$ choose
 $m_{\sss E}\in \OO$ with $E\OO_S=m_{\sss E} \OO_S$. We can choose the elements $E$ and $m_{\sss E}$ as follows \cite[p.318]{FHL03}: writing $R_\ci$ as a direct product of cyclic groups, fix a set of representatives $\EE_0\subset I(S)$ for the generators of these cyclic groups, and for $E_0\in \EE_0$ let $m_{E_0}$ be a generator of $E_0$ as an $\OO_S$-ideal. Then the set $\EE$ consists
 of a complete set of representatives for $R_\ci$ of the form $E=\prod_{E_0\in \EE_0} E_0^{n_{E_0}}$ with $n_{E_0}\ge 0$,
 and for such an $E$ we choose $m_E=\prod_{E_0\in \EE_0} m_{E_0}^{n_{E_0}}$ as a generator of $E\OO_S$.

For any fractional ideal $\ai$ coprime to $\ci$, we denote by $[\ai]\in \EE$ the unique representative in the same ideal class.

	\section{Preliminaries on Gauss sums and the \texorpdfstring{$r$}~-th order residue symbol}

\begin{para}
We recall the construction of the $r$-th order Hecke characters $\chi_\mi$, for ideals $\mi\in I(S)$,
following Friedberg, Hoffstein, and Lieman \cite{FHL03}--see also the more recent
treatments~\cite{GL13, BGL14}.
The main result of \S\ref{sec:Gauss sums} establishes the connection between the
Gauss sum $G(\chi_{\mi \ni})$ and the product $G(\chi_{\mi}) G(\chi_\ni)$ for $\mi$, $\ni\in I(S)$ coprime. We assume $r>2$,
as in this case the Hecke characters $\chi_\mi$ have trivial infinity type.
\end{para}

\subsection{The \texorpdfstring{$r$}~-th order characters}\label{sec: res symbols}
\begin{para}
For any $a\in F^\times$ and any ideal $\bi\in I(S)$ coprime to $a$, the $r$-th power residue symbol $\left(\frac{a}{\bi}\right)$ is an $r$-th root of unity defined in terms of the Frobenius substitution at the primes dividing $\bi$ in the extension
$F(\sqrt[r]{a})\slash F$--see e.g. \cite[Ch. VI \S 8]{N} or \cite[Exercises]{CF} for its definition and properties.
If $\bi=(b)$ is principal, we also write  $\big(\frac ab\big) =\big(\frac a\bi\big)$.
\end{para}

\begin{para}
	We now define the $r$-th order character $\chi_\bi$ for an integral ideal $\mathfrak{b}\in I(S)$, following~\cite{FHL03}.
The definition depends on the choices fixed in \S\ref{sec:char_data}.

	For $\ai\in I(S)$ coprime to $\bi$, we write
	$
	\mathfrak{b} = (b)E\mathfrak{g}^r
	$
	with $E \in \mathscr{E}$, $b\in F^{\times}, b \equiv 1 \pmod{\ci}$, and
	$\mathfrak{g}\in I(S)$ such that $(\mathfrak{a}, \mathfrak{g}) = 1$.
	Since $\OO_S$ has class number one, there is $x\in\OO$ such that
	$\mathfrak{a} = (x)$ as $\mathscr{O}_{S}$-ideals, and using the $r$-th power residue symbol
	we define
	\be\label{eq: char def}
	\chi_{\mathfrak{b}}(\mathfrak{a}) =\psym{b m_{\sss E}}{\ai}=\psym{b m_{\sss E}}{x}.
	\ee
	 The reciprocity law gives
	\be\label{eq: rec}
	\chi_{\mathfrak{b}}(\mathfrak{a})= (x, m_{\scalebox{1.}{$\scriptscriptstyle E$}})_{\sss S}
	\Big(\frac{x}{bm_{\scalebox{1.}{$\scriptscriptstyle E$}}}\Big)
	\ee
	where $(\cdot, \cdot)_{\sss S}=\prod_{v\in S_f} (\cdot, \cdot)_v$ is the product of local Hilbert symbols.
	By the choice of $\ci$, we have
	 $(y, m_{\sss E})_{\sss S} = 1$, \!if $y\in F^{\times}$ and $y \equiv 1 \pmod{\mathfrak{c}}$. It follows that
	 \[
	 \psi_{\sss E}(\cdot)=(\cdot, m_{\scalebox{1.}{$\scriptscriptstyle E$}})_{\sss S}
	\]
	 is periodic on $ F(\mathfrak{c}) : = \{x \in  F^{\times} : ((x), \mathfrak{c}) = 1\}$
	of period $\mathfrak{c}$. Let $\mathfrak{c} \subseteq \mathfrak{c}_{\scalebox{1.}{$\scriptscriptstyle E$}}$ be the largest integral ideal with this property.

	We set $\chi_\bi(\ai)=0$ if $(\ai,\bi)\ne 1$.
\end{para}
\begin{para}
	The character $\chi_\bi$ is a Hecke character of trivial infinity type, as $r >2$~\cite{GL13,BGL14}.
	Its conductor is $\bi_1\ci_{\sss E}$, where $\bi_1$ denotes the product of the primes
	dividing the $r$-th power-free part of $\bi$, and we can also rewrite the reciprocity law as follows
	\[
	\chi_\bi(\ai)=\widetilde{\chi}_\bi(x) \psi_{\sss E}(x),
	\]
	where $\widetilde{\chi}_\bi(\cdot):=\big(\frac\cdot{bm_{\sss E}}\big)$ is a \emph{Dirichlet} character of conductor $\bi_1$.
\end{para}
\begin{para}
    We need the following explicit reciprocity law.
 \begin{lem}\label{lem: reciprocity}
     Let $\mathfrak{a}, \mathfrak{b} \in I(S)$ be coprime and $r$-th power-free ideals.
     Let $E = [\mathfrak{b}]$ and $G = [\mathfrak{a}]$. Then
     \be\label{eq: reciprocity}
	\chi_{\mathfrak{b}}(\mathfrak{a}) = \chi_\ai(\bi) (m_{\sss G}, m_{\sss E})_{\sss S}.
	\ee
 \end{lem}
\begin{proof}
Write $\mathfrak{b} = (b)(m_{\sss E})\mathfrak{g}^r$, with $ m_{\sss E}\in \OO$, $b\in F^{\times}, b \equiv 1 \pmod{\mathfrak{c}}$.

We can assume $(\mathfrak{g}, \mathfrak{a}) = 1$. To show this, let $\mathfrak{p} \mid (\mathfrak{g}, \mathfrak{a})$, and let
$\mathfrak{m}$ be an ideal of $\OO$, with $(\mathfrak{m}, \mathfrak{c}  \mathfrak{p}) = 1$, such that $\mathfrak{p} \mathfrak{m} = (\gamma_0)$, $\gamma_0 \in \OO$.
By the Chinese remainder theorem, we can find an integral element $\gamma$ such that
$\gamma \equiv 1 \pmod {\mathfrak{c}}$ and $\gamma \equiv \gamma_0 \pmod{\mathfrak{p}^2}$; indeed, since $(\mathfrak{c}, \mathfrak{p}^2) = 1$, we can
write $1 = a_1 + a_2$ with $a_1 \in \mathfrak{c}$ and  $a_2 \in \mathfrak{p}^2$. Then take $\gamma = \gamma_0 a_1 + a_2$.
This yields a different representation $\mathfrak{b} = (b')(m_{\sss E})\mathfrak{g}'^r$, where
we take $b' = b \gamma^{r v_\mathfrak{p}(\mathfrak{g})}$ and $\mathfrak{g}' = \mathfrak{g} \gamma^{-v_\mathfrak{p}(\mathfrak{g})}$. Note that $\mathfrak{p} \nmid \mathfrak{g}'$.
In this way, we eliminate all primes dividing both $\mathfrak{a}$ and $\mathfrak{g}$, so we can indeed assume $(\mathfrak{g}, \mathfrak{a}) = 1$.

Then $\mathrm{gcd}(b m_{\sss E}, \mathfrak{a}) = 1$. Let $\mathfrak{a} = (x)$ as $\OO_S$-ideals. Thus
$$
\left(\frac{\mathfrak{b}}{\mathfrak{a}}\right) = \left(\frac{b m_{\sss E}}{x}\right) = (x, b m_{\sss E})_{\sss S}\left(\frac{x}{b m_{\sss E}}\right) = (x, b m_{\sss E})_{\sss S}\left(\frac{x}{\mathfrak{b}}\right)\!,
$$
and the lemma follows.
\end{proof}
\end{para}

	\subsection{Gauss sums}
\begin{para}
	For $\chi$ a primitive Hecke character of conductor $\mi$ with trivial infinity type,
	its (normalized) associated Gauss sum is defined as in~\cite[Sec.12.2]{I}
	\[
	G(\chi)=\chi(\ri)^{-1}|\mi|^{-1/2} \sum_{x\in \ri\slash \ri \mi} \chi( x ) e(\tr(x/y)),
	\]
	where $\ri$ is an integral ideal coprime to $\mi$ such that $\ri\di\mi=(y)$ is principal, for some $y\in \OO$. We
	denote by $\di$ the different of $F$, namely $\di^{-1}=\{x\in F\mid \tr (x) \in \ZZ \}$,
	and we denote by $\chi(x)$ the value of $\chi$ on the principal ideal $(x)$, for $x\in F$.
	The Gauss sum does not depend on the choices of $\ri$ and $y$.
\end{para}
\begin{para}
	Since
		$
		\mathfrak{r} + \mathfrak{m} = \mathscr{O}
		$,
		we have
		$
		\mathfrak{r}/\mathfrak{r}\mathfrak{m} \cong \mathscr{O}/\mathfrak{m}
		$,
		and hence we can replace the sum with a sum over $\{z t\}$, where $t \in \mathfrak{r}, ((t), \mathfrak{m}) = 1$,
		and $\{z\}$ is any system of representatives of
		$\mathscr{O}/\mathfrak{m}$:
		\be\label{eq: Gauss sum Hecke char}
		G(\chi) = \chi( t )
		\overline{\chi}(\mathfrak{r})
		|\mathfrak{m}|^{-1/2}\sum_{z \in
			\mathscr{O}/\mathfrak{m}} \chi( z ) e(\mathrm{Tr}(zt/y)).
		\ee

	The Hecke characters constructed from the $r$-th power residue symbol are trivial on the $S$-units $\OO_S^\times$, and have a natural decomposition as a product of two characters of $(\OO/\mi)^\times$. Therefore it is convenient to extend the definition
	of Gauss sums to Dirichlet characters as follows.

	\begin{defn} For a primitive character $\chi$ of $(\mathscr{O}/\mathfrak{m})^{\times}$, let
	$\Delta = (\mathfrak{r}, x, y, t)$ be the following dataset: $\mathfrak{r}$ is an integral ideal,
	prime to $\mathfrak{c}\mathfrak{m}$, such that
	$\ri\di\mi=(y)$ is principal ($y \in \mathscr{O}$), $t \in \mathfrak{r}$ with $((t), \mathfrak{m}) = 1$, and $\mathfrak{r} = (x)$ as $\mathscr{O}_{S}$-ideals. The associated Gauss sum is:
		\[
		W\!\left(\chi;\frac{t}{y}, \frac{t}{x}\right) : = \chi(t/x) \Eu{N}\mathfrak{m}^{-1/2}
		\sum_{\substack{z \bmod \mathfrak{m}\\ (z, \mathfrak{m}) = 1}}
		\chi(z) e(\mathrm{Tr}(z t/y)).
		\]
	\end{defn}

	Note that $\frac{t}{y} \in \mathfrak{m}^{-1}\mathfrak{d}^{-1}$, so this is the same as the definition in \cite[Def. (6.3)]{N}
	without the factor $\chi(t/x)$. If $\chi$ is a Hecke character of trivial infinity type which is trivial on $\OO_S^\times$, then $\chi(\ri)=\chi(x)$ and
	we have $W(\chi; t/y,t/x)=G(\chi)$ independent of the choice of dataset $\Delta$.
	\end{para}
\begin{para}
	In general, the dependence of the Gauss sum on the
	choice of dataset $\Delta$ is determined in the next lemma.

	\begin{lem}\label{lem-GS-inv}
		For a primitive character $\chi$ of $(\mathscr{O}/\mathfrak{m})^{\times}$ and two datasets $\Delta = (\mathfrak{r}, x, y, t)$ and $\Delta' = (\mathfrak{r}', x', y', t')$, we have
		\[
		W\!\left(\chi;\frac{t}{y}, \frac{t}{x}\right)
		= \chi(\varepsilon)W\!\left(\chi; \frac{t'}{y'}, \frac{t'}{x'}\right)
		\]
		where $\varepsilon = \frac{x'y}{xy'} \in \mathscr{O}_{S}^{\times}$.
	\end{lem}

	\begin{proof} Let $\mathfrak{s}$ be an integral ideal, prime to $\mathfrak{m}$,
		such that $\mathfrak{s}\mathfrak{r} = (a)$ is a principal ideal. Then
		\[
		\frac{(a)}{\mathfrak{s}\mathfrak{r}'} = \frac{\mathfrak{r}}{\mathfrak{r}'} =
		\frac{\mathfrak{r}
			\mathfrak{d}
			\mathfrak{m}}{\mathfrak{r}'
			\mathfrak{d}
			\mathfrak{m}} = \frac{(y)}{(y')}
		\]
		from which it follows that $\mathfrak{s}\mathfrak{r}' = (a')$ is an integral principal ideal. Thus, one can find elements $\delta, \delta' \in \mathscr{O}$ such that
		$(\delta\delta', \mathfrak{m}) = 1$ and
		\[
		\frac{\delta t}{y} = \frac{\delta' t'}{y'}.
		\]
		(Upon changing one of the generators of the principal ideals $(a)$ or $(a')$,
		we can assume that $\delta = t' a$ and $\delta' = t a'$.) It follows that:
		\[
		\overline{\chi(\delta)}W\!\left(\chi;\frac{t}{y}, \frac{t}{x}\right)
		= W\!\left(\chi;\frac{\delta t}{y}, \frac{t}{x}\right)
		= W\!\left(\chi; \frac{\delta' t'}{y'}, \frac{t}{x}\right)
		= \overline{\chi(\delta')}
		W\!\left(\chi; \frac{t'}{y'}, \frac{t}{x}\right),
		\]
		and thus we obtain
		\[
		W\!\left(\chi;\frac{t}{y}, \frac{t}{x}\right)
		= W\!\left(\chi; \frac{t'}{y'}, \frac{t' a}{a' x}\right).
		\]
		To finish the proof, we just note that, by viewing $(a)\mathfrak{r}' = (a')\mathfrak{r}$ as an equality of $\mathscr{O}_{S}$-ideals, we have
		$
		a x' = \varepsilon a'x
		$
		for some $\varepsilon \in \mathscr{O}_{S}^{\times}$.
	\end{proof}
\end{para}
\begin{para}
	Now, let $\chi_{\scalebox{1.1}{$\scriptscriptstyle 1$}}$ and $\chi_{\scalebox{1.1}{$\scriptscriptstyle 2$}}$ be two primitive characters modulo
	$\mathfrak{m}_{\scalebox{1.1}{$\scriptscriptstyle 1$}}$ and $\mathfrak{m}_{\scalebox{1.1}{$\scriptscriptstyle 2$}}$, respectively, with $(\mathfrak{m}_{\scalebox{1.1}{$\scriptscriptstyle 1$}}, \mathfrak{m}_{\scalebox{1.1}{$\scriptscriptstyle 2$}}) = 1$. Suppose that $\chi = \chi_{\scalebox{1.1}{$\scriptscriptstyle 1$}}\chi_{\scalebox{1.1}{$\scriptscriptstyle 2$}}$ is a Hecke character of trivial infinity type of conductor $\mathfrak{m} = \mathfrak{m}_{\scalebox{1.1}{$\scriptscriptstyle 1$}}\mathfrak{m}_{\scalebox{1.1}{$\scriptscriptstyle 2$}}$. We further assume that $\chi$ is trivial on $\mathscr{O}_{S}^{\times}$.

	As an immediate consequence of \cref{lem-GS-inv}, we have

	\begin{cor}\label{lem-Hecke-GS-inv}
		Let the notation be as above. Then
		\[
		\chi_{\scalebox{1.1}{$\scriptscriptstyle 2$}}\Big(\frac{y}{x}\Big)
		W\!\left(\chi_{\scalebox{1.1}{$\scriptscriptstyle 1$}};\frac{t}{y}, \frac{t}{x}\right)
		\]
		is independent of the choice of dataset $\Delta = (\mathfrak{r}, x, y, t)$.
	\end{cor}

	\begin{proof} By \cref{lem-GS-inv}, the discrepancy between the Gauss sums associated with the pairs $(\chi_{\scalebox{1.1}{$\scriptscriptstyle 1$}}, \Delta)$ and $(\chi_{\scalebox{1.1}{$\scriptscriptstyle 1$}}, \Delta')$ is given by the factor $\chi(\varepsilon)$, where $\varepsilon = \frac{x'y}{xy'} \in \mathscr{O}_{S}^{\times}$. Since $\chi$ is trivial on $\mathscr{O}_{S}^{\times}$, we have
		$
		\chi_{\scalebox{1.1}{$\scriptscriptstyle 1$}}(\varepsilon) = \chi_{\scalebox{1.1}{$\scriptscriptstyle 2$}}(\varepsilon^{-1})=
		\chi_{\scalebox{1.1}{$\scriptscriptstyle 2$}}(\frac{x}{y})	\chi_{\scalebox{1.1}{$\scriptscriptstyle 2$}}(\frac{y'}{x'})
		$,
		and our assertion follows.
	\end{proof}
\end{para}
\begin{para} The following lemma spells out the key multiplicativity property of Gauss sums that will be used throughout.
\begin{lem}\label{lem-Hecke} Let
	$
	\chi = \chi_{\scalebox{1.1}{$\scriptscriptstyle 1$}}\chi_{\scalebox{1.1}{$\scriptscriptstyle 2$}}
	$
	be as above, and let
	$
	\Delta_{\scalebox{1.1}{$\scriptscriptstyle 1$}} = (\mathfrak{r}_{\scalebox{1.1}{$\scriptscriptstyle 1$}}, x_{\scalebox{1.1}{$\scriptscriptstyle 1$}}, y_{\scalebox{1.1}{$\scriptscriptstyle 1$}}, t_{\scalebox{1.1}{$\scriptscriptstyle 1$}})
	$
	and
	$
	\Delta_{\scalebox{1.1}{$\scriptscriptstyle 2$}} = (\mathfrak{r}_{\scalebox{1.1}{$\scriptscriptstyle 2$}}, x_{\scalebox{1.1}{$\scriptscriptstyle 2$}}, y_{\scalebox{1.1}{$\scriptscriptstyle 2$}}, t_{\scalebox{1.1}{$\scriptscriptstyle 2$}})
	$
	be two datasets for the characters $\chi_{\scalebox{1.1}{$\scriptscriptstyle 1$}}$ and $\chi_{\scalebox{1.1}{$\scriptscriptstyle 2$}}$, respectively. Then the normalized Gauss sum
		$G(\chi)$ can be decomposed as
		\[
		G(\chi) =
		\chi_{\scalebox{1.1}{$\scriptscriptstyle 1$}}\Big(\frac{y_{\scalebox{1.1}{$\scriptscriptstyle 2$}}}{x_{\scalebox{1.1}{$\scriptscriptstyle 2$}}}\Big)
		\chi_{\scalebox{1.1}{$\scriptscriptstyle 2$}}\Big(\frac{y_{\scalebox{1.1}{$\scriptscriptstyle 1$}}}{x_{\scalebox{1.1}{$\scriptscriptstyle 1$}}}\Big)
		W\!\left(\chi_{\scalebox{1.1}{$\scriptscriptstyle 1$}};\frac{t_{\scalebox{1.1}{$\scriptscriptstyle 1$}}}{y_{\scalebox{1.1}{$\scriptscriptstyle 1$}}}, \frac{t_{\scalebox{1.1}{$\scriptscriptstyle 1$}}}{x_{\scalebox{1.1}{$\scriptscriptstyle 1$}}}\right)
		\!W\!\left(\chi_{\scalebox{1.1}{$\scriptscriptstyle 2$}};\frac{t_{\scalebox{1.1}{$\scriptscriptstyle 2$}}}{y_{\scalebox{1.1}{$\scriptscriptstyle 2$}}}, \frac{t_{\scalebox{1.1}{$\scriptscriptstyle 2$}}}{x_{\scalebox{1.1}{$\scriptscriptstyle 2$}}}\right)\!.
		\]
	\end{lem}

	\begin{proof}
		By \cref{lem-Hecke-GS-inv}, it suffices to establish the decomposition of $G(\chi)$ for two specific datasets, which we achieve via a standard application of the Chinese Remainder Theorem.
		Let $\mathfrak{r}$ be an integral ideal, prime to
		$\mathfrak{m}$, such that $\mathfrak{r}\mathfrak{d}\mathfrak{m} = (y)$ is principal ($y \in \mathscr{O}$),
		and let $t \in \mathfrak{r}, ((t), \mathfrak{m}) = 1$ as in \eqref{eq: Gauss sum Hecke char}.

		Now choose integral ideals $\mathfrak{r}_{\scalebox{1.1}{$\scriptscriptstyle \mathrm{diff}$}}, \mathfrak{r}_{\scalebox{1.1}{$\scriptscriptstyle 1$}}, \mathfrak{r}_{\scalebox{1.1}{$\scriptscriptstyle 2$}}$ that are prime to $\mathfrak{c}\mathfrak{m}$, such that:
		$
		\mathfrak{r}_{\scalebox{1.1}{$\scriptscriptstyle \mathrm{diff}$}}
		\mathfrak{d} = (d)
		$,
		$\mathfrak{r}_{\scalebox{1.1}{$\scriptscriptstyle 1 $}}\mathfrak{m}_{\scalebox{1.1}{$\scriptscriptstyle 1$}} = (y_{\scalebox{1.1}{$\scriptscriptstyle 1$}})
		$,
		$
		\mathfrak{r}_{\scalebox{1.1}{$\scriptscriptstyle 2$}}
		\mathfrak{m}_{\scalebox{1.1}{$\scriptscriptstyle 2$}} = (y_{\scalebox{1.1}{$\scriptscriptstyle 2$}})
		$
		are principal ideals with $(y_{\scalebox{1.1}{$\scriptscriptstyle 1$}}, y_{\scalebox{1.1}{$\scriptscriptstyle 2$}}) = 1$.
		\!For example, take
		$
		(\mathfrak{r}_{\scalebox{1.1}{$\scriptscriptstyle 1$}}, \mathfrak{r}_{\scalebox{1.1}{$\scriptscriptstyle 2$}}) = 1
		$.
		\!Let
		$
		\mathfrak{r}_{\scalebox{1.1}{$\scriptscriptstyle \mathrm{diff}$}}
		$
		$
		= (x_{\scalebox{1.1}{$\scriptscriptstyle \mathrm{diff}$}}), \mathfrak{r}_{\scalebox{1.1}{$\scriptscriptstyle 1$}} = (x_{\scalebox{1.1}{$\scriptscriptstyle 1$}}), \mathfrak{r}_{\scalebox{1.1}{$\scriptscriptstyle 2$}} = (x_{\scalebox{1.1}{$\scriptscriptstyle 2$}})
		$
		as $\mathscr{O}_{S}$-ideals, and set
		$
		\mathfrak{r} =
		\mathfrak{r}_{\scalebox{1.1}{$\scriptscriptstyle 1$}}
		\mathfrak{r}_{\scalebox{1.1}{$\scriptscriptstyle 2$}}
		\mathfrak{r}_{\scalebox{1.1}{$\scriptscriptstyle \mathrm{diff}$}}
		= (x_{\scalebox{1.1}{$\scriptscriptstyle 1$}}
		x_{\scalebox{1.1}{$\scriptscriptstyle 2$}}
		x_{\scalebox{1.1}{$\scriptscriptstyle \mathrm{diff}$}})
		$.
		Then, by expressing the parameter
		$
		z = z_{\scalebox{1.1}{$\scriptscriptstyle 2$}}
		y_{\scalebox{1.1}{$\scriptscriptstyle 1$}} +
		z_{\scalebox{1.1}{$\scriptscriptstyle 1$}}
		y_{\scalebox{1.1}{$\scriptscriptstyle 2$}}
		$ in \eqref{eq: Gauss sum Hecke char},
		with $z_{\scalebox{1.1}{$\scriptscriptstyle 1$}} \bmod \mathfrak{m}_{\scalebox{1.1}{$\scriptscriptstyle 1$}}$ and
		$z_{\scalebox{1.1}{$\scriptscriptstyle 2$}} \bmod \mathfrak{m}_{\scalebox{1.1}{$\scriptscriptstyle 2$}}$, we can decompose $G(\chi)$ as
		\begin{align*}
			G(\chi) & =
			\chi_{\scalebox{1.1}{$\scriptscriptstyle 1$}}(y_{\scalebox{1.1}{$\scriptscriptstyle 2$}})
			\chi_{\scalebox{1.1}{$\scriptscriptstyle 2$}}(y_{\scalebox{1.1}{$\scriptscriptstyle 1$}})
			\overline{\chi}(\mathfrak{r})\\
			&\cdot \Eu{N}(\mathfrak{m}_{\scalebox{1.1}{$\scriptscriptstyle 1$}})^{-1/2}\chi_{\scalebox{1.1}{$\scriptscriptstyle 1$}}(t) \sum_{\substack{z_{\scalebox{.85}{$\scriptscriptstyle 1$}} \bmod \mathfrak{m}_{\scalebox{.85}{$\scriptscriptstyle 1$}}\\ (z_{\scalebox{.85}{$\scriptscriptstyle 1$}}, \mathfrak{m}_{\scalebox{.85}{$\scriptscriptstyle 1$}}) = 1}}
			\chi_{\scalebox{1.1}{$\scriptscriptstyle 1$}}(z_{\scalebox{1.1}{$\scriptscriptstyle 1$}})
			e\Big(\mathrm{Tr}\Big(\frac{z_{\scalebox{1.1}{$\scriptscriptstyle 1$}}t}
			{y_{\scalebox{1.1}{$\scriptscriptstyle 1$}}d}\Big)\Big)\\
			&\cdot \Eu{N}(\mathfrak{m}_{\scalebox{1.1}{$\scriptscriptstyle 2$}})^{-1/2}
			\chi_{\scalebox{1.1}{$\scriptscriptstyle 2$}}(t)
			\sum_{\substack{z_{\scalebox{.85}{$\scriptscriptstyle 2$}} \bmod \mathfrak{m}_{\scalebox{.85}{$\scriptscriptstyle 2$}}\\ (z_{\scalebox{.85}{$\scriptscriptstyle 2$}}, \mathfrak{m}_{\scalebox{.85}{$\scriptscriptstyle 2$}}) = 1}}\chi_{\scalebox{1.1}{$\scriptscriptstyle 2$}}(z_{\scalebox{1.1}{$\scriptscriptstyle 2$}})
			e\Big(\mathrm{Tr}\Big(
			\frac{z_{\scalebox{1.1}{$\scriptscriptstyle 2$}}t}
			{y_{\scalebox{1.1}{$\scriptscriptstyle 2$}}d}\Big)\Big)\\
			&= \chi_{\scalebox{1.1}{$\scriptscriptstyle 1$}}
			\Big(\frac{y_{\scalebox{1.1}{$\scriptscriptstyle 2$}}}{x_{\scalebox{1.1}{$\scriptscriptstyle 2$}}}\Big)
			\chi_{\scalebox{1.1}{$\scriptscriptstyle 2$}}\Big(\frac{y_{\scalebox{1.1}{$\scriptscriptstyle 1$}}}{x_{\scalebox{1.1}{$\scriptscriptstyle 1$}}}\Big)
			W\!\left(\chi_{\scalebox{1.1}{$\scriptscriptstyle 1$}}; \frac{t}{y_{\scalebox{1.1}{$\scriptscriptstyle 1$}}d}, \frac{t}{x_{\scalebox{1.1}{$\scriptscriptstyle 1$}}
				x_{\scalebox{1.1}{$\scriptscriptstyle \mathrm{diff}$}}}\right)
			\!W\!\left(\chi_{\scalebox{1.1}{$\scriptscriptstyle 2$}}; \frac{t}{y_{\scalebox{1.1}{$\scriptscriptstyle 2$}}d}, \frac{t}{x_{\scalebox{1.1}{$\scriptscriptstyle 2$}}
				x_{\scalebox{1.1}{$\scriptscriptstyle \mathrm{diff}$}}}\right)\!.
		\end{align*}
	Since $d/x_{\scalebox{1.1}{$\scriptscriptstyle \mathrm{diff}$}}$ is an $S$-unit, we have
	\[
	\chi_{\scalebox{1.1}{$\scriptscriptstyle 1$}}
	\Big(\frac{y_{\scalebox{1.1}{$\scriptscriptstyle 2$}}}{x_{\scalebox{1.1}{$\scriptscriptstyle 2$}}}\Big)
	\chi_{\scalebox{1.1}{$\scriptscriptstyle 2$}}\Big(\frac{y_{\scalebox{1.1}{$\scriptscriptstyle 1$}}}{x_{\scalebox{1.1}{$\scriptscriptstyle 1$}}}\Big)
	= \chi_{\scalebox{1.1}{$\scriptscriptstyle 1$}}
	\Big(\frac{y_{\scalebox{1.1}{$\scriptscriptstyle 2$}}d}{x_{\scalebox{1.1}{$\scriptscriptstyle 2$}}
		x_{\scalebox{1.1}{$\scriptscriptstyle \mathrm{diff}$}}}\Big)
	\chi_{\scalebox{1.1}{$\scriptscriptstyle 2$}}\Big(\frac{y_{\scalebox{1.1}{$\scriptscriptstyle 1$}}d}{x_{\scalebox{1.1}{$\scriptscriptstyle 1$}}
		x_{\scalebox{1.1}{$\scriptscriptstyle \mathrm{diff}$}}}\Big)
	\]
	and thus, this gives the decomposition of $G(\chi)$ with respect to the datasets
	\[
	\text{
		$\Delta_{\mathfrak{m}_{\scalebox{.9}{$\scriptscriptstyle 1$}}} \! = (\mathfrak{r}_{\scalebox{1.1}{$\scriptscriptstyle 1$}}
		\mathfrak{r}_{\scalebox{1.1}{$\scriptscriptstyle \mathrm{diff}$}}, x_{\scalebox{1.1}{$\scriptscriptstyle 1$}}
			x_{\scalebox{1.1}{$\scriptscriptstyle \mathrm{diff}$}}, y_{\scalebox{1.1}{$\scriptscriptstyle 1$}}d, t)$\,
	and \,
	$\Delta_{\mathfrak{m}_{\scalebox{.9}{$\scriptscriptstyle 2$}}} = (\mathfrak{r}_{\scalebox{1.1}{$\scriptscriptstyle 2$}}
		\mathfrak{r}_{\scalebox{1.1}{$\scriptscriptstyle \mathrm{diff}$}}, x_{\scalebox{1.1}{$\scriptscriptstyle 2$}}
			x_{\scalebox{1.1}{$\scriptscriptstyle \mathrm{diff}$}}, y_{\scalebox{1.1}{$\scriptscriptstyle 2$}}d, t)$}
		\]
		for the characters $\chi_{\scalebox{1.1}{$\scriptscriptstyle 1$}}$ and $\chi_{\scalebox{1.1}{$\scriptscriptstyle 2$}}$, respectively. This completes the proof.
	\end{proof}
\end{para}

\subsection{Gauss sums of products of \texorpdfstring{$r$}~--th power residue symbols}\label{sec:Gauss sums}
\begin{para}
We now apply the results of the previous subsection to the Hecke characters $\chi_\mi$ constructed in \S\ref{sec: res symbols}.
Let
$
\mathfrak{m}_{\scalebox{1.1}{$\scriptscriptstyle 1$}}, \mathfrak{m}_{\scalebox{1.1}{$\scriptscriptstyle 2$}}
\in I(S)
$
be coprime $r$-th power-free ideals, and let
$
[\mathfrak{m}_{\scalebox{1.1}{$\scriptscriptstyle 1$}}] = E_{\scalebox{1.1}{$\scriptscriptstyle 1$}}, [\mathfrak{m}_{\scalebox{1.1}{$\scriptscriptstyle 2$}}] =
E_{\scalebox{1.1}{$\scriptscriptstyle 2$}}
$
with
$
E_{\scalebox{1.1}{$\scriptscriptstyle 1$}}, E_{\scalebox{1.1}{$\scriptscriptstyle 2$}} \in \mathscr{E}
$.
We now seek to compare the normalized Gauss sum
$
G(\chi_{\mathfrak{m}_{\scalebox{.9}{$\scriptscriptstyle 1$}}\mathfrak{m}_{\scalebox{.9}{$\scriptscriptstyle 2$}}})
$
and the product
$
G(\chi_{\mathfrak{m}_{\scalebox{.9}{$\scriptscriptstyle 1$}}})
G(\chi_{\mathfrak{m}_{\scalebox{.9}{$\scriptscriptstyle 2$}}})
$.

\begin{lem}\label{lem-Hecke-prod-G-sum-comp} The ratio
	\[
		\tau(E_{\scalebox{1.1}{$\scriptscriptstyle 1$}}, E_{\scalebox{1.1}{$\scriptscriptstyle 2$}}) = \frac{G(\chi_{\mathfrak{m}_{\scalebox{.9}{$\scriptscriptstyle 1$}}\mathfrak{m}_{\scalebox{.9}{$\scriptscriptstyle 2$}}})}
		{\chi_{\mathfrak{m}_{\scalebox{.9}{$\scriptscriptstyle 1$}}}
			\!\big(\prod_{\mathfrak{p}_{v} \mid \mathfrak{m}_{\scalebox{.85}{$\scriptscriptstyle 2$}}}
			\mathfrak{p}_{v}\big)
			\chi_{\mathfrak{m}_{\scalebox{.9}{$\scriptscriptstyle 2$}}}\!\big(\prod_{\mathfrak{p}_{v} \mid \mathfrak{m}_{\scalebox{.85}{$\scriptscriptstyle 1$}}}
			\mathfrak{p}_{v}\big)
			G(\chi_{\mathfrak{m}_{\scalebox{.9}{$\scriptscriptstyle 1$}}})
			G(\chi_{\mathfrak{m}_{\scalebox{.9}{$\scriptscriptstyle 2$}}})}
		\]
		depends only on the classes
		$
		E_{\scalebox{1.1}{$\scriptscriptstyle 1$}}, E_{\scalebox{1.1}{$\scriptscriptstyle 2$}} \in \mathscr{E}
		$.
	\end{lem}

\begin{proof} \!Set
$
E = [\mi_{\scalebox{1.1}{$\scriptscriptstyle 1$}}\mi_{\scalebox{1.1}{$\scriptscriptstyle 2$}}] \in \mathscr{E}$.
By \cref{lem-Hecke}, we can decompose
\[
G(\chi_{\mathfrak{m}_{\scalebox{.9}{$\scriptscriptstyle 1$}}\mathfrak{m}_{\scalebox{.9}{$\scriptscriptstyle 2$}}}) =
G(\psi_{\scalebox{1.}{$\scriptscriptstyle E$}}\widetilde{\chi}_{\mathfrak{m}_{\scalebox{.9}{$\scriptscriptstyle 1$}}\mathfrak{m}_{\scalebox{.9}{$\scriptscriptstyle 2$}}}) = \psi_{\scalebox{1.}{$\scriptscriptstyle E$}}\Big(\frac{y}{x}\Big)
\widetilde{\chi}_{\mathfrak{m}_{\scalebox{.9}{$\scriptscriptstyle 1$}}\mathfrak{m}_{\scalebox{.9}{$\scriptscriptstyle 2$}}}\Big(\frac{y_{\scalebox{.9}{$\scriptscriptstyle E$}}}{x_{\scalebox{.9}{$\scriptscriptstyle E$}}}\Big)
W\!\left(\psi_{\scalebox{1.}{$\scriptscriptstyle E$}};\frac{t_{\scalebox{.9}{$\scriptscriptstyle E$}}}{y_{\scalebox{.9}{$\scriptscriptstyle E$}}}, \frac{t_{\scalebox{.9}{$\scriptscriptstyle E$}}}{x_{\scalebox{.9}{$\scriptscriptstyle E$}}}\right)
\!W\!\left(\widetilde{\chi}_{\mathfrak{m}_{\scalebox{.9}{$\scriptscriptstyle 1$}}\mathfrak{m}_{\scalebox{.9}{$\scriptscriptstyle 2$}}};\frac{t}{y}, \frac{t}{x}\right)\!,
\]
where
$
\Delta_{\scalebox{1.1}{$\scriptscriptstyle E$}} = (\mathfrak{r}_{\scalebox{1.}{$\scriptscriptstyle E$}}, x_{\scalebox{1.}{$\scriptscriptstyle E$}}, y_{\scalebox{1.}{$\scriptscriptstyle E$}}, t_{\scalebox{1.}{$\scriptscriptstyle E$}})
$
and
$
\Delta = (\mathfrak{r}, x, y, t)
$
are two datasets for
$\psi_{\scalebox{1.1}{$\scriptscriptstyle E$}}$ and
$
\widetilde{\chi}_{\mathfrak{m}_{\scalebox{.9}{$\scriptscriptstyle 1$}}\mathfrak{m}_{\scalebox{.9}{$\scriptscriptstyle 2$}}}
$, respectively. As in the proof of \cref{lem-Hecke}, by taking
$
\Delta_{\mathfrak{m}_{\scalebox{.9}{$\scriptscriptstyle 1$}}, \mathfrak{m}_{\scalebox{.9}{$\scriptscriptstyle 2$}}} = (\mathfrak{r}_{\scalebox{1.1}{$\scriptscriptstyle 1$}}
\mathfrak{r}_{\scalebox{1.1}{$\scriptscriptstyle 2$}}
\mathfrak{r}_{\scalebox{1.1}{$\scriptscriptstyle \mathrm{diff}$}},
x_{\scalebox{1.1}{$\scriptscriptstyle 1$}}
x_{\scalebox{1.1}{$\scriptscriptstyle 2$}}
x_{\scalebox{1.1}{$\scriptscriptstyle \mathrm{diff}$}},
y_{\scalebox{1.1}{$\scriptscriptstyle 1$}} y_{\scalebox{1.1}{$\scriptscriptstyle 2$}}d, t)
$,
we can further decompose
\begin{align*}
	&\psi_{\scalebox{1.}{$\scriptscriptstyle E$}}\Big(\frac{y_{\scalebox{1.1}{$\scriptscriptstyle 1$}}y_{\scalebox{1.1}{$\scriptscriptstyle 2$}}d}{x_{\scalebox{1.1}{$\scriptscriptstyle 1$}}x_{\scalebox{1.1}{$\scriptscriptstyle 2$}}
		x_{\scalebox{1.1}{$\scriptscriptstyle \mathrm{diff}$}}}\Big)
	W\!\left(\widetilde{\chi}_{\mathfrak{m}_{\scalebox{.9}{$\scriptscriptstyle 1$}}\mathfrak{m}_{\scalebox{.9}{$\scriptscriptstyle 2$}}}; \frac{t}{y_{\scalebox{1.1}{$\scriptscriptstyle 1$}}y_{\scalebox{1.1}{$\scriptscriptstyle 2$}}d}, \frac{t}{x_{\scalebox{1.1}{$\scriptscriptstyle 1$}}x_{\scalebox{1.1}{$\scriptscriptstyle 2$}}
		x_{\scalebox{1.1}{$\scriptscriptstyle \mathrm{diff}$}}}\right)\\
	&= \psi_{\scalebox{1.}{$\scriptscriptstyle E$}}\Big(\frac{y_{\scalebox{1.1}{$\scriptscriptstyle 1$}}y_{\scalebox{1.1}{$\scriptscriptstyle 2$}}d}{x_{\scalebox{1.1}{$\scriptscriptstyle 1$}}x_{\scalebox{1.1}{$\scriptscriptstyle 2$}}
		x_{\scalebox{1.1}{$\scriptscriptstyle \mathrm{diff}$}}}\Big)
	\widetilde{\chi}_{\mathfrak{m}_{\scalebox{.9}{$\scriptscriptstyle 1$}}}\Big(\frac{y_{\scalebox{1.1}{$\scriptscriptstyle 2$}}}{x_{\scalebox{1.1}{$\scriptscriptstyle 2$}}}\Big)
	\widetilde{\chi}_{\mathfrak{m}_{\scalebox{.9}{$\scriptscriptstyle 2$}}}\Big(\frac{y_{\scalebox{1.1}{$\scriptscriptstyle 1$}}}{x_{\scalebox{1.1}{$\scriptscriptstyle 1$}}}\Big)
	W\!\left(\widetilde{\chi}_{\mathfrak{m}_{\scalebox{.9}{$\scriptscriptstyle 1$}}};\frac{t}{y_{\scalebox{1.1}{$\scriptscriptstyle 1$}}d}, \frac{t}{x_{\scalebox{1.1}{$\scriptscriptstyle 1$}}
		x_{\scalebox{1.1}{$\scriptscriptstyle \mathrm{diff}$}}}\right)	\!W\!\left(\widetilde{\chi}_{\mathfrak{m}_{\scalebox{.9}{$\scriptscriptstyle 2$}}};\frac{t}{y_{\scalebox{1.1}{$\scriptscriptstyle 2$}}d}, \frac{t}{x_{\scalebox{1.1}{$\scriptscriptstyle 2$}}
		x_{\scalebox{1.1}{$\scriptscriptstyle \mathrm{diff}$}}}\right)\!.
\end{align*}
Similarly, for $j = 1, 2$,
\[
G(\chi_{\mathfrak{m}_{\scalebox{.9}{$\scriptscriptstyle j$}}}) =
G(\psi_{\scalebox{1.}{$\scriptscriptstyle E_{\scalebox{.9}{$\scriptscriptstyle j$}}$}}\widetilde{\chi}_{\mathfrak{m}_{\scalebox{.9}{$\scriptscriptstyle j$}}}) = \psi_{\scalebox{1.}{$\scriptscriptstyle E_{\scalebox{.9}{$\scriptscriptstyle j$}}$}}\!\Big(\frac{y_{\scalebox{1.1}{$\scriptscriptstyle j$}}d}{x_{\scalebox{1.1}{$\scriptscriptstyle j$}}
	x_{\scalebox{1.1}{$\scriptscriptstyle \mathrm{diff}$}}}\Big)
\widetilde{\chi}_{\mathfrak{m}_{\scalebox{.9}{$\scriptscriptstyle j$}}}\!\Big(\frac{y_{\scalebox{.9}{$\scriptscriptstyle E_{\scalebox{.85}{$\scriptscriptstyle j$}}$}}}{x_{\scalebox{.9}{$\scriptscriptstyle E_{\scalebox{.85}{$\scriptscriptstyle j$}}$}}}\Big)
W\!\left(\psi_{\scalebox{1.}{$\scriptscriptstyle E_{\scalebox{.9}{$\scriptscriptstyle j$}}$}};\frac{t_{\scalebox{.9}{$\scriptscriptstyle E_{\scalebox{.85}{$\scriptscriptstyle j$}}$}}}{y_{\scalebox{.9}{$\scriptscriptstyle E_{\scalebox{.85}{$\scriptscriptstyle j$}}$}}}, \frac{t_{\scalebox{.9}{$\scriptscriptstyle E_{\scalebox{.85}{$\scriptscriptstyle j$}}$}}}{x_{\scalebox{.9}{$\scriptscriptstyle E_{\scalebox{.85}{$\scriptscriptstyle j$}}$}}}\right)
\!W\!\left(\widetilde{\chi}_{\mathfrak{m}_{\scalebox{.9}{$\scriptscriptstyle j$}}}; \frac{t}{y_{\scalebox{1.1}{$\scriptscriptstyle j$}}d}, \frac{t}{x_{\scalebox{1.1}{$\scriptscriptstyle j$}}
	x_{\scalebox{1.1}{$\scriptscriptstyle \mathrm{diff}$}}}
\right)\!.
\]
It follows that
\begin{align*}
&\frac{G(\chi_{\mathfrak{m}_{\scalebox{.9}{$\scriptscriptstyle 1$}}\mathfrak{m}_{\scalebox{.9}{$\scriptscriptstyle 2$}}})}
{\chi_{\mathfrak{m}_{\scalebox{.9}{$\scriptscriptstyle 1$}}}
	\!\big(\prod_{\mathfrak{p}_{v} \mid \mathfrak{m}_{\scalebox{.85}{$\scriptscriptstyle 2$}}}
	\mathfrak{p}_{v}\big)
	\chi_{\mathfrak{m}_{\scalebox{.9}{$\scriptscriptstyle 2$}}}\!\big(\prod_{\mathfrak{p}_{v} \mid \mathfrak{m}_{\scalebox{.85}{$\scriptscriptstyle 1$}}}
	\mathfrak{p}_{v}\big)G(\chi_{\mathfrak{m}_{\scalebox{.9}{$\scriptscriptstyle 1$}}})
	G(\chi_{\mathfrak{m}_{\scalebox{.9}{$\scriptscriptstyle 2$}}})}\\
& = \frac{\widetilde{\chi}_{\mathfrak{m}_{\scalebox{.9}{$\scriptscriptstyle 1$}}\mathfrak{m}_{\scalebox{.9}{$\scriptscriptstyle 2$}}}\big(\frac{y_{\scalebox{.7}{$\scriptscriptstyle E$}}}{x_{\scalebox{.7}{$\scriptscriptstyle E$}}}\big)
	W\big(\psi_{\scalebox{.9}{$\scriptscriptstyle E$}};\frac{t_{\scalebox{.7}{$\scriptscriptstyle E$}}}{y_{\scalebox{.7}{$\scriptscriptstyle E$}}}, \frac{t_{\scalebox{.7}{$\scriptscriptstyle E$}}}{x_{\scalebox{.7}{$\scriptscriptstyle E$}}}\big)}{\widetilde{\chi}_{\mathfrak{m}_{\scalebox{.9}{$\scriptscriptstyle 1$}}}\!\big(\frac{y_{\scalebox{.7}{$\scriptscriptstyle E_{\scalebox{.85}{$\scriptscriptstyle 1$}}$}}}{x_{\scalebox{.7}{$\scriptscriptstyle E_{\scalebox{.85}{$\scriptscriptstyle 1$}}$}}}\big)
	\widetilde{\chi}_{\mathfrak{m}_{\scalebox{.9}{$\scriptscriptstyle 2$}}}\!\big(\frac{y_{\scalebox{.7}{$\scriptscriptstyle E_{\scalebox{.85}{$\scriptscriptstyle 2$}}$}}}{x_{\scalebox{.7}{$\scriptscriptstyle E_{\scalebox{.85}{$\scriptscriptstyle 2$}}$}}}\big)
	W\!\big(\psi_{\scalebox{.9}{$\scriptscriptstyle E_{\scalebox{.9}{$\scriptscriptstyle 1$}}$}};\frac{t_{\scalebox{.7}{$\scriptscriptstyle E_{\scalebox{.85}{$\scriptscriptstyle 1$}}$}}}{y_{\scalebox{.7}{$\scriptscriptstyle E_{\scalebox{.85}{$\scriptscriptstyle 1$}}$}}}, \frac{t_{\scalebox{.7}{$\scriptscriptstyle E_{\scalebox{.85}{$\scriptscriptstyle 1$}}$}}}{x_{\scalebox{.7}{$\scriptscriptstyle E_{\scalebox{.85}{$\scriptscriptstyle 1$}}$}}}\big)
	W\!\big(\psi_{\scalebox{.9}{$\scriptscriptstyle E_{\scalebox{.9}{$\scriptscriptstyle 2$}}$}};\frac{t_{\scalebox{.7}{$\scriptscriptstyle E_{\scalebox{.85}{$\scriptscriptstyle 2$}}$}}}{y_{\scalebox{.7}{$\scriptscriptstyle E_{\scalebox{.85}{$\scriptscriptstyle 2$}}$}}}, \frac{t_{\scalebox{.7}{$\scriptscriptstyle E_{\scalebox{.85}{$\scriptscriptstyle 2$}}$}}}{x_{\scalebox{.7}{$\scriptscriptstyle E_{\scalebox{.85}{$\scriptscriptstyle 2$}}$}}}\big)}\\
& = \frac{\psi_{\scalebox{.9}{$\scriptscriptstyle E$}}(\varepsilon_{\!\scalebox{.9}{$\scriptscriptstyle E$}})
	W\big(\psi_{\scalebox{.9}{$\scriptscriptstyle E$}};\frac{t_{\scalebox{.7}{$\scriptscriptstyle E$}}}{y_{\scalebox{.7}{$\scriptscriptstyle E$}}}, \frac{t_{\scalebox{.7}{$\scriptscriptstyle E$}}}{x_{\scalebox{.7}{$\scriptscriptstyle E$}}}\big)}{\psi_{\scalebox{.9}{$\scriptscriptstyle E_{\scalebox{.9}{$\scriptscriptstyle 1$}}$}}(\varepsilon_{\!\scalebox{.85}{$\scriptscriptstyle E_{\scalebox{.85}{$\scriptscriptstyle 1$}}$}})
	\psi_{\scalebox{.9}{$\scriptscriptstyle E_{\scalebox{.9}{$\scriptscriptstyle 2$}}$}}(\varepsilon_{\!\scalebox{.85}{$\scriptscriptstyle E_{\scalebox{.85}{$\scriptscriptstyle 2$}}$}})
	W\!\big(\psi_{\scalebox{.9}{$\scriptscriptstyle E_{\scalebox{.9}{$\scriptscriptstyle 1$}}$}};\frac{t_{\scalebox{.7}{$\scriptscriptstyle E_{\scalebox{.85}{$\scriptscriptstyle 1$}}$}}}{y_{\scalebox{.7}{$\scriptscriptstyle E_{\scalebox{.85}{$\scriptscriptstyle 1$}}$}}}, \frac{t_{\scalebox{.7}{$\scriptscriptstyle E_{\scalebox{.85}{$\scriptscriptstyle 1$}}$}}}{x_{\scalebox{.7}{$\scriptscriptstyle E_{\scalebox{.85}{$\scriptscriptstyle 1$}}$}}}\big)
	W\!\big(\psi_{\scalebox{.9}{$\scriptscriptstyle E_{\scalebox{.9}{$\scriptscriptstyle 2$}}$}};\frac{t_{\scalebox{.7}{$\scriptscriptstyle E_{\scalebox{.85}{$\scriptscriptstyle 2$}}$}}}{y_{\scalebox{.7}{$\scriptscriptstyle E_{\scalebox{.85}{$\scriptscriptstyle 2$}}$}}}, \frac{t_{\scalebox{.7}{$\scriptscriptstyle E_{\scalebox{.85}{$\scriptscriptstyle 2$}}$}}}{x_{\scalebox{.7}{$\scriptscriptstyle E_{\scalebox{.85}{$\scriptscriptstyle 2$}}$}}}\big)}
\end{align*}
where
$
\varepsilon_{\!\scalebox{.85}{$\scriptscriptstyle E$}}
=\frac{x_{\scalebox{.7}{$\scriptscriptstyle E$}}}{y_{\scalebox{.7}{$\scriptscriptstyle E$}}}
$,
$
\varepsilon_{\!\scalebox{.85}{$\scriptscriptstyle E$}_{\scalebox{.7}{$\scriptscriptstyle 1$}}}
\! = \frac{x_{\scalebox{.7}{$\scriptscriptstyle E_{\scalebox{.85}{$\scriptscriptstyle 1$}}$}}}{y_{\scalebox{.7}{$\scriptscriptstyle E_{\scalebox{.85}{$\scriptscriptstyle 1$}}$}}}
$
and
$
\varepsilon_{\!\scalebox{.85}{$\scriptscriptstyle E$}_{\scalebox{.7}{$\scriptscriptstyle 2$}}}
\! = \frac{x_{\scalebox{.7}{$\scriptscriptstyle E_{\scalebox{.85}{$\scriptscriptstyle 2$}}$}}}{y_{\scalebox{.7}{$\scriptscriptstyle E_{\scalebox{.85}{$\scriptscriptstyle 2$}}$}}}
$
are $S$-units. This completes the proof.
\end{proof}

For any $E \in \mathscr{E}$, set
\be\label{eq: gamma psi}
\gamma(\psi_{\scalebox{.9}{$\scriptscriptstyle E$}}) =
\psi_{\scalebox{.9}{$\scriptscriptstyle E$}}(\varepsilon_{\!\scalebox{.85}{$\scriptscriptstyle E$}})
W\!\Big(\psi_{\scalebox{.9}{$\scriptscriptstyle E$}};\frac{t_{\scalebox{.9}{$\scriptscriptstyle E$}}}{y_{\scalebox{.9}{$\scriptscriptstyle E$}}}, \frac{t_{\scalebox{.9}{$\scriptscriptstyle E$}}}{x_{\scalebox{.9}{$\scriptscriptstyle E$}}}\Big)
\ee
with
$
\varepsilon_{\!\scalebox{.9}{$\scriptscriptstyle E$}}
= \frac{x_{\scalebox{.7}{$\scriptscriptstyle E$}}}{y_{\scalebox{.7}{$\scriptscriptstyle E$}}}
$. Note that
\be\label{eq: tau formula}
\tau(E_{\scalebox{1.1}{$\scriptscriptstyle 1$}}, E_{\scalebox{1.1}{$\scriptscriptstyle 2$}}) = \frac{\gamma(\psi_{\scalebox{.9}{$\scriptscriptstyle E$}_{\scalebox{.7}{$\scriptscriptstyle 1$}}\scalebox{.9}{$\scriptscriptstyle E$}_{\scalebox{.7}{$\scriptscriptstyle 2$}}})}
{\gamma(\psi_{\scalebox{.9}{$\scriptscriptstyle E$}_{\scalebox{.7}{$\scriptscriptstyle 1$}}}) \gamma(\psi_{\scalebox{.9}{$\scriptscriptstyle E$}_{\scalebox{.7}{$\scriptscriptstyle 2$}}})}.
\ee

\end{para}

\section{Twisted sums of Gauss sums}

\begin{para}In this section, we introduce the Dirichlet series $D^{T}(s, \mathfrak{a}, \psi) $
in one complex variable built out of $r$-th order Gauss sums, and discuss their analytic properties.
\!These series were studied by Kubota \cite{Kub69}, who showed that they arise as the Whittaker
coefficients of an Eisenstein series on the $r$-fold metaplectic cover of $\GL_2$.
They depend on a finite set of places $T$ containing our fixed set $S$, and, for later use,
it is essential to keep track of the factors in the functional equation for $D^T$ depending on the places $v\in T\setminus S$. We will do that by passing from $D^T$ to $D^S$
(\cref{corollary: transition-places}), applying the functional equation to $D^S$
(\cref{theorem:Kubota-func-eq}), and then returning from $D^S$ to $D^T$ (\cref{corollary: transition-places2}).
\end{para}

\subsection{Functional equation of Kubota's Dirichlet series}\label{sec:Kubota}

\begin{para}
We start by introducing the generalized Gauss sums used to define Kubota's Dirichlet series.
For $\mathfrak{a}, \mathfrak{b}\in I(S)$, define
	\[
	G_{0}(\mathfrak{a}, \mathfrak{b})
	= \prod_{\substack{v\\ \mathrm{ord}_{v}(\mathfrak{a}) = k\\ \mathrm{ord}_{v}(\mathfrak{b}) = l}}
	G_{0}(\mathfrak{p}_{v}^{k}, \mathfrak{p}_{v}^{l}),
	\]
	where, for $k, l \ge 0$,
	\[
	G_{0}(\mathfrak{p}_{v}^{k}, \mathfrak{p}_{v}^{l})
	= \begin{cases}
		1, & \text{if $l = 0$,}\\
		q_{v}^{k/2}, & \text{if $k = l - 1$; $l \not \equiv 0 \pmod r$,}\\
		- q_{v}^{(k - 1)/2}, & \text{if $k = l - 1$; $l > 0$; $l \equiv 0 \pmod r$,}\\
		q_{v}^{l/2}(1 - q_{v}^{-1}), & \text{if $k \ge l$; $l > 0$; $l \equiv 0 \pmod r$,}\\
		0, & \text{otherwise.}\\
	\end{cases}
	\]
	 Let $G(\chi_{\mathfrak{b}_{\scalebox{.95}{$\scriptscriptstyle 0$}}})$ be the normalized
	Gauss sum appearing in the functional equation of the (primitive) Hecke $L$-function associated with $\chi_{\mathfrak{b}_{\scalebox{.95}{$\scriptscriptstyle 0$}}}$, defined in~\eqref{eq: Gauss sum Hecke char}.
	If $\mathfrak{a}^{(\bi_0)}$ denotes
	the part of $\mathfrak{a}$ coprime to $\mathfrak{b}_{\scalebox{1.25}{$\scriptscriptstyle 0$}}$,
	then set
	\[
	G(\mathfrak{a}, \mathfrak{b}) : =
	\overline{\chi}_{\mathfrak{b}_{\scalebox{.95}{$\scriptscriptstyle 0$}}}(\mathfrak{a}^{(\bi_0)})
	G(\chi_{\mathfrak{b}_{\scalebox{.95}{$\scriptscriptstyle 0$}}})
	G_{0}(\mathfrak{a}, \mathfrak{b}).
	\]
\end{para}

\begin{para}\label{par:Kub_def}
Let $S$ be as before, and let $S \subseteq T$ be a larger finite set of places. For $\mathfrak{a} \in I(T)$ and a Hecke character $\psi$ unramified outside $T$ and of order dividing $r$, let
	\be\label{eq:Kubota_def}
	D^{T}(s, \mathfrak{a}, \psi) : =
	\sum_{\mathfrak{b} \in I(T)} \psi(\mathfrak{b})G(\mathfrak{a}, \mathfrak{b})|\mathfrak{b}|^{-s}.
	\ee
	Also, for $E \in \mathscr{E}$ a representative of a class in $R_{\mathfrak{c}}$, let
	\[
	D^{T}(s, \mathfrak{a}, \psi; E) :=
	\sum_{\substack{\mathfrak{b} \in I(T)\\ [\mathfrak{b}] = E}} \psi(\mathfrak{b})G(\mathfrak{a}, \mathfrak{b})|\mathfrak{b}|^{-s}.
	\]
	Using the definition of $G(\mathfrak{a}, \mathfrak{b})$, it is easy to see that these series
	converge absolutely for $\Re(s) > 1$, and in this half-plane, we can write
	\be\label{eq: DTE}
	D^{T}(s, \mathfrak{a}, \psi; E) = \frac{1}{|R_{\mathfrak{c}}|}\sum_{\xi}
	\bar{\xi}(E)
	D^{T}(s, \mathfrak{a}, \psi \xi)
	\ee
	the sum being over all id\`ele class characters, \!unramified outside $S$, of conductor at
most $\mathfrak{c}$, such that $\xi^r = 1$.
\end{para}

\begin{para} \label{par:G}
To discuss the analytic properties of these series, set
\[
G(s) : = (2 \pi)^{-(r-1)s} r^{rs-(r+1)/2} \prod_{j = 1}^{r-1}
\Gamma(s -\tfrac{1}{2} + \tfrac{j}{r}) =
(2\pi)^{-(r-1)\left(s-\frac{1}{2}\right)}
\frac{\Gamma\!\left(rs-\tfrac{r}{2}\right)}{\Gamma\!\left(s-\tfrac{1}{2}\right)},
\]
where we have applied the multiplication formula for the Gamma function. Define
\[
\widetilde{D}^{T}(s, \mathfrak{a}, \psi) =
\zeta_{\scalebox{1.1}{$\scriptscriptstyle F$}}^{\scalebox{1.2}{$\scriptscriptstyle T$}}
(rs - r/2  + 1)
D^{T}(s, \mathfrak{a}, \psi),
\]
and $\widehat{D}^{T}(s, \mathfrak{a}, \psi)=G(s)^{[F : \Q]/2} \widetilde{D}^{T}(s, \mathfrak{a}, \psi)$,
where $\zeta_{F}$ denotes the Dedekind zeta function of $F$.
For $E \in \mathscr{E}$, denote
\[
\widetilde{D}^{T}(s, \mathfrak{a}, \psi; E) =
\zeta_{\scalebox{1.1}{$\scriptscriptstyle F$}}^{\scalebox{1.2}{$\scriptscriptstyle T$}}
(rs - r/2  + 1)
D^{T}(s, \mathfrak{a}, \psi; E),
\]
and similarly for $\widehat{D}^{T}(s, \mathfrak{a}, \psi; E)$. Then we have the following result,
which can be derived from the version proved in \cite[Thm. 1]{BB06}:
\end{para}

\begin{thm}[\cite{BB06}, Theorem 1]\label{theorem:Kubota-func-eq}
	\!The function $\widehat{D}^{S}(s, \mathfrak{a}, \psi)$ has meromorphic continuation to all $\C$,
	analytic except possibly at $s = \tfrac{1}{2} \pm \tfrac{1}{r}$, where it might have simple poles.
	Moreover, letting $F_{S}^{\times} : = \prod_{v \in S} F_{v}^{\times}$,
	there exists a set of polynomials $P_{\eta}(s) \in \C[\{q_{v}^{-s}\}_{v \in S_{\mathrm{fin}}}]$,
	parametrized by elements $\eta \in F_{S}^{\times}$ and depending only on the image
	of $\eta$ in $F_{S}^{\times}/(F_{S}^{\times})^{r}$, such that
	\begin{align*}
		\widehat{D}^{S}(s, \mathfrak{a}, \psi) & =
		\psi(\mathfrak{a})
		|\mathfrak{a}|^{\frac{1}{2}-s}
		\prod_{v \in S_{\mathrm{fin}}} \Big(1 - q_{v}^{rs-\frac{r}{2}-1}\Big)^{-1}\\
		&\cdot \sum_{\eta \in F_{S}^{\times}/(F_{S}^{\times})^{^{r}}}
		P_{\eta}(s) \overline{\psi}(\eta)\sum_{E \in \mathscr{E}}
		K_{\eta, E, [\mathfrak{a}]}\widehat{D}^{S}(1 - s, \mathfrak{a}, \overline{\psi}; E)
	\end{align*}
where $K_{\eta, E, [\mathfrak{a}]}$ are constants on the unit circle depending on
$\eta, E$ and the class of $\mathfrak{a}$ in $R_{\mathfrak{c}}$.
\end{thm}

\begin{para} Using~\eqref{eq: DTE}, the functional equation can also be rewritten as
\[\widehat{D}^{S}(s, \mathfrak{a}, \psi)  =
		\psi(\mathfrak{a})
		|\mathfrak{a}|^{\frac{1}{2}-s}
		\prod_{v \in S_{\mathrm{fin}}} \Big(1 - q_{v}^{rs-\frac{r}{2}-1}\Big)^{-1}\sum_{\xi\in \widehat{R}_\ci} J_{\xi,\psi,[\ai]}(s) \widehat{D}^{S}(1 - s, \mathfrak{a}, \overline{\psi}\xi),
\]
 with $J_{\xi,\psi,[\ai]}(s)$ bounded if $\Re(s)$ is bounded.
\end{para}

\begin{para}
Brubaker and Bump \cite[Theorem 1]{BB06} prove this theorem for Dirichlet series of the form
$$
\Eu{D}(s, \Psi, \alpha) = \sum_{\substack{c \in \mathscr{O}_{S}/\mathscr{O}_{S}^{\times} \\ c \ne 0}}
\Psi(c)G(\alpha, c)|c|^{-s}
$$
with notation as in \textit{loc. cit.} To obtain the above theorem, one chooses
$$
\Psi(c, \xi):= \psi_{\scalebox{1.05}{$\scriptscriptstyle [c]$}}(c)\overline{\psi_{\scalebox{1.05}{$\scriptscriptstyle [c]$}}(\alpha)}\xi(c)
\gamma(\psi_{\scalebox{1.05}{$\scriptscriptstyle [c]$}}),
$$
where $\xi$ is a unitary id\`ele class character unramified outside $S$ and of order dividing $r$. Indeed, by comparing coefficients, one checks that
$$
D^{S}(s, \mathfrak{a}, \psi) =
\sum_{\mathfrak{b} \in I(S)} \psi(\mathfrak{b})G(\mathfrak{a}, \mathfrak{b})|\mathfrak{b}|^{-s} = \Eu{D}(s, \Psi(\cdot, \psi), \alpha),
$$
where $\mathfrak{a} = (\alpha)$ as an $\mathscr{O}_{S}$-ideal, and that
$$
D^{S}(s, \mathfrak{a}, \psi; E) = \gamma(\psi_{\scalebox{1.05}{$\scriptscriptstyle E$}})\sum_{\substack{c \in \mathscr{O}_{S}/\mathscr{O}_{S}^{\times}\\ [c] = E}}
\psi_{\scalebox{1.05}{$\scriptscriptstyle E$}}(c)\overline{\psi_{\scalebox{1.05}{$\scriptscriptstyle E$}}(\alpha)}\psi(c)
G(\alpha, c)|c|^{-s}.
$$
\end{para}

\begin{para}
    It is clear from the proof that the function $\big(s - \tfrac{1}{2} - \tfrac{1}{r}\big)\big(s - \tfrac{1}{2} + \tfrac{1}{r}\big)\widetilde{D}^{S}(s, \ai, \psi)$ is entire of order $1$, and hence it satisfies a convexity bound.
\end{para}

\begin{para}
    We imposed the condition that the underlying number field contains the full group $\mu_{2r}$ of $2r$-th roots of unity strictly as a matter of convenience, enabling us to invoke \cite[Theorem 1]{BB06}. Relaxing this hypothesis is certainly possible. Furthermore, as noted in the introduction, explicitly evaluating the residue of $\widehat{D}^{S}(s, \mathfrak{a}, \psi)$ at the simple pole $s = \tfrac{1}{2} + \tfrac{1}{r}$ remains an outstanding open problem for $r \ge 4$.

\end{para}

\subsection{Relations between \texorpdfstring{$D^T$} ~and \texorpdfstring{$D^S$}~}

\begin{para} The following lemma relates the functions $D^{T}$ and $D^{T'}$
for different sets $T, T'$ containing $S$.
\end{para}

\begin{lem}\label{lemma:DE-place-comp} With the notation in \S\ref{par:Kub_def}, let
	$v_{\mathfrak{p}} \notin T$ be a place, with $\mathfrak{p}$ the corresponding prime ideal of $\mathscr{O}$,
	and let $\psi$ be an $r$-th order Hecke character unramified outside $T$.
	For $\qi$ a power of $\pp$, define
	\[C_{\qi}(\ai,\psi;E)=-\chi_{\qi}(\ai)
\ov{\psi}(\qi ) \psi_{\sss E} (m_{[\qi]})
\ov{\tau(E,[\qi])} \ov{G(\chi_\qi)},
	\]
	where $\tau(\cdot,\cdot)$ is given by \eqref{eq: tau formula}. Then,
	for $E \in \mathscr{E}$ and $\mathfrak{a} \in I(T \cup \{v_{\mathfrak{p}} \})$, we have
	\begin{equation*}
		\begin{split}
	\Big(1 - |\mathfrak{p}|^{\frac{r}{2}-1-rs}\Big)&D^{T \cup \{v_{\mathfrak{p}} \}}(s, \mathfrak{a}, \psi \widetilde{\chi}_{\mathfrak{p}}^{k}; E) =
		D^{T}(s, \mathfrak{p}^{r-k} \mathfrak{a}, \psi; E) \\
		& + C_{\mathfrak{p}^{k-1}}(\ai, \psi;E)
		|\mathfrak{p}|^{\frac {r-k}2-(r-k + 1)s}
		D^{T}(s, \mathfrak{p}^{k - 2} \mathfrak{a}, \psi; [E \pp^{k-1}])
\end{split}
\end{equation*}
for $2\le k \le r $, and
	\[
	\Big(1 - |\mathfrak{p}|^{\frac{r}{2} - 1 - rs}\Big)D^{T \cup \{v_{\mathfrak{p}} \}}(s, \mathfrak{a}, \psi \widetilde{\chi}_{\mathfrak{p}}; E) =
	D^{T}(s, \mathfrak{p}^{r - 1} \mathfrak{a}, \psi; E)
	\]
	for $k = 1$.
\end{lem}

\begin{proof} Let $0\le k\le r-1$. We assume $\Re(s)>1$, as the proof in general follows
by meromorphic continuation. Using the properties of the Gauss sums, we have
\begin{equation*}
	\begin{split}
		D^{T}(s, \mathfrak{p}^{k} \mathfrak{a}, \psi; E) & =
		\sum_{\substack{\mathfrak{b} \in I(T)\\ [\mathfrak{b}] = E}} \frac{\psi(\mathfrak{b})G(\mathfrak{p}^ k \mathfrak{a}, \mathfrak{b})}{|\mathfrak{b}|^{s}}\\
		& = \sum_{l \ge 0} \sum_{\substack{\mathfrak{b} \in I(T
				\cup \{v_{\mathfrak{p}} \})\\ [\mathfrak{p}^l \mathfrak{b}] = E}}
			\frac{\psi(\mathfrak{p}^l \mathfrak{b})G(\mathfrak{p}^ k \mathfrak{a}, \mathfrak{p}^l \mathfrak{b})}{|\mathfrak{b}|^{s}|\mathfrak{p}|^{ls}}\\
			& = \sum_{\substack{\mathfrak{b} \in I(T
					\cup \{v_{\mathfrak{p}} \})\\ [\mathfrak{b}] = E}}
			\frac{\psi(\mathfrak{b})G(\mathfrak{p}^ k \mathfrak{a}, \mathfrak{b})}{|\mathfrak{b}|^{s}} \, + \,
			\sum_{\substack{\mathfrak{b} \in I(T
			\cup \{v_{\mathfrak{p}} \})\\ [\mathfrak{p}^{k + 1} \mathfrak{b}] = E}}
			\frac{\psi(\mathfrak{p}^{k + 1} \mathfrak{b})G(\mathfrak{p}^ k \mathfrak{a}, \mathfrak{p}^{k + 1} \mathfrak{b})}{|\mathfrak{b}|^{s}|\mathfrak{p}|^{(k + 1)s}}\\
			& =
			D^{T \cup \{v_{\mathfrak{p}} \}}(s, \mathfrak{a}, \psi \widetilde{\chi}_{\mathfrak{p}}^{-k}; E)+
			D^{T \cup \{v_{\mathfrak{p}} \}}(s, \mathfrak{a}, \psi \widetilde{\chi}_{\mathfrak{p}}^{k + 2}; [E \pp^{- k - 1}])\cdot
			\\
			& 
			\qquad \qquad\cdot \tau([E\pp^{-k-1}],[\pp^{k+1}])\ov{\psi}_{[{\sss E\pp^{-k-1}}]}(m_{[{\sss \pp^{k+1}}]})
			\frac{\psi^{k + 1}(\mathfrak{p})G(\mathfrak{p}^k \mathfrak{a}, \mathfrak{p}^{k + 1})}{|\mathfrak{p}|^{(k + 1)s}},
\end{split}
\end{equation*}
where in the last equality we used Lemma~\ref{lem-Hecke-prod-G-sum-comp} and the fact that
if $G(\ai,\bi)\ne 0$, $\pp^l\| \bi$, and $\pp|\bi_0$, then $\pp^{l-1}\| \ai$.
When $k = r - 1$, the formula above reduces to
\[
D^{T}(s, \mathfrak{p}^{r - 1} \mathfrak{a}, \psi; E) =
\Big(1 - |\mathfrak{p}|^{\frac{r}{2} - 1 - rs}\Big)D^{T \cup \{v_{\mathfrak{p}} \}}(s, \mathfrak{a}, \psi \widetilde{\chi}_{\mathfrak{p}}; E).
\]
For $0\le k \le r-2$, using the formula for $G(\mathfrak{p}^k \mathfrak{a}, \mathfrak{p}^{k + 1})$ we obtain:
\[\begin{aligned}
D^{T}(s, \mathfrak{p}^{k} \mathfrak{a}, \psi; E)
&=D^{T \cup \{v_{\mathfrak{p}} \}}(s, \mathfrak{a}, \psi \widetilde{\chi}_{\mathfrak{p}}^{r-k}; E)\\
&\qquad + \wC_{\mathfrak{p}^{k+1}}(\ai, \psi;E) |\mathfrak{p}|^{\frac k2-(k + 1)s}D^{T \cup \{v_{\mathfrak{p}} \}}(s, \mathfrak{a}, \psi \widetilde{\chi}_{\mathfrak{p}}^{k + 2}; [E {\mathfrak{p}}^{- k - 1}]),
\end{aligned}
\]
where for $\qi=\pp^{k+1}$ we define:
\[
	\wC_{\qi}(\ai, \psi;E) = \ov{\chi}_{\qi}(\ai)\psi(\qi)\ov{\psi}_{[{\sss E \qi^{-1}}]}(m_{\sss [\qi]})
	\cdot \tau([E\qi^{-1}],[\qi]) G(\chi_\qi).
	\]
We pair this identity with the one corresponding to $r - k - 2$ in which $E$ is replaced by $[E {\mathfrak{p}}^{- k - 1}]$.
Using that $\wC_{\pp^{k+1}}(\ai,\psi;E)\cdot \wC_{\pp^{r-k-1}}(\ai,\psi;[E\pp^{-k-1}])=1  $, we  solve the system for
$D^{T \cup \{v_{\mathfrak{p}} \}}(s, \mathfrak{a}, \psi \widetilde{\chi}_{\mathfrak{p}}^{r-k}; E)$. The stated identity follows
after replacing $k$ by $r-k$, using that
\[
\wC_{\pp^{r-k+1}}(\ai,\psi;E)=-C_{\pp^{k-1}}(\ai,\psi;E).\qedhere
\]
\end{proof}



\begin{para}\label{par:fromTtoS}
	Let $T$ be a finite set of places that contains all the places in $S$. Let $a(k)=r-k$, $b(k)=k-2$ be the functions
	 appearing in the exponents of $\pp$ in the previous lemma.  For an ideal
\[\mathfrak{f} = \prod_{v \in T \setminus S} \mathfrak{p}_{v}^{k_{v}}, \text{ with } 1 \le k_{v} \le r,\]
consider the set consisting of tuples of functions
\[
E_{\mathfrak{f}}=\Big\{ (e_v)_{v\in T \setminus S} \Big|\  \begin{matrix}
                                                       e_v=a \text{ if }\ k_v=1\hfill   \\
														e_v\in\{a,b\} \text{ if }\ k_v\ge 2
                                                       \end{matrix}\Big\}.
\]
For $\eb=(e_v)_{v\in T \setminus S}\in E_\ff$, let $S_\eb=\{v\in T \setminus S : e_v=b \}$, and define the integral ideals
\[
\gi_\eb=\prod_{v \in T \setminus S} \pp_{v}^{e_{v}(k_v)},\quad
\ff_\eb=\prod_{v\in S_{\eb} } \pp_v^{k_v-1},\quad \ri_\eb=\prod_{v\in T\setminus (S\cup S_{\eb}) } \pp_v^{k_v}.
\]
We can apply \cref{lemma:DE-place-comp} recursively to express the function
$
\widetilde{D}^{T}(s, \mathfrak{a}, \psi \widetilde{\chi}_{\mathfrak{f}}; E)
$
in terms of the functions
$
\widetilde{D}^{S}(s, \mathfrak{g_\eb}\mathfrak{a}, \psi; [E \ff_{\eb}])
$
for $\eb \in E_\ff$. We will need the function
\[
C_{\qi,\ri}(\psi;E)= (-1)^{\omega(\qi)}\ov{\chi}_\qi(\ri)
\ov{\psi}(\qi ) \psi_{\sss E} (m_{[\qi]})
\ov{\tau(E,[\qi])} \ov{G(\chi_\qi)},
\]
with $\omega(\qi)$ the number of distinct prime divisors of $\qi$.

\end{para}

\begin{cor}\label{corollary: transition-places} Let the notation be as above. \!Then, for a unitary id\`ele class character $\psi$ unramified outside $S$ and of order dividing $r$, $\mathfrak{a} \in I(T)$ and $E \in \mathscr{E}$, we have
\[
\widetilde{D}^{T}(s, \mathfrak{a}, \psi \widetilde{\chi}_{\mathfrak{f}}; E) =
\sum_{\eb \in E_{\mathfrak{f}}}\chi_{\ff_\eb}(\ai)
C_{\ff_{\eb}, \ri_\eb}(\psi;E)\prod_{v \in S_{\eb}}
 q_v^{\frac{r-k_v}{2}-(r-k_v+1)s}\widetilde{D}^{S}(s, \gi_\eb\mathfrak{a}, \psi; [E\ff_{\eb}]).
\]


\end{cor}
\begin{proof}
Note that for $\qi=\pp^{k-1}$, $\ri=\OO$, the function
$C_{\qi}(\ai,\psi;E)$  from the previous lemma equals $\chi_{\qi}(\ai)C_{\qi,\ri}(\psi;E). $
The statement then follows by induction on the number of places in $T$ using \cref{lem-Hecke-prod-G-sum-comp}.
\end{proof}

\begin{para}We also need a converse of the previous corollary. Let $a(k)=r-k$ as before, and let $b^{-1}(k)=k+2$ be the inverse of
the function $b$ from~\ref{par:fromTtoS}. Define the constant on the unit circle
\[
\wC_{\qi,\ri}(\psi;E)= \ov{\chi}_\qi(\ri)
\psi(\qi ) \ov{\psi}_{[{\sss E\qi^{-1}}]} (m_{[\qi]})
\tau([E\qi^{-1}],[\qi]) G(\chi_\qi).
\]

\begin{cor}\label{corollary: transition-places2}
For an ideal
$\gi = \prod_{v \in T \setminus S} \mathfrak{p}_{v}^{k_{v}}, \text{ with } 0 \le k_{v} \le r-1,$
consider the set consisting of tuples of functions
\[
E_{\gi}=\Big\{ (e_v)_{v\in T \setminus S} \Big|\  \begin{matrix}
                                                         e_v=a \text{ if }\ k_v=r-1\hfill   \\
														 e_v\in\{a,b^{-1}\} \text{ if }\ k_v\le r-2
                                                       \end{matrix}\Big\}.
\]
For $\eb=(e_v)_{v\in T \setminus S}\in E_\gi$, let $S_\eb=\{v\in T \setminus S : e_v=b^{-1} \}$, and define the integral ideals
\[
\gi_\eb=\prod_{v \in T \setminus S} \pp_{v}^{e_{v}(k_v)},\quad
\ff_\eb=\prod_{v\in S_{\eb} } \pp_v^{k_v+1},\quad \ri_\eb=\prod_{v\in T\setminus (S\cup S_{\eb}) } \pp_v^{k_v}.
\]
\!Then, for a unitary id\`ele class character $\psi$ unramified outside $S$ and of order dividing $r$, $\mathfrak{a} \in I(T)$ and $E \in \mathscr{E}$, we have
\[\begin{aligned}
\widetilde{D}^{S}(s, \ai\gi, \psi; E) =\prod_{\substack{v \in T\setminus S \\ k_v\ne r-1}} (1-q_v^{\frac r2-1-rs})^{-1}
&\sum_{\eb \in E_{\gi}}\ov{\chi}_{\ff_\eb}(\ai) \wC_{\ff_{\eb}, \ri_\eb}(\psi;E)\prod_{v \in S_{\eb}}
 q_v^{\frac{k_v}{2}-(k_v+1)s}\\
&\cdot \widetilde{D}^{T}(s, \mathfrak{a}, \psi\widetilde{\chi}_{\gi_\eb}; [E\ff_{\eb}^{-1}]).
\end{aligned}
\]
\end{cor}
\begin{proof}Note that $\ov{\chi}_\qi(\ai) \wC_{\qi,\ri}(\psi;E)$ becomes $\wC_{\qi}(\ai,\psi;E)$ in the proof of \cref{lemma:DE-place-comp}
for $\qi=\pp^{k+1}$, $\ri=\OO$. The statement follows by induction on the number of places in $T$ using the relations in the proof
of \cref{lem-Hecke-prod-G-sum-comp}.
\end{proof}
\end{para}

\section{Construction of the MDS associated with the second moment}
\label{sec:ZMDS}

\begin{para}\label{par:ZMDS}
In this section, we review the construction of the
perfect MDS (i.e., a multiple Dirichlet series
possessing meromorphic continuation and satisfying a group of functional equations), associated with the second moment,  following \cite{Di04}. Because of the sieving
procedure outlined in \S\ref{sec:strategy-alt}, we work with a slightly more
general version depending on a finite set $T$ containing $S$; this  introduces
some additional technical complications. We also keep careful track of the factors
appearing in the functional equations, since these factors determine the
strength of the convexity bounds for the MDS.

The perfect MDS is defined by
 	\[
	Z^T(\sb;\psi_1,\psi_2,\rho)=\sum_{\ai\in I(T)} \frac{L^{T}(s_1, \psi_1\chi_{\ai_0})L^{T}(s_2, \ov{\psi}_2\ov{\chi}_{\ai_0})}{|\ai|^{s_3}}
	\rho(\ai)P_\ai(s_1,s_2;\psi_1,\psi_2),
	\]
 where 	$\psi_1,\psi_2,\rho$ are $r$-th order Hecke characters unramified outside $T$. The subtle point is how to choose the ``correction Dirichlet polynomials" $P_\ai(s_1,s_2;\psi_1,\psi_2)$, so that $Z^T$ has a large group of functional equations, and therefore good analytic properties.
 We determine formulas for the correction polynomials in \cref{prop:Z0},
 starting with a canonical object $Z_\aux^T({\sb}; \psi_1, \psi_2, \rho)$ constructed from a pair of Kubota Dirichlet series.
\end{para}

\begin{para}
Recall that for an ideal $\mathfrak{a} \in I(S)$, we denote by
	$\mathfrak{a}_{\scalebox{1.25}{$\scriptscriptstyle 0$}}$ its $r$-th power-free part, and
	by
	$
	\mathfrak{a}_{\scalebox{1.25}{$\scriptscriptstyle 1$}} = \prod_{\mathfrak{p}_{v} \mid \mathfrak{a}_{\scalebox{.85}{$\scriptscriptstyle 0$}}}
	\mathfrak{p}_{v}
	$
	the square-free kernel of $\mathfrak{a}_{\scalebox{1.25}{$\scriptscriptstyle 0$}}$.
	\end{para}
\subsection{An MDS constructed from Kubota's Dirichlet series}

\begin{para}Fix a finite set of places $T$ containing $S$.
Our starting point is the following multiple Dirichlet series introduced in~\cite{Di04}:
\be\label{eq: Zaux}
\begin{aligned}
Z_\aux^T({\sb}; \psi_1, \psi_2, \rho)  &=\sum_{\ai\in I(T)}	\frac{
		D^{T}(s_1, \mathfrak{a}, \psi_{\scalebox{1.}{$\scriptscriptstyle 1$}} \rho)
		\overline{D^{T}(\overline{s_2}, \mathfrak{a},
			\psi_{\scalebox{1.}{$\scriptscriptstyle 2$}}
			\overline{\rho})}
		\, \overline{\rho}(\mathfrak{a})}
	{|\mathfrak{a}|^{s_3}}\\
	&=\sum_{\ai,\bi,\bi'\in I(T)}\frac{\psi_1\rho(\bi)\ov{\psi}_2\rho(\bi') G(\ai,\bi)\ov{G(\ai,\bi')} \ov{\rho}(\ai) }{|\bi|^{s_1}|\bi'|^{s_2}|\ai|^{s_3}},
	\end{aligned}
\ee
where $D^T$ was defined in~\eqref{eq:Kubota_def} and $\psi_1, \psi_2, \rho$
are $r$-th order characters unramified outside $T$. We shall always consider such characters
of the type $\eta\chi_\ff$, with $\eta$ an $r$-th order Hecke character unramified outside $S$,
 and $\ff$ an $r$-th power-free ideal supported on the primes in $T\setminus S$.

The goal of this section is to show that the family of functions
	 $\{Z_\aux^T({\sb}; \psi_1, \psi_2, \rho)\}_{\psi_1,\psi_2,\rho}$ is related by a functional equation
	to the family $\{ Z^T(\sb;\psi_1,\psi_2,\rho)\}_{\psi_1,\psi_2,\rho}$  introduced in \S\ref{par:ZMDS}.
\end{para}

\subsection{The product \texorpdfstring{$ G(\mathfrak{a}, \mathfrak{b})\overline{G(\mathfrak{a}, \mathfrak{b}')} $}~  }
\label{sec:product-Gauss}
\begin{para}
	Let $\mathfrak{a}, \mathfrak{b}, \mathfrak{b}' \in I(S)$, with
	$
	\mathfrak{e} = (\mathfrak{b}, \mathfrak{b}')
	$
	and write $\mathfrak{b} = \mathfrak{m}\mathfrak{e}$ and $\mathfrak{b}' = \mathfrak{m}' \mathfrak{e}$. Let $\mathfrak{p}_v$ be a prime divisor of $\mathfrak{e}$.
	Then, from the definition of $G_{0}(\mathfrak{a}, \mathfrak{b})$, the product
	$
	G(\mathfrak{a}, \mathfrak{b})\overline{G(\mathfrak{a}, \mathfrak{b}')}
	$
	vanishes unless one of the following conditions holds:
	\begin{enumerate}
		\item $\mathfrak{p}_{v}^{l} \parallel \mathfrak{e}$ with $l \not \equiv 0 \pmod r$,
		$\mathfrak{p}_v \nmid \mathfrak{m}\mathfrak{m}'$, and $\mathfrak{p}_{v}^{l - 1} \parallel \mathfrak{a}$.

		\item $\mathfrak{p}_{v}^{l} \parallel \mathfrak{e}$ with $l \equiv 0 \pmod r$, $\mathfrak{p}_v \nmid \mathfrak{m}\mathfrak{m}'$, and $\mathfrak{p}_{v}^{k} \parallel \mathfrak{a}$ for some $k \ge l - 1$.

		\item $\mathfrak{p}_{v}^{l} \parallel \mathfrak{e}$ with $l \equiv 0 \pmod r$, $\mathfrak{p}_v \mid \mathfrak{m}\mathfrak{m}'$, and $\mathfrak{p}_{v}^{k} \parallel \mathfrak{a}$ for some $k \ge l - 1 + \mathrm{ord}_{v}(\mathfrak{m}\mathfrak{m}')$.
	\end{enumerate}
	Thus, assuming that the product of Gauss sums is non-vanishing, we can split
	$
	\mathfrak{e} =
	\mathfrak{e}^{\scalebox{.95}{$\scriptscriptstyle (1)$}} \mathfrak{e}^{\scalebox{.95}{$\scriptscriptstyle (2)$}} \mathfrak{e}^{\scalebox{.95}{$\scriptscriptstyle (3)$}},
	$  (depending only on $\mi$, $\mi'$)
	with $\mathfrak{e}^{\scalebox{.95}{$\scriptscriptstyle (j)$}}$ containing the powers of prime ideals dividing $\mathfrak{e}$ that satisfy the corresponding condition above. Hence, we can write
	\begin{equation*}
		\begin{split}
			G(\mathfrak{a}, \mathfrak{b})\overline{G(\mathfrak{a}, \mathfrak{b}')}
			& = G(\mathfrak{a}', \mathfrak{m}\mathfrak{e}^{\scalebox{.95}{$\scriptscriptstyle (1)$}})\overline{G(\mathfrak{a}', \mathfrak{m}'\mathfrak{e}^{\scalebox{.95}{$\scriptscriptstyle (1)$}})}\\
			&\cdot \prod_{\substack{\mathfrak{p}_{v}^{l} \parallel \mathfrak{e}^{\scalebox{.85}{$\scriptscriptstyle (2)$}}\\
					\mathfrak{p}_{v}^{k}  \parallel \mathfrak{a} \\ k \ge l - 1}}
			\overline{\chi}_{\mathfrak{m}}\chi_{\mathfrak{m'}}(\mathfrak{p}_{v}^{k})
			G_{0}(\mathfrak{p}_{v}^{k}, \mathfrak{p}_{v}^{l})^2
			\cdot |\mathfrak{e}^{\scalebox{.95}{$\scriptscriptstyle (3)$}}|
			\prod_{\mathfrak{p}_{v} \mid \mathfrak{e}^{\scalebox{.85}{$\scriptscriptstyle (3)$}}}
			(1 - q_{v}^{-1})
		\end{split}
	\end{equation*}
	where we set
	\[
	\mathfrak{a}' = \mathfrak{a}
	(\mathfrak{e}^{\scalebox{.95}{$\scriptscriptstyle (3)$}})^{-1}
	\Big(\prod_{\mathfrak{p}_{v} \mid (\mathfrak{a}, \mathfrak{e}^{\scalebox{.85}{$\scriptscriptstyle (2)$}})} \mathfrak{p}_{v}^{k}\Big)^{-1}.
	\]
	It is easy to see that
	$
	G(\mathfrak{a}', \mathfrak{m}\mathfrak{e}^{\scalebox{.95}{$\scriptscriptstyle (1)$}})\overline{G(\mathfrak{a}', \mathfrak{m}'\mathfrak{e}^{\scalebox{.95}{$\scriptscriptstyle (1)$}})} = 0
	$, unless
	\[
	\mathfrak{a}' =
	\mathfrak{n} \cdot
	\frac{\mathfrak{e}^{\scalebox{.95}{$\scriptscriptstyle (1)$}}}
	{\mathfrak{e}_{\scalebox{.95}{$\scriptscriptstyle 1$}}^{\scalebox{.95}{$\scriptscriptstyle (1)$}}}
	\prod_{\substack{\mathfrak{p}_{v}^{k}  \parallel \mathfrak{a}' \\
			\mathfrak{p}_{v} \mid 	\mathfrak{m}
			\mathfrak{m}'}}
	\mathfrak{p}_{v}^{k}
	\]
	with $(\mathfrak{n}, \mathfrak{m}\mathfrak{m}'\mathfrak{e}) = 1$. Since
	$
	\mathfrak{p}_{v} \mid
	\mathfrak{m}_{\scalebox{1.25}{$\scriptscriptstyle 0$}}
	 \mathfrak{m}_{\scalebox{1.25}{$\scriptscriptstyle 0$}}'
	 $
	 and
	 $\mathfrak{p}_{v}^{k}  \parallel \mathfrak{a}'$ implies that
	 $
	 k = \mathrm{ord}_{v}(\mathfrak{m}\mathfrak{m}') - 1
	 $, we have
	 \begin{equation} \label{eq: Prod-GS-1}
	\begin{split}
	G(\mathfrak{a}', \mathfrak{m}\mathfrak{e}^{\scalebox{.95}{$\scriptscriptstyle (1)$}})\overline{G(\mathfrak{a}', \mathfrak{m}'\mathfrak{e}^{\scalebox{.95}{$\scriptscriptstyle (1)$}})} =\,  &
		\overline{\chi}_{
		\mathfrak{m}_{\scalebox{.95}{$\scriptscriptstyle 0$}}}
	\chi_{
		\mathfrak{m}_{\scalebox{.95}{$\scriptscriptstyle 0$}}'}
	\Bigg(\mathfrak{n}
	\prod_{\substack{\mathfrak{p}_{v}^{k}  \parallel \mathfrak{a}' \\
			\mathfrak{p}_{v} \mid 	\mathfrak{m}\mathfrak{m}'	\\
			\mathfrak{p}_{v} \nmid
			\mathfrak{m}_{\scalebox{.95}{$\scriptscriptstyle 0$}}
			\mathfrak{m}_{\scalebox{.95}{$\scriptscriptstyle 0$}}'}} \mathfrak{p}_{v}^{k}\Bigg)
	\overline{\chi}_{
		\mathfrak{m}_{\scalebox{.95}{$\scriptscriptstyle 0$}}
		\mathfrak{e}_{\scalebox{.95}{$\scriptscriptstyle 0$}}^{\scalebox{.95}{$\scriptscriptstyle (1)$}}}\Big(
	\frac{\mathfrak{m}_{\scalebox{1.1}{$\scriptscriptstyle 0$}}'}
	{\mathfrak{m}_{\scalebox{1.1}{$\scriptscriptstyle 1$}}'}
	\Big)
	\chi_{
		\mathfrak{m}_{\scalebox{.95}{$\scriptscriptstyle 0$}}'
		\mathfrak{e}_{\scalebox{.95}{$\scriptscriptstyle 0$}}^{\scalebox{.95}{$\scriptscriptstyle (1)$}}}\Big(
	\frac{\mathfrak{m}_{\scalebox{1.1}{$\scriptscriptstyle 0$}}}
	{\mathfrak{m}_{\scalebox{1.1}{$\scriptscriptstyle 1$}}}
	\Big)\\
	& \cdot
	G(\chi_{
		\mathfrak{m}_{\scalebox{.95}{$\scriptscriptstyle 0$}}
		\mathfrak{e}_{\scalebox{.95}{$\scriptscriptstyle 0$}}^{\scalebox{.95}{$\scriptscriptstyle (1)$}}})
\overline{G(\chi_{
		\mathfrak{m}_{\scalebox{.95}{$\scriptscriptstyle 0$}}'
		\mathfrak{e}_{\scalebox{.95}{$\scriptscriptstyle 0$}}^{\scalebox{.95}{$\scriptscriptstyle (1)$}}})}
	G_{0}(\mathfrak{a}', \mathfrak{m}\mathfrak{e}^{\scalebox{.95}{$\scriptscriptstyle (1)$}})G_{0}(\mathfrak{a}', \mathfrak{m}'\mathfrak{e}^{\scalebox{.95}{$\scriptscriptstyle (1)$}}).
\end{split}
\end{equation}
\end{para}

\begin{para}
We use \cref{lem-Hecke-prod-G-sum-comp} to isolate the remaining part
	$
	\mathfrak{e}^{\scalebox{.95}{$\scriptscriptstyle (1)$}}
	$
	of $\mathfrak{e}$ from the product of Gauss sums.


\begin{lem}\label{prod-gauss-sums-final} With notation as above, we have
	\begin{equation*}
	\begin{split}
		G(\mathfrak{a}, \mathfrak{b})\overline{G(\mathfrak{a}, \mathfrak{b}')} & = \g([\bi], [\bi'])\cdot
			\overline{\chi}_{
				\mathfrak{m}_{\scalebox{.95}{$\scriptscriptstyle 0$}}}
			\chi_{
				\mathfrak{m}_{\scalebox{.95}{$\scriptscriptstyle 0$}}'}
			\Bigg(\mathfrak{n}
			\prod_{\substack{\mathfrak{p}_{v}^{k}  \parallel \mathfrak{a}' \\
					\mathfrak{p}_{v} \mid 	\mathfrak{m}\mathfrak{m}'	\\
					\mathfrak{p}_{v} \nmid
					\mathfrak{m}_{\scalebox{.95}{$\scriptscriptstyle 0$}}
					\mathfrak{m}_{\scalebox{.95}{$\scriptscriptstyle 0$}}'}} \mathfrak{p}_{v}^{k}\Bigg)
			\cdot \chi_{\mathfrak{m}}
			\overline{\chi}_{\mathfrak{m}'}(\mathfrak{e}_{\scalebox{1.}{$\scriptscriptstyle 0$}}^{\scalebox{1.}{$\scriptscriptstyle (1)$}}
			\mathfrak{e}_{\scalebox{1.}{$\scriptscriptstyle 1$}}^{\scalebox{1.}{$\scriptscriptstyle (1)$}})
			\\
			& \cdot
			G(\chi_{\mathfrak{m}_{\scalebox{.95}{$\scriptscriptstyle 0$}}}\overline{\chi}_{\mathfrak{m}_{\scalebox{.95}{$\scriptscriptstyle 0$}}'})
			G_{0}(\mathfrak{a}', \mathfrak{m}\mathfrak{e}^{\scalebox{.95}{$\scriptscriptstyle (1)$}})G_{0}(\mathfrak{a}', \mathfrak{m}'\mathfrak{e}^{\scalebox{.95}{$\scriptscriptstyle (1)$}})\\
		&
		\cdot \prod_{\substack{\mathfrak{p}_{v}^{l} \parallel \mathfrak{e}^{\scalebox{.85}{$\scriptscriptstyle (2)$}}\\
				\mathfrak{p}_{v}^{k}  \parallel \mathfrak{a} \\ k \ge l - 1}}
		\overline{\chi}_{\mathfrak{m}}\chi_{\mathfrak{m'}}(\mathfrak{p}_{v}^{k})
		G_{0}(\mathfrak{p}_{v}^{k}, \mathfrak{p}_{v}^{l})^2 \,
		\cdot \,  \Eu{N}\mathfrak{e}^{\scalebox{.95}{$\scriptscriptstyle (3)$}}
		\prod_{\mathfrak{p}_{v} \mid \mathfrak{e}^{\scalebox{.85}{$\scriptscriptstyle (3)$}}}
		(1 - q_{v}^{-1}),
	\end{split}
\end{equation*}
where the factor $\g([\bi], [\bi'])$ is on the unit circle and it is given in terms of the factor defined in~\eqref{eq: gamma psi} by
\[
\g([\bi], [\bi'])=\gamma(\psi_{[\bi]})\ov{\gamma(\psi_{[\bi']})} \ov{\gamma(\psi_{[\bi \bi'^{-1}]})}\psi_{[\bi']}(-1).
\]

\end{lem}

\begin{proof} Applying Lemma~\ref{lem-Hecke-prod-G-sum-comp} twice and \eqref{eq: Prod-GS-1},
\begin{align*}
		G(\mathfrak{a}', \mathfrak{m}\mathfrak{e}^{\scalebox{.95}{$\scriptscriptstyle (1)$}})\overline{G(\mathfrak{a}', \mathfrak{m}'\mathfrak{e}^{\scalebox{.95}{$\scriptscriptstyle (1)$}})} =\,  & 
		 \tau([\mathfrak{m}_{\scalebox{1.1}{$\scriptscriptstyle 0$}}], [\mathfrak{e}_{\scalebox{1.}{$\scriptscriptstyle 0$}}^{\scalebox{1.}{$\scriptscriptstyle (1)$}}])
		\overline{\tau([\mathfrak{m}_{\scalebox{1.1}{$\scriptscriptstyle 0$}}'] , [\mathfrak{e}_{\scalebox{1.}{$\scriptscriptstyle 0$}}^{\scalebox{1.}{$\scriptscriptstyle (1)$}}])}
		\overline{\tau([\mathfrak{m}_{\scalebox{1.1}{$\scriptscriptstyle 0$}}], [\mathfrak{m}_{\scalebox{1.1}{$\scriptscriptstyle 0$}}']^{\scalebox{1.1}{$\scriptscriptstyle -1$}})}
		\overline{\chi}_{
			\mathfrak{m}_{\scalebox{.95}{$\scriptscriptstyle 0$}}}
		\chi_{
			\mathfrak{m}_{\scalebox{.95}{$\scriptscriptstyle 0$}}'}
		\Bigg(\mathfrak{n}
		\prod_{\substack{\mathfrak{p}_{v}^{k}  \parallel \mathfrak{a}' \\
				\mathfrak{p}_{v} \mid 	\mathfrak{m}\mathfrak{m}'	\\
				\mathfrak{p}_{v} \nmid
				\mathfrak{m}_{\scalebox{.95}{$\scriptscriptstyle 0$}}
				\mathfrak{m}_{\scalebox{.95}{$\scriptscriptstyle 0$}}'}} \mathfrak{p}_{v}^{k}\Bigg)\\
		&\cdot \overline{\chi}_{\mathfrak{m}_{\scalebox{.95}{$\scriptscriptstyle 0$}}}(\mathfrak{m}_{\scalebox{1.1}{$\scriptscriptstyle 0$}}')
		\chi_{\mathfrak{m}_{\scalebox{.95}{$\scriptscriptstyle 0$}}'}(\mathfrak{m}_{\scalebox{1.1}{$\scriptscriptstyle 0$}})\chi_{\mathfrak{e}_{\scalebox{.95}{$\scriptscriptstyle 0$}}^{\scalebox{.95}{$\scriptscriptstyle (1)$}}}(\mathfrak{m}_{\scalebox{1.1}{$\scriptscriptstyle 0$}})
		\overline{\chi}_{\mathfrak{e}_{\scalebox{.95}{$\scriptscriptstyle 0$}}^{\scalebox{.95}{$\scriptscriptstyle (1)$}}}(\mathfrak{m}_{\scalebox{1.1}{$\scriptscriptstyle 0$}}')
		\chi_{\mathfrak{m}_{\scalebox{.95}{$\scriptscriptstyle 0$}}}
		\overline{\chi}_{
			\mathfrak{m}_{\scalebox{.95}{$\scriptscriptstyle 0$}}'}(\mathfrak{e}_{\scalebox{1.}{$\scriptscriptstyle 1$}}^{\scalebox{1.}{$\scriptscriptstyle (1)$}})
		\\
		& \cdot
		G(\chi_{\mathfrak{m}_{\scalebox{.95}{$\scriptscriptstyle 0$}}}\overline{\chi}_{\mathfrak{m}_{\scalebox{.95}{$\scriptscriptstyle 0$}}'})
		G_{0}(\mathfrak{a}', \mathfrak{m}\mathfrak{e}^{\scalebox{.95}{$\scriptscriptstyle (1)$}})G_{0}(\mathfrak{a}', \mathfrak{m}'\mathfrak{e}^{\scalebox{.95}{$\scriptscriptstyle (1)$}}).
	\end{align*}
Taking into account that $[\ai_0]=[\ai]$ for all $\ai\in I(S)$, by~\eqref{eq: tau formula} we have
\[ \tau([\mathfrak{m}_{\scalebox{1.1}{$\scriptscriptstyle 0$}}], [\mathfrak{e}_{\scalebox{1.}{$\scriptscriptstyle 0$}}^{\scalebox{1.}{$\scriptscriptstyle (1)$}}])
		\overline{\tau([\mathfrak{m}_{\scalebox{1.1}{$\scriptscriptstyle 0$}}'] , [\mathfrak{e}_{\scalebox{1.}{$\scriptscriptstyle 0$}}^{\scalebox{1.}{$\scriptscriptstyle (1)$}}])}
		\overline{\tau([\mathfrak{m}_{\scalebox{1.1}{$\scriptscriptstyle 0$}}], [\mathfrak{m}_{\scalebox{1.1}{$\scriptscriptstyle 0$}}']^{\scalebox{1.1}{$\scriptscriptstyle -1$}})}=
		\g(\psi_{[\bi]})\ov{\g(\psi_{[\bi']})} \ov{\g(\psi_{[\bi\bi'^{-1}]})}
		\cdot \g(\psi_{[\mi']})\g(\psi_{[\mi']^{-1}}),
\]
One easily checks that $\g(\psi_{[\mi']})\g(\psi_{[\mi'^{-1}]})=\psi_{[\mi']}(-1)$.
On the other hand, by the reciprocity law~\eqref{eq: reciprocity}
\[\overline{\chi}_{\mathfrak{m}_{\scalebox{.95}{$\scriptscriptstyle 0$}}}(\mathfrak{m}_{\scalebox{1.1}{$\scriptscriptstyle 0$}}')
		\chi_{\mathfrak{m}_{\scalebox{.95}{$\scriptscriptstyle 0$}}'}(\mathfrak{m}_{\scalebox{1.1}{$\scriptscriptstyle 0$}})\chi_{\mathfrak{e}_{\scalebox{.95}{$\scriptscriptstyle 0$}}^{\scalebox{.95}{$\scriptscriptstyle (1)$}}}(\mathfrak{m}_{\scalebox{1.1}{$\scriptscriptstyle 0$}})
		\overline{\chi}_{\mathfrak{e}_{\scalebox{.95}{$\scriptscriptstyle 0$}}^{\scalebox{.95}{$\scriptscriptstyle (1)$}}}(\mathfrak{m}_{\scalebox{1.1}{$\scriptscriptstyle 0$}}')=
		(m_{[\mi]}, m_{[\mi']} )_{\sss S} (m_{[\mi]}, m_{[\ei]} )_{\sss S}(m_{[\ei]}, m_{[\mi']} )_{\sss S}\cdot
		\chi_{\mathfrak{m}_{\scalebox{.95}{$\scriptscriptstyle 0$}}}
		\overline{\chi}_{
			\mathfrak{m}_{\scalebox{.95}{$\scriptscriptstyle 0$}}'}(\mathfrak{e}_{\scalebox{1.}{$\scriptscriptstyle 0$}}^{\scalebox{1.}{$\scriptscriptstyle (1)$}}).
\]
From the properties of the Hilbert symbol \cite[Prop. 1]{BB06} we have
\[\begin{aligned}
   (m_{[\mi]}, m_{[\mi']} )_{\sss S} (m_{[\mi]}, m_{[\ei]} )_{\sss S}(m_{[\ei]}, m_{[\mi']} )_{\sss S}
&=(m_{[\bi]},m_{[\bi']})_{\sss S} \ov{ (m_{[\ei]}, m_{[\ei]} )}_{\sss S}\\
&=(m_{[\bi]},m_{[\bi']})_{\sss S} \psi_{[\ei]}(-1).
  \end{aligned}
\]
Finally, using that $\psi_{[\ei]}\psi_{[\mi']}=\psi_{[\bi']}$ finishes the proof.
\end{proof}
\end{para}

\subsection{Another expression for \texorpdfstring{$Z_\aux$}~}
\begin{para}
Returning to the multiple Dirichlet series $Z_\aux$ defined in~\eqref{eq: Zaux}, we can now sum over $\ai$,
keeping $\bi,\bi'$ fixed. Let $\sigma_3$ be the involution on $\C^3$ given by
\[\sigma_3 \sb=(s_1+s_3-1/2, s_2+s_3-1/2, 1-s_3).\]
\begin{prop} \label{prop-Zaux}
Let $T=S\cup S_{\hi}$ and $\gi|\hi^{r-1}$, where $\hi\in I(S)$ is square-free. For $\psi_1$, $\psi_2$ Hecke characters of
order dividing $r$ and conductor dividing $\ci \hi$ and $\rho$ a Hecke character of
order dividing $r$ and conductor dividing $\ci$ we have
\[\begin{aligned}
Z_\aux^{T}(&\sigma_3\sb;\psi_1\chi_{\gi_1},\psi_2\chi_{\gi_1},\rho\chi_\gi)=|\gi_1|^{s_3-1\slash 2}
\sum_{\substack{\mi,\mi'\in I(T)\\ (\mi,\mi')=1}}\frac{\psi_1(\mi)\ov{\psi}_2(\mi')}{|\mi|^{s_1}|\mi'|^{s_2}} L^T(s_3,\chi_{\mi}\ov{\chi}_{\mi'}\rho\chi_\gi)\\
&\cdot\prod_{v\in T} \frac{L_v(s_3, \chi_v)}{L_v(1-s_3, \ov{\chi}_{v})}
\sum_{\ei\in I(T)}\frac{\psi_1(\ei)\ov{\psi}_2
(\ei)}{|\ei|^{s_1+s_2+s_3-1}} C([\mi \ei],[\mi' \ei],\gi,\rho,s_3)\\
&\cdot\prod_{\pp_v^l\| \ei^{(1)}}(\chi(\pp_v)^{l+1}q_v^{-s_3}-\chi(\pp_v)^{l}q_v^{-1})
\cdot\prod_{\pp_v^l\| \ei^{(2)}}(1-2q_v^{-1}+\chi(\pp_v)q_v^{-s_3-1})\cdot\prod_{\pp_v| \ei^{(3)}}(1-q_v^{-1}),
\end{aligned}
\]
 where $\chi=\chi_{\mi_0}\ov{\chi}_{\mi_0'}\rho\chi_\gi$, the decomposition $\mathfrak{e} =
	\mathfrak{e}^{\scalebox{.95}{$\scriptscriptstyle (1)$}} \mathfrak{e}^{\scalebox{.95}{$\scriptscriptstyle (2)$}} \mathfrak{e}^{\scalebox{.95}{$\scriptscriptstyle (3)$}}$ (depending on $\mi,\mi'$) was introduced at the beginning of \S\ref{sec:product-Gauss}, and
	\[\begin{aligned}
C(E,E',\gi,\rho,s_3)=&\ov{\rho}(\gi_1)\g(\psi_{\sss E}) \ov{\g}(\psi_{\sss E'})\g(\psi_{[{\sss \gi}]})
		\g(\rho\psi_{[{\sss E'E^{-1}\gi^{-1}}]})\psi_{\sss E'}(-1) \psi_{[{\sss E'E^{-1}}]}(m_{[\gi_1]})\\
	&\cdot|D_F\cond(\rho\psi_{[{\sss E'E^{-1}\gi^{-1}}]})|^{s_3-1/2},
	\end{aligned}
	\]
	with $D_F$ the discriminant of $F$.
\end{prop}
\begin{proof}
We start assuming that $\rho$ has order dividing $r$ and conductor dividing $\ci \hi$.
Using \cref{prod-gauss-sums-final},
\begin{equation*}
	\begin{split}
		\sum_{\mathfrak{a} \in I(T)}
		&\frac{\bar{\rho}(\mathfrak{a})
			G(\mathfrak{a}, \mathfrak{b})\overline{G(\mathfrak{a}, \mathfrak{b}')}}{|\mathfrak{a}|^{s_{\scalebox{.85}{$\scriptscriptstyle 3$}}}}=\g([\bi],[\bi'])
		\chi_{\mathfrak{m}}
		\overline{\chi}_{\mathfrak{m}'}(\mathfrak{e}^{\scalebox{1.}{$\scriptscriptstyle (1)$}}
		\mathfrak{e}_{\scalebox{1.}{$\scriptscriptstyle 1$}}^{\scalebox{1.}{$\scriptscriptstyle (1)$}})
		\bar{\rho}\bigg(\frac{\mathfrak{e}^{\scalebox{.95}{$\scriptscriptstyle (1)$}}}
		{\mathfrak{e}_{\scalebox{.95}{$\scriptscriptstyle 1$}}^{\scalebox{.95}{$\scriptscriptstyle (1)$}}}\bigg)
		\bigg|\frac{\mathfrak{e}^{\scalebox{.95}{$\scriptscriptstyle (1)$}}}
		{\mathfrak{e}_{\scalebox{.95}{$\scriptscriptstyle 1$}}^{\scalebox{.95}{$\scriptscriptstyle (1)$}}}\bigg|^{\!1 - s_{\scalebox{.85}{$\scriptscriptstyle 3$}}}
		G(\chi_{\mathfrak{m}_{\scalebox{.95}{$\scriptscriptstyle 0$}}}\overline{\chi}_{\mathfrak{m}_{\scalebox{.95}{$\scriptscriptstyle 0$}}'})\\
		& \cdot \prod_{\mathfrak{p}_{v}^{l} \parallel \mathfrak{e}^{\scalebox{.85}{$\scriptscriptstyle (2)$}}}
		\Bigg(\sum_{k \ge l - 1}\frac{\overline{\chi}_{\mathfrak{m}}
			\chi_{\mathfrak{m'}}\bar{\rho}(\mathfrak{p}_{v}^{k})
			G_{0}(\mathfrak{p}_{v}^{k}, \mathfrak{p}_{v}^{l})^2}
		{|\mathfrak{p}_{v}|^{k s_{\scalebox{.85}{$\scriptscriptstyle 3$}}}}\Bigg)
		\cdot|\mathfrak{e}^{\scalebox{.95}{$\scriptscriptstyle (3)$}}|^{1 - s_{\scalebox{.85}{$\scriptscriptstyle 3$}}}
		\prod_{\mathfrak{p}_{v} \mid \mathfrak{e}^{\scalebox{.85}{$\scriptscriptstyle (3)$}}}
		(1 - q_{v}^{-1})\\
		& \cdot
		\prod_{\mathfrak{p}_{v} \mid
				\mathfrak{m}_{\scalebox{.95}{$\scriptscriptstyle 0$}}
				\mathfrak{m}_{\scalebox{.95}{$\scriptscriptstyle 0$}}'}
		\bar{\rho}\big(
		\mathfrak{p}_{v}^{\mathrm{ord}_{v}(\mathfrak{m}\mathfrak{m}') - 1}\big)
		\big|\mathfrak{p}_{v}^{\mathrm{ord}_{v}(\mathfrak{m}\mathfrak{m}') - 1}\big|^{\!\frac{1}{2} - s_{\scalebox{.85}{$\scriptscriptstyle 3$}}}\\
		& \cdot
		\prod_{\substack{\mathfrak{p}_{v} \mid 	\mathfrak{m}\mathfrak{m}'	\\
				\mathfrak{p}_{v} \nmid
				\mathfrak{m}_{\scalebox{.95}{$\scriptscriptstyle 0$}}
				\mathfrak{m}_{\scalebox{.95}{$\scriptscriptstyle 0$}}'}}
		\sum_{k \ge
			\mathrm{ord}_{v}(\mathfrak{m}\mathfrak{m}') - 1}
		\frac{(\overline{\chi}_{\mathfrak{m}_{\scalebox{.95}{$\scriptscriptstyle 0$}}}
			\chi_{\mathfrak{m}_{\scalebox{.95}{$\scriptscriptstyle 0$}}'}
			\bar{\rho}(\mathfrak{p}_{v}))^{k}
			G_{0}\big(\mathfrak{p}_{v}^{k}, \mathfrak{p}_{v}^{\mathrm{ord}_{v}(\mathfrak{m}\mathfrak{m}')}\big)}{|\mathfrak{p}_{v}|^{k s_{\scalebox{.85}{$\scriptscriptstyle 3$}}}}
		\sum_{\substack{\mathfrak{n} \in I(T)\\ (\mathfrak{n}, \mathfrak{m}\mathfrak{m}'\mathfrak{e}) = 1}}
		\frac{\overline{\chi}_{\mathfrak{m}_{\scalebox{.95}{$\scriptscriptstyle 0$}}}
			\chi_{\mathfrak{m}_{\scalebox{.95}{$\scriptscriptstyle 0$}}'}\bar{\rho}(\mathfrak{n})}{|\mathfrak{n}|^{s_{\scalebox{.85}{$\scriptscriptstyle 3$}}}}.
	\end{split}
\end{equation*}
Using the formulas for the elementary Gauss sums $G_0(\cdot,\cdot)$ and summing the resulting geometric series,
we obtain
\begin{equation} \label{eq: interm-L-func}
	\begin{split}
		 \frac{\psi_{\scalebox{1.}{$\scriptscriptstyle 1$}} \rho(\mathfrak{b})\overline{\psi}_{\scalebox{1.}{$\scriptscriptstyle 2$}}\rho(\mathfrak{b}')}
		{|\mathfrak{b}|^{s_{\scalebox{.85}{$\scriptscriptstyle 1$} } } |\mathfrak{b}'|^{s_{\scalebox{.85}{$\scriptscriptstyle 2$}}}}
		&\sum_{\mathfrak{a} \in I(T)}
		\frac{\bar{\rho}(\mathfrak{a})
			G(\mathfrak{a}, \mathfrak{b})\overline{G(\mathfrak{a}, \mathfrak{b}')}}{|\mathfrak{a}|^{s_{\scalebox{.85}{$\scriptscriptstyle 3$}}}}
		=\g([\bi],[\bi'])
		\frac{\psi_{\sss 1}(\bi)\overline{\psi}_{\sss 2}(\bi')|\ei|^{s_{\scalebox{.85}{$\scriptscriptstyle 3$}}}}
		{|\bi|^{s_{\scalebox{.85}{$\scriptscriptstyle 1$}} + s_{\scalebox{.85}{$\scriptscriptstyle 3$}} - \frac{1}{2}}
		|\bi'|^{s_{\scalebox{.85}{$\scriptscriptstyle 2$}} + s_{\scalebox{.85}{$\scriptscriptstyle 3$}} - \frac{1}{2}}}\\
		&\cdot \prod_{\mathfrak{p}_{v}^{l} \parallel \mathfrak{e}^{\scalebox{.85}{$\scriptscriptstyle (1)$}}}
		\!\Big(\chi(\pp_{v})^{l+1} q_{v}^{s_{\scalebox{.85}{$\scriptscriptstyle 3$}}-1} -
		       \chi(\pp_{v})^{l} q_{v}^{-1}\Big)\\
		& \cdot \prod_{\mathfrak{p}_{v}^{l} \parallel \mathfrak{e}^{\scalebox{.85}{$\scriptscriptstyle (2)$}}}
		\big(\chi(\pp_{v})q_{v}^{s_{\scalebox{.85}{$\scriptscriptstyle 3$}} - 2} + 1 - 2 q_{v}^{-1}\big) \prod_{\mathfrak{p}_{v} \mid \mathfrak{e}^{\scalebox{.85}{$\scriptscriptstyle (3)$}}}
		(1 - q_{v}^{-1})\\
		& \cdot \,
		\prod_{\substack{\mathfrak{p}_{v} \mid 	\mathfrak{m}\mathfrak{m}'	\\
					\mathfrak{p}_{v} \nmid
					\mathfrak{m}_{\scalebox{.95}{$\scriptscriptstyle 0$}}
					\mathfrak{m}_{\scalebox{.95}{$\scriptscriptstyle 0$}}'}}
			\!\big(1 - \chi(\pp_{v})q_{v}^{s_{\scalebox{.85}{$\scriptscriptstyle 3$}} - 1}\big)
		\cdot \rho(\mathfrak{m}_{\scalebox{1.25}{$\scriptscriptstyle 1$}}\mathfrak{m}_{\scalebox{1.25}{$\scriptscriptstyle 1$}}')|\mathfrak{m}_{\scalebox{1.25}{$\scriptscriptstyle 1$}}\mathfrak{m}_{\scalebox{1.25}{$\scriptscriptstyle 1$}}'|^{s_{\scalebox{.85}{$\scriptscriptstyle 3$}} - \frac{1}{2}}
		G(\chi_{\mathfrak{m}_{\scalebox{.95}{$\scriptscriptstyle 0$}}}\overline{\chi}_{\mathfrak{m}_{\scalebox{.95}{$\scriptscriptstyle 0$}}'})
			\cdot L^{T}(s_{\scalebox{1.1}{$\scriptscriptstyle 3$}}, \ov{\chi}),
	\end{split}
\end{equation}
where we set $ \chi = \chi_{\mathfrak{m}_{\scalebox{.95}{$\scriptscriptstyle 0$}}}
\ov{\chi}_{\mathfrak{m}_{\scalebox{.95}{$\scriptscriptstyle 0$}}'}\rho$.

At this point, we make the substitutions $\rho\mapsto\rho \chi_\gi$ and $\psi_j\mapsto\psi_j \chi_{\gi_1}$,
so that $\chi$ becomes $\chi=\chi_{\mathfrak{m}_{\scalebox{.95}{$\scriptscriptstyle 0$}}}
\ov{\chi}_{\mathfrak{m}_{\scalebox{.95}{$\scriptscriptstyle 0$}}'}\rho\chi_\gi$.
We apply the functional equation of the Hecke $L$-function $L(s, \chi)$,
in the form
\[
L^{T}(s, \ov{\chi}) = L^{T}(1 - s, \chi)\epsilon(s, \ov{\chi}) \prod_{v \in T}
\frac{L_{v}(1 - s, \chi_{v})}{L_{v}(s,  \ov{\chi}_{v})}
\]
with $L_{v}(s, \chi_{v}) = (1 - \chi(\pp_v)q_{v}^{-s})^{-1}$
for $v \notin S_{\infty}$ and the epsilon factor
\[\epsilon(s, \ov{\chi})= G(\ov{\chi})|\mi_1\mi_1'\gi_1|^{1/2-s} \cdot |D_F\cond(\rho\psi_{[\mi'\mi^{-1}\gi^{-1}]}) |^{1/2-s}.
 \]
As in the proof of
\cref{lem-Hecke-prod-G-sum-comp}, we have
\[G(\chi_{\mathfrak{m}_{\scalebox{.95}{$\scriptscriptstyle 0$}}}\overline{\chi}_{\mathfrak{m}_{\scalebox{.95}{$\scriptscriptstyle 0$}}'})G(\overline{\chi}_{\mathfrak{m}_{\scalebox{.95}{$\scriptscriptstyle 0$}}}
\chi_{\mathfrak{m}_{\scalebox{.95}{$\scriptscriptstyle 0$}}'}\bar{\rho}\bar{\chi}_\gi) =\ov{G(\chi_\gi)}
\ov{\rho}\ov{\chi}_{\gi}(\mathfrak{m}_{\scalebox{1.25}{$\scriptscriptstyle 1$}}\mathfrak{m}_{\scalebox{1.25}{$\scriptscriptstyle 1$}}')
\ov{\chi}_{\mi_0}\chi_{\mi_0'}\ov{\rho}(\gi_1) C(\rho, [\mi'\mi^{-1}], [\gi])
\]
where
$C(\rho, E, G)=\g(\ov{\psi}_{E})\g(\rho\psi_{[EG^{-1}]})\g(\psi_G) \psi_{E}(-1).$

Note that $\g(\ov{\psi}_{[\mi'\mi^{-1}]}) \cdot\ov{\g}(\psi_{[\mi\mi'^{-1}]})  =\psi_{[\mi'\mi^{-1}]}(-1) $,
so one of the $\g$-factors above cancels the corresponding one from $\g([\bi],[\bi'])$.
We also have:
\[\prod_{\substack{\mathfrak{p}_{v} \mid 	\mathfrak{m}\mathfrak{m}'	\\
					\mathfrak{p}_{v} \nmid
					\mathfrak{m}_{\scalebox{.95}{$\scriptscriptstyle 0$}}
					\mathfrak{m}_{\scalebox{.95}{$\scriptscriptstyle 0$}}'}}
			\!\big(1 - \chi_{\mathfrak{m}_{\scalebox{.95}{$\scriptscriptstyle 0$}}}
			\overline{\chi}_{\mathfrak{m}_{\scalebox{.95}{$\scriptscriptstyle 0$}}'}
			\rho(\pp_{v})q_{v}^{s - 1}\big)
			L^{T}(1-s, \chi_{\mathfrak{m}_{\scalebox{.95}{$\scriptscriptstyle 0$}}}
		\ov{\chi}_{\mathfrak{m}_{\scalebox{.95}{$\scriptscriptstyle 0$}}'}\rho)=
		L^{T}(1-s, \chi_{\mathfrak{m}}
		\ov{\chi}_{\mathfrak{m}'}\rho)
\]
yielding the final formula.
\end{proof}
\end{para}

\subsection{Cleaning up}\label{sec:cleaning}
\begin{para}
Next we remove the factor $ C(-)$, and the product over the
finite places in $T$ from the formula in \cref{prop-Zaux}, by forming linear combinations of the original function $Z_\aux^T$.
Note that for  $E,E'\in \EE$,  the finite linear combination
\[Z_{\aux}^T(\sigma_3\sb;\psi_1,\psi_2, \rho;{E,E'}):=\frac{1}{|R_\ci|^2}\sum_{\vp_1,\vp_2\in\widehat{R}_\ci}
\ov{\vp}_1(E)\vp_2(E') Z_\aux^T(\sigma_3\sb;\psi_1\vp_1,\psi_2 \vp_2, \rho),
\]
is the subseries in \cref{prop-Zaux} consisting of those $\mi,\mi',\ei$ with $[\mi \ei]=E$ and $[\mi'\ei]=E'$.
Recall that $\widehat{R}_\ci$ denotes the set of Hecke characters of order dividing~$r$ and conductor dividing~$\ci$.
Therefore the function $Z^{T,(1)}(-)= \sum_{E,E'\in\mathscr{E}} Z^{T,(1)}(-;{E,E'})$
with
\[Z^{T,(1)}(\sb;\psi_1,\psi_2, \rho\chi_\gi;{E,E'}):=
\frac{|\gi_1|^{1/2-s_3}}{C(E,E',\gi, \rho,s_3)} Z_\aux^T(\sigma_3\sb;\psi_1\chi_{\gi_1},\psi_2\chi_{\gi_1} , \rho\chi_{\gi};{E,E'})
\]
satisfies the same relation as $Z_\aux^T$ in \cref{prop-Zaux}, but without the factors
$C(-)$ and $|\gi_1|^{s_3-1/2}$. We choose the notation so that the functions $Z^{T,(i)}$ are ``getting closer'' to
the perfect MDS $Z^{T}$  described in \S\ref{par:ZMDS} as $i=1,2,3$.
\end{para}
\begin{para}
For the next proposition, we need to extend the definition of the $r$-th order characters $\chi_\mathfrak{t}$ to
ideals $\ti$ supported on $S_f$. For each $v\in S_f$ we fix a generator $p_v$
of the $\OO_S$-ideal $\pp_v\OO_S$ (that is, $\ord_v(p_v)=1$ and $\ord_u(p_v)=0$ for $u\notin S$),
and for $\ti=\prod_{v\in S_f} \pp_v^{n_v}$ $r$-th power-free, we fix the generator
$t=\prod p_v^{n_v}$ of $\ti\OO_S$. For $\mi$ coprime to the places in $S$ we define
$$
\chi_\ti(\mi)=\bigg(\frac {t}\mi\bigg),
$$
a Hecke character of conductor  dividing $\ci$.
Note that if $\ti$ is coprime to the conductor of $\chi_{\mi}$, then by~\eqref{eq: rec} we have
\be\label{eq:chim0}
\chi_{\mi}(\ti)=\chi_\ti(\mi)\psi_{[\mi]}(t).
\ee
By abuse of notation, in the next proposition we use this relation for $\ti\in I(S\cup S_\mi)$ as well,
by defining $t:=m_{[\ti]}$ in this case, according to~\eqref{eq: reciprocity}.

\end{para}
\begin{para}
Recall that all our Hecke characters $\chi$ have trivial infinity type, so $L_v(s,\chi_v)=\zeta_{F,v}(s)$
for $v\in S_\infty$.
\begin{prop}\label{prop: Zaux2}
Let $T$ be a finite set of places containing $S$, and let
$\psi_1$, $\psi_2$, $\rho$ be Hecke characters of
order dividing $r$ and conductor dividing $\ci \prod_{v\in T\setminus S} \pp_v$.
Fix two classes $E,E'\in \EE$ and let
$$\Ci=\Ci_{[EE'^{-1}],\rho}:=\prod_{\substack{v\in T_\fin \\ \pp_v\nmid \cond(\rho \psi_{[{\sss EE'^{-1}}]})}} \pp_v .$$
Defining
\[\begin{aligned}
Z^{T,(2)}(\sb; & \psi_1,\psi_2, \rho ;{E,E'})=\prod_{v\in S_{\Ci}}(1-q_v^{r(s_3-1)})^{-1}
\sum_{\substack{\ti\mid \Ci^{r-1}\\ \ti'\mid \Ci }}
\psi_{[{\sss EE'^{-1}}]}(t'/t)\rho(\ti'/\ti)\\
&\cdot|\ti|^{s_3-1}\mu(\ti')|\ti'|^{-s_3}
Z^{T,(1)}(\sb;\psi_1\ov{\chi}_{\ti}\chi_{\ti'},\psi_2\ov{\chi}_{\ti}\chi_{\ti'}, \rho ;E,E'),
\end{aligned}
\]
we have that the function $Z^{T,(2)}(-):=\sum_{E,E'\in\mathscr{E}} Z^{T,(2)}(-;E,E')$
is given by:
\[\begin{aligned}
Z^{T,(2)}(\sb;\psi_1,\psi_2, \rho )&=
\prod_{v\in S_\infty}\frac{\zeta_{F,v}(s_3)}{\zeta_{F,v}(1-s_3)}
\sum_{\substack{\mi,\mi'\in I(T)\\(\mi,\mi')=1}}\frac{\psi_1(\mi)\ov{\psi}_2(\mi')}{|\mi|^{s_1}|\mi'|^{s_2}} L^T(s_3,\chi_{\mi}\ov{\chi}_{\mi'}\rho )\\
&\cdot\sum_{\ei\in I(T)}\frac{\psi_1(\ei)\ov{\psi}_2(\ei)}{|\ei|^{s_1+s_2+s_3-1}}
\cdot\prod_{\pp_v^l\| \ei^{(1)}}(\chi(\pp_v)^{l+1}q_v^{-s_3}-\chi(\pp_v)^{l}q_v^{-1})\\
&\cdot\prod_{\pp_v^l\| \ei^{(2)}}(1-2q_v^{-1}+\chi(\pp_v)q_v^{-s_3-1})\cdot\prod_{\pp_v| \ei^{(3)}}(1-q_v^{-1}),
\end{aligned}
\]
where $\chi=\chi_{\mi_0}\ov{\chi}_{\mi'_0}\rho $.
\end{prop}

\begin{proof}
Let $\bi=\mi \ei$,
 $\bi'=\mi' \ei$ with $[\bi]=E$, $[\bi']=E'$, and $\chi=\chi_{\mi_0}\ov{\chi}_{\mi'_0}\rho $.
 For $v\in T_\fin$, we have  $\chi(\pp_v)\ne 0$ if and only if $v\in S_\Ci$, and we obtain
 \be\label{eq:prod_fin}
\prod_{v\in S_\Ci}(1-q_v^{r(s_3-1)})^{-1}\frac{L_v(s_3, \chi)}{L_v(1-s_3, \ov{\chi})}=
 \bigg[\sum_{\substack{\ti|\Ci^{r-1}\\ \ti'|\Ci}}
 \ov{\chi}(\ti)|\ti|^{s_3-1} \chi(\ti') \mu(\ti') |\ti'|^{-s_3}\bigg]^{-1}.
 \ee
 By~\eqref{eq:chim0} we have $\chi(\ti)=\rho (\ti)\chi_{\ti} (\bi\bi'^{-1})\psi_{[{\sss EE'^{-1}}]} (t) $
 for any ideal $\ti$ with $S_\ti\subset S_\Ci$, where $t$ is a generator of $\ti\OO_S$ chosen as in~\eqref{eq:chim0}.
Exchanging the order of summation in the formula for $Z^{T,(2)}(-;E,E')$ and using the identity above finishes the proof.
\end{proof}
\begin{rem} A similar argument shows that $Z^{T,(1)}$ can be expressed in terms of $Z^{T,(2)}$ as follows:
 \[\begin{aligned}
Z^{T,(1)}(\sb; & \psi_1,\psi_2, \rho ;{E,E'})=\prod_{v\in S_{\Ci}}(1-q_v^{-rs_3})^{-1}
\sum_{\substack{\ti\mid \Ci^{r-1}\\ \ti'\mid \Ci }}
\psi_{[{\sss EE'^{-1}}]}(t/t')\rho(\ti/\ti')\\
&\cdot|\ti|^{-s_3}\mu(\ti')|\ti'|^{s_3-1}
Z^{T,(2)}(\sb;\psi_1\chi_{\ti}\ov{\chi}_{\ti'},\psi_2\chi_{\ti}\ov{\chi}_{\ti'}, \rho ;E,E').
\end{aligned}
\]
\end{rem}
\end{para}

\subsection{The perfect MDS}\label{sec:2ndmoment}
\begin{para}
As in the previous subsection, we let $T$ be a finite set of places containing $S$, and let $\psi_1$, $\psi_2$, $\rho$ be
Hecke characters unramified outside $T$ and of order dividing $r$.  Let
\[
Z^{T,(3)}(\sb;\psi_1,\psi_2, \rho)=L^{T}(s_1+s_2, \psi_1\ov{\psi}_2)
\prod_{v\in S_\infty}\frac{\zeta_{F,v}(1-s_3)}{\zeta_{F,v}(s_3)}\cdot Z^{T,(2)}(\sb;\psi_1,\psi_2, \rho)
\]
where $Z^{T,(2)}$ is defined in \cref{prop: Zaux2}.
Expressing $L^S(s_3, \chi)=\sum_{\ai\in I(S)} \chi(\ai)|\ai|^{-s_3}$
in the formula there and collecting the terms involving $s_3$ gives
\be\label{eq:Z1}
Z^{T,(3)}(\sb;\psi_1,\psi_2, \rho)=L^{T}(s_1+s_2, \psi_1\ov{\psi}_2)
\cdot\sum_{\ai, \bi,\bi'\in I(T)} \frac{\psi_1(\bi)\ov{\psi}_2(\bi')}{|\bi|^{s_1}|\bi'|^{s_2}}
\sum_{\gi\in A_\ei}
\frac{\chi_\mi\ov{\chi}_{\mi'}\rho(\ai \gi_0)}{|\ai \gi|^{s_3}}\alpha_{\gi},
\ee
where in the sum we write as before $\bi=\ei\mi$, $\bi'=\ei\mi'$ with $(\mi, \mi')=1$.
Here $A_\ei$ is a set of ideals supported on the primes in $\ei$, and $\alpha_\gi$ for $\gi\in A_\ei$ are integers
defined as follows in terms of the decomposition $\ei=\ei^{\sss (1)}\ei^{\sss (2)}\ei^{\sss (3)}$ from \S\ref{sec:product-Gauss}.
Letting $\ei=\prod_v \pp_v^{l_v}$, define
\[A_\ei=\bigg\{\ei^{\sss (3)}\prod_{\pp_v|\ei^{\sss (1)}\ei^{\sss (2)}} \pp_v^{a_v} \mid a_v \in \{l_v, l_{v}+1\}
\bigg\}
\]
and for $\gi\in A_\ei$ given as above in terms of $a_v$ we set
\[
\alpha_\gi=|\ei^{\sss (3)}|\prod_{\pp_v|\ei^{\sss (1)}} c_{a_v}  \prod_{\pp_v|\ei^{(2)}} d_{a_v}\prod_{\pp_v|\ei^{\sss (3)}} (1-q_v^{-1})
\]
with
$c_{a}=\begin{cases}
          -q_v^{l_v-1} &  \text{ if } a=l_v \\
          q_v^{l_v} &  \text{ if } a=l_v+1
         \end{cases}$
and
$d_{a}=\begin{cases}
          q_v^{l_v}-2q_v^{l_v-1} &  \text{ if } a=l_v \\
          q_v^{l_v-1} &  \text{ if } a=l_v+1
         \end{cases}$.
\end{para}
\begin{para}\label{par:ZMDSdef}
As before, denote by $Z^{T,(3)}(-;{E,E'})$ the subseries in~\eqref{eq:Z1} such that $[\bi]=E$,  $[\bi']=E'$.
For fixed $E,E',G\in \EE$, define
 \[
Z^T(\sb;\psi_1,\psi_2, \rho;{E,E',G})=\frac{1}{|R_\ci|} \sum_{\eta\in \widehat{R}_\ci}
	\ov{\eta}(G)\ov{\psi}_{[{\sss EE'^{-1}}]}(m_{\sss G}) Z^{T,(3)}(\sb;\psi_1,\psi_2, \rho\eta;E,E'),
\]
and let $Z^T(-):=\sum_{E,E',G\in \EE} Z^T(-;{E,E',G})$.
\end{para}
\begin{prop}[Perfect MDS]\label{prop:Z0}
Let $T$ be a finite set of places containing $S$, and let $\psi_1$, $\psi_2$, $\rho$ be
Hecke characters unramified outside $T$ and of order dividing $r$.
With $Z^T(\sb;\psi_1,\psi_2, \rho)$ defined in~\S\ref{par:ZMDSdef} we have:
\[\begin{aligned}
Z^T(\sb;\psi_1,\psi_2, \rho)&=L^{T}(s_1+s_2, \psi_1\ov{\psi}_2)
\sum_{(\mi,\mi')=1} \frac{\psi_1(\mi)\ov{\psi}_2(\mi')}{|\mi|^{s_1}|\mi'|^{s_2}}\\
&\qquad\cdot\sum_{\ei\in I(T)} \frac{\psi_1(\ei)\ov{\psi}_2(\ei)}{|\ei|^{s_1+s_2}}
\sum_{\substack{(\ai,\mi\mi')=1\\  \gi\in A_\ei}}
\frac{\chi_{(\ai\gi)_0}(\mi\mi'^{-1}) \rho(\ai \gi_0)}{|\ai \gi|^{s_3}}\alpha_{\gi}\\
&=\sum_{\ai\in I(T)} \frac{L^{T}(s_1, \psi_1\chi_{\ai_0}) L^{T}(s_2, \ov{\psi}_2\ov{\chi}_{\ai_0})}{|\ai|^{s_3}}
\rho(\ai)P_{\ai}(s_1,s_2; \psi_1, \psi_2),
\end{aligned}
\]
where the Dirichlet polynomials $P_{\ai}$ are defined by
\[\begin{aligned}
	P_\ai(s_1,s_2;\psi_1,\psi_2)=\prod_{\substack{\pp^{k}\| \ai \\ \pp\mid\ai_{0} } }
	P_{k}(\psi_1(\pp)|\pp|^{-s_1},& \ov{\psi}_2(\pp)|\pp|^{-s_2}; |\pp|)\\
	&\cdot\prod_{\substack{\pp^{k}\| \ai \\ \pp\nmid\ai_{0} } }
	P_{k}(\psi_1\chi_{\ai_0}(\pp)|\pp|^{-s_1}, \ov{\psi}_2\ov{\chi}_{\ai_0}(\pp)|\pp|^{-s_2}; |\pp|).
	\end{aligned}
	\]
	Here $P_k(\xx;q)$ are the following polynomials in $\xx=(x_1,x_2)$:
\[	P_{k}(\mathbf{x}; q) =
\begin{cases} \frac{1  - (q x_1 x_2)^{k}}{1 - q x_1 x_2}& \text{ if } k \not \equiv 0 \pmod{r} \\[5pt]
 \frac{(1 - x_1) (1 - x_2)(1  - (q x_1 x_2)^{k})}{1 - q x_1 x_2} + (q x_1 x_2)^{k}(1 - q^{-1}) & \text{ if }0\ne  k \equiv 0 \pmod{r}
\end{cases}
\]
and $P_{0}(\mathbf{x}; q) = 1$.
\end{prop}
\begin{proof}
The first equality follows from~\eqref{eq:Z1} by the orthogonality relation for characters and the reciprocity law
$\chi_\mi(\ff)=\chi_{\ff}(\mi) \psi_{[\mi]}(m_{[\ff]})$, where $\ff=\ai\gi_0$ is coprime to $\mi\mi'$.

To prove the second, we fix  $\ai'=\ai\gi$ in the first formula
and compute the contribution to the term $|\ai'|^{-s_3}$ from the
primes $\pp|\ai'$ that also divide $\ei\mi$ or $\ei\mi'$.
This contribution is multiplicative in $\pp$, so
we fix $\pp=\pp_v$ with $\pp^\dd\| \ai'$, and denote $q = |\mathfrak{p}|$.

\textbf{Case 1:} $\dd\not \equiv 0 \pmod r$ (i.e.,
	$
	\mathfrak{p} \mid \mathfrak{a}'_{\scalebox{.95}{$\scriptscriptstyle 0$}}
	$). Letting $
	x_1 = \psi_{\scalebox{1.}{$\scriptscriptstyle 1$}}(\mathfrak{p})
	|\mathfrak{p}|^{- s_1}
	$,
	$x_2 = \overline{\psi}_{\scalebox{1.}{$\scriptscriptstyle 2$}}(\mathfrak{p})
	|\pp|^{- s_2}$ this contribution comes from the terms for which $\ei=(1)$ or $\pp\mid\ei^{\sss (1)}\ei^{\sss (2)} $ and it equals
	\begin{equation*}
		\begin{split}
			1  + \sum_{\substack{\alpha = 1 \\ r \nmid \alpha}}^{\delta - 1}
			q^{\alpha} x_1^{\alpha} x_2^{\alpha}
			&- \sum_{\substack{\alpha = 1 \\ r \nmid \alpha}}^{\delta} q^{\alpha - 1} x_1^{\alpha} x_2^{\alpha}
			 + \sum_{\substack{\alpha = 1 \\ r \mid \alpha}}^{\delta - 1}   q^{\aa - 1} x_1^{\aa} x_2^{\aa}
			+ \sum_{\substack{\alpha = 1 \\ r \mid \alpha}}^{\delta - 1}  (q^{\aa}-2 q^{\aa-1} ) x_1^{\aa} x_2^{\aa}= \\
			& = \sum_{\alpha = 0}^{\delta - 1} q^{\alpha} x_1^{\alpha} x_2^{\alpha}
			- \sum_{\alpha = 1}^{\delta} q^{\alpha - 1} x_1^{\alpha} x_2^{\alpha}\\
			& = (1 - x_1 x_2) \sum_{\alpha = 0}^{\delta - 1} q^{\alpha} x_1^{\alpha} x_2^{\alpha}=(1 - x_1 x_2)P_{\dd} (x_1,x_2;q).
		\end{split}
	\end{equation*}

	\textbf{Case 2:} $\dd \equiv 0 \pmod r$ (i.e.,
	$
	\mathfrak{p} \nmid \mathfrak{a}'_{\scalebox{.95}{$\scriptscriptstyle 0$}}
	$).
	There are two separate contributions in this case. The first corresponds, as in the previous case, to terms $\ei$ with $\pp\mid\ei^{\sss (1)}\ei^{\sss (2)} $, and is given by
	\[
	(1 - q^{-1})q^{\delta} x_1^{\delta} x_2^{\delta}
	+ (1 - x_1 x_2) \sum_{\alpha = 0}^{\delta - 1} q^{\alpha} x_1^{\alpha} x_2^{\alpha}.
	\]
	The second contribution corresponds to coprime pairs $(\mi, \mi')$ in the sum for which $\pp \mid \mi \mi'$
	and $\pp\mid \ei^{\sss (3)}$. Letting
	$x_1 = \psi_{\scalebox{1.}{$\scriptscriptstyle 1$}}(\mathfrak{p})
	\chi_{\ai'_{0}}(\pp) |\mathfrak{p}|^{- s_1}$,
	$x_2 = \overline{\psi}_{\scalebox{1.}{$\scriptscriptstyle 2$}}(\mathfrak{p})
	\overline{\chi}_{\ai'_{0}}(\pp)|\pp|^{- s_2}$ we find that this $\mathfrak{p}$-contribution is
	\[
	(1 - q^{-1})q^{\delta} x_1^{\delta} x_2^{\delta}\Big(\frac{x_1}{1 - x_1} + \frac{x_2}{1 - x_2}\Big).
	\]
	Thus, the total contribution in this case is
	\[
	(1 - x_1 x_2)\sum_{\alpha = 0}^{\delta - 1} q^{\alpha} x_1^{\alpha} x_2^{\alpha} \, +\,
	(1 - q^{-1})q^{\delta} x_1^{\delta} x_2^{\delta}\cdot \frac{(1 - x_1 x_2)}{(1 - x_1)(1 - x_2)}=
	\frac{(1 - x_1 x_2)P_\dd(x_1,x_2;q)}{(1 - x_1)(1 - x_2)}.
	\]

	In both cases, these contributions depend only on $\ai'$ and are independent of the parts of $\mi, \mi'$ coprime to $\ai'$.
	Relabeling $\ai'$ as $\ai$ and switching the order of summation we obtain:
\[\begin{aligned}
Z^T(\sb;\psi_1,\psi_2, \rho)&=L^{T}(s_1+s_2, \psi_1\ov{\psi}_2)\sum_{\ai \in I(T)} P_{\ai}(s_1,s_2; \psi_1,\psi_2) \rho(\ai)
\prod_{\pp_v\mid \ai}L_v(s_1+s_2, \psi_1 \ov{\psi}_2)^{-1}\\
&\cdot
\prod_{\substack{\pp_v\mid \ai\\ \pp_v\nmid \ai_0}}
L_v(s_1, \psi_1\chi_{\ai_0}) L_v(s_2, \ov{\psi}_2\ov{\chi}_{\ai_0})
\sum_{\substack{\mi,\mi'\in I(S\cup S_\ai)\\(\mi,\mi')=1 }} \frac{\psi_1\chi_{\ai_0}(\mi)\ov{\psi}_2\ov{\chi}_{\ai_0}(\mi')}{|\mi|^{s_1}|\mi'|^{s_2}}.
\end{aligned}
\]
The inner sum equals $L^{S\cup S_\ai}(s_1, \psi_1\chi_{\ai_0}) L^{S\cup S_\ai}(s_2, \ov{\psi}_2\ov{\chi}_{\ai_0})
L^{S\cup S_\ai}(s_1+s_2, \psi_1 \ov{\psi}_2)^{-1}$, proving the second equality.
\end{proof}

\subsection{Another normalization}\label{sec:wZ}
\begin{para}
In view of the functional equation and analytic properties satisfied by Kubota's Dirichlet series
(see \S\ref{sec:Kubota}), it is convenient to renormalize $Z^T_\aux$ and $Z^T$ as follows:
	\[ \wZ_\aux^T(\sb;\psi_1,\psi_2,\rho)=
\zeta_{F}^T (rs_1 - r/2  + 1)\zeta_{F}^T (rs_2 - r/2  + 1) Z_\aux^T(\sb;\psi_1,\psi_2,\rho),
\]
	\[\begin{aligned}
	   \wZ^T(\sb;\psi_1,\psi_2,\rho)&=\zeta_F^T(rs_1+rs_3+1-r) \zeta_F^T(rs_2+rs_3+1-r)  Z^T(\sb;\psi_1,\psi_2,\rho)\\
	   &=\sum_{\ai\in I(T)} \frac{L^{T}(s_1, \psi_1\chi_{\ai_0}) L^{T}(s_2, \ov{\psi}_2\ov{\chi}_{\ai_0})}{|\ai|^{s_3}} \rho(\ai)Q_{\ai}(s_1,s_2; \psi_1, \psi_2),
	   \end{aligned}
	\]
where the Dirichlet polynomials $Q_\ai$ are defined by
\[
Q_\ai(s_1,s_2;\psi_1,\psi_2)=\sum_{\ai=\ni \ni_1^r\ni_2^r} P_\ni(s_1,s_2;\psi_1,\psi_2)|\ni_1|^{r-1-rs_1} |\ni_2|^{r-1-rs_2} .
\]
The polynomials $Q_\ai$ decompose in the same way as the polynomials $P_\ai$ in  \cref{prop:Z0},
in terms of polynomials $Q_k(\xx;q)$, whose generating series satisfies
\[
 \sum_{k\ge 0} Q_k(\xx;q)y^k = ((1 - q^{r - 1} x_1^r y^r)(1 - q^{r - 1} x_2^r y^r))^{-1} \sum_{k\ge 0} P_k(\xx;q)y^k.
\]
\begin{rem}
The polynomials $P_\ai$ and $Q_\ai$ satisfy the functional equation:
\be\label{eq:Qfunc_eq}
Q_\ai(s_2, s_1; \psi_2,\psi_1)=Q_\ai(1-s_1, 1-s_2; \psi_1,\psi_2)
|\ai\slash {\ai_1}|^{1-s_1-s_2}\ov{\psi}_1\psi_2(\ai\slash \ai_1).
\ee
\end{rem}
\end{para}

\subsection{Functional equations}\label{sec:fun_eqs}
\begin{para}
Putting together the results of this section we obtain the following functional equations relating
the functions $\wZ^T$ and $\wZ_\aux^T$ defined in \S\ref{sec:wZ}.

\begin{prop}\label{prop: Z funeq}
Let $T=S\cup S_\hi$ for square-free $\hi\in I(S)$ and let $\ff$, $\ff'$, $\gi$ be divisors of $\hi^{r-1}$.
For $\psi_1$, $\psi_2$, $\rho$ Hecke characters of
order dividing $r$ and conductor dividing $\ci$, we have the functional equation
\[
\begin{aligned}
 \wZ^{T}(\sb;\psi_1\chi_{\ff}, & \psi_2\chi_{\ff'}, \rho\chi_{\gi})=
 |\gi_1|^{1\slash 2 -s_3} L^{T}(s_1+s_2,\psi_1\ov{\psi}_2 \chi_{\ff}\ov{\chi}_{\ff'})\prod_{v\in S_\infty}\frac{\zeta_{F,v}(1-s_3)}{\zeta_{F,v}(s_3)}\\
 &\cdot \sum_{\substack{\ti\mid\hi^r\\ (\ti,\gi)=1} }\sum_{\vp,\eta\in\widehat{R}_\ci}
C(\vp,\eta,\ti;s_3)
\wZ_\aux^T(\sigma_3\sb;\psi_1\vp\chi_{\ff\ti\gi_1}, \psi_2\vp\chi_{\ff'\ti\gi_1},\eta\chi_\gi)
\end{aligned}
\]
where the coefficients satisfy $C(-; s_3)\ll |\hi|^\eps $ for $\Re(s_3)=-\eps$ for $\eps>0$.
\end{prop}
From the results of this section one similarly obtains a functional equation in the opposite direction.
\begin{prop}\label{prop: Z funeq2}
Let $T=S\cup S_\hi$ for square-free $\hi\in I(S)$ and let $\ff$, $\ff'$, $\gi$ be divisors of $\hi^{r-1}$.
For $\psi_1$, $\psi_2$, $\rho$ Hecke characters of
order dividing $r$ and conductor dividing $\ci$, we have the functional equation
\[
\begin{aligned}
 L^{T}(s_1+& s_2+2s_3-1,\psi_1\ov{\psi}_2 \chi_{\ff}\ov{\chi}_{\ff'})^{-1} \wZ_\aux^{T}(\sb;\psi_1\chi_{\ff\gi_1},  \psi_2\chi_{\ff'\gi_1}, \rho\chi_{\gi})=|\gi_1|^{1\slash 2 -s_3}
 \\
 &\cdot  \prod_{v\in S_\infty}\frac{\zeta_{F,v}(1-s_3)}{\zeta_{F,v}(s_3)} \sum_{\substack{\ti\mid\hi^r\\ (\ti,\gi)=1} }\sum_{\vp,\eta\in\widehat{R}_\ci} C'(\vp,\eta,\ti;s_3)
\wZ^T(\sigma_3\sb;\psi_1\vp\chi_{\ff\ti}, \psi_2\vp\chi_{\ff'\ti},\eta\rho\chi_\gi)
\end{aligned}
\]
where the coefficients satisfy $C'(-; s_3)\ll |\hi|^\eps $ for $\Re(s_3)=-\eps$ for $\eps>0$.
\end{prop}
\end{para}

	\section{Sieving}
\label{sec:sieving}

\subsection{Sieved MDS}
\begin{para}
	We now introduce the sieved series, whose analytic properties are required to prove the asymptotic formulas
	for the second moment. To treat both the square-free and the $r$-th power-free families uniformly,
	we introduce a parameter $\kappa=0,1$ and recall from Section~\ref{sec:notation} that $I(S)_\kappa$
	denotes the subset of ideals $\ai\in\ I(S)$ with $\ai=\ai_\kappa$, namely
the $r$-th power-free ideals for $\kappa=0$ and the square-free ideals for $\kappa=1$.

	For $s_1, s_2, s_3\in \C$ with sufficiently large real parts, consider
	\be\label{eq:Zsieved}
	Z_{\kappa}^S(\mathbf{s};
	\psi_{\scalebox{1.}{$\scriptscriptstyle 1$}},  \psi_{\scalebox{1.}{$\scriptscriptstyle 2$}}, \rho) =
	\sum_{\ai\in I(S)_\kappa}\frac{
	L^{S}(s_1, \psi_{\scalebox{1.}{$\scriptscriptstyle 1$}}  \chi_{\mathfrak{a}})
	L^{S}(s_2, \overline{\psi}_{\scalebox{1.}{$\scriptscriptstyle 2$}}\overline{\chi}_{\mathfrak{a}})}
	{|\mathfrak{a}|^{s_3}} \rho(\mathfrak{a})Q_\ai(s_1, s_2; \psi_{\scalebox{1.}{$\scriptscriptstyle 1$}}, \psi_{\scalebox{1.}{$\scriptscriptstyle 2$}}),
	\ee
	with the polynomials $Q_\ai$ defined in \S\ref{sec:wZ} and $\psi_1$, $\psi_2$, $\rho$ Hecke characters of order dividing $r$, unramified outside $S$.
Note that  $Q_\ai(-)=P_\ai(-)=1$ if $\ai=\ai_1$ is square-free, so the term $Q_\ai(-)$ does not appear in the definition of $Z_1^S$.

To investigate the untwisted second moment, we only need the analytic properties
of the series $Z_{\kappa}^S(\mathbf{s})$ obtained by taking the characters $\psi_1$, $\psi_2$, $\rho$
to be trivial in the definition~\eqref{eq:Zsieved}. However, the extra flexibility afforded by
these characters will prove useful in the study of the twisted second moment in \cref{thm:twisted_2ndmom}, which is necessary for the mollification procedure in
Section~\ref{sec:mollification}.
\end{para}
\begin{para}
For $\kappa=0,1$ and $\mathfrak{h} \in I(S)$ square-free, define
	\[
	Z_\kappa(\mathbf{s}; \psi_{\scalebox{1.}{$\scriptscriptstyle 1$}},  \psi_{\scalebox{1.}{$\scriptscriptstyle 2$}}, \rho; \mathfrak{h}) = \sum_{\substack{\mathfrak{a} \in I(S) \\ \mathfrak{h} \mid  \ai\ai_\kappa^{\scalebox{.95}{$\scriptscriptstyle -1$}}}}
	\frac{L^{S}(s_1, \psi_{\scalebox{1.}{$\scriptscriptstyle 1$}}  \chi_{\mathfrak{a}_{\scalebox{.9}{$\scriptscriptstyle 0$}}})
	L^{S}(s_2, \overline{\psi}_{\scalebox{1.}{$\scriptscriptstyle 2$}}\overline{\chi}_{\mathfrak{a}_{\scalebox{.9}{$\scriptscriptstyle 0$}}})}
	{|\mathfrak{a}|^{s_3}}
	\rho(\mathfrak{a})Q_\ai(s_1, s_2; \psi_{\scalebox{1.}{$\scriptscriptstyle 1$}}, \psi_{\scalebox{1.}{$\scriptscriptstyle 2$}}).
	\]
	By M\"obius inversion, for $\Re(s_i)>1$ we have the equality
	\be\label{eq:Zkappa}
	Z_{\kappa}^S(\mathbf{s}; \psi_{\scalebox{1.}{$\scriptscriptstyle 1$}},  \psi_{\scalebox{1.}{$\scriptscriptstyle 2$}}, \rho) =
	\sum_{\mathfrak{h} \in I(S)}\mu(\mathfrak{h}) Z_\kappa(\mathbf{s}; \psi_{\scalebox{1.}{$\scriptscriptstyle 1$}},  \psi_{\scalebox{1.}{$\scriptscriptstyle 2$}}, \rho; \mathfrak{h}).
	\ee
	The goal of this section is to express
	the series $Z_\kappa(\mathbf{s}; \psi_{\scalebox{1.}{$\scriptscriptstyle 1$}},  \psi_{\scalebox{1.}{$\scriptscriptstyle 2$}}, \rho; \mathfrak{h})$ in terms of
	 the building blocks
\be\label{eq:wZ}
\wZ^{T} (\sb; \psi_1\chi_\ff, \psi_2\chi_\ff,\rho\chi_\gi)=
	\sum_{\ai\in I(T)} \frac{L^{T}(s_1, \psi_1\chi_{\ff\ai_0}) L^{T}(s_2, \ov{\psi}_2\ov{\chi}_{\ff\ai_0})}{|\ai|^{s_3}}
\rho\chi_\gi(\ai)Q_{\ai}(s_1,s_2; \psi_1, \psi_2)
\ee
	introduced in \S\ref{sec:wZ}, for $T= S \, \cup \, S_{\hi}$
	and coprime $\ff,\gi$ dividing $\hi^{r-1}$.

\end{para}
\subsection{The \texorpdfstring{$v$}~-part}\label{sec:v-part}
\begin{para}
To prepare for the main result, we introduce the generating
series for the polynomials $Q_k$ from \S\ref{sec:wZ}, together with certain subseries.
These series are all rational functions, and are directly responsible for the explicit
 Euler products that appear in the asymptotic formulas for the first and second moments.

Let $f(\mathbf{x}, y; q)$ be the function defined by
	\begin{equation*}
			f(\mathbf{x}, y; q)  =
			 \sum_{k \not \equiv 0 \!\!\!\!\pmod{r}}
			Q_{k}(\mathbf{x}; q)y^{k}
			 +  ((1 - x_1) (1 - x_2))^{-1} \, \cdot \sum_{k \equiv 0 \!\!\!\pmod{r}} Q_{k}(\mathbf{x}; q)y^{k}.
	\end{equation*}
	This is a rational function $f = N/D$, where
	\begin{equation*}
		\begin{split}
			N(\mathbf{x}, y; q) & = 1 - x_1 y - x_2 y + x_1 x_2 y - q x_1 x_2 y + q x_1 x_2 y^2 \\
			& - q^{r - 1} x_1^r x_2^r y^r (1 - y + q y - q x_1 y - q x_2 y +
			q x_1 x_2 y^2)
		\end{split}
	\end{equation*}
	and
	\[
	D(\mathbf{x}, y; q) = (1 - x_1) (1 - x_2) (1 - y) (1 - q x_1 x_2 y) (1 -
	q^{r - 1} x_1^r y^r) (1 - q^{r - 1} x_2^r y^r) (1 -
	q^r x_1^r x_2^r y^r).
	\]
	The function
	$
	f(\psi_{\scalebox{1.}{$\scriptscriptstyle 1$}}(\pi_{v})q_{v}^{\, - s_1}\!, \,  \overline{\psi}_{\scalebox{1.}{$\scriptscriptstyle 2$}}(\pi_{v})q_{v}^{\, - s_2}\!, \, \rho(\pi_{v})q_{v}^{\, - s_3}; q_{v})
	$
	is what is called the $v$-part ($v \notin S$) of 
	$
	\wZ^S(\sb;\psi_{\scalebox{1.}{$\scriptscriptstyle 1$}}, \psi_{\scalebox{1.}{$\scriptscriptstyle 2$}}, \rho)
	$.
\end{para}
\begin{para}\label{sec:sieve_fg}
	For $\kappa=0,1$ and $\varepsilon \in \{1, \ldots, r -1 \}$, define $f_\kappa^{(\varepsilon)}(\xx, y; q)$ by
	\[f_0^{(\varepsilon)}(\mathbf{x}, y; q)  =  \sum_{\substack{k \ge r  \\ k \equiv \varepsilon \!\!\!\!\pmod{r}}}
	Q_{k}(\mathbf{x}; q)y^{k},\quad
	f_1^{(\varepsilon)}(\mathbf{x}, y; q)  =  \sum_{\substack{k > 1 \\ k \equiv \varepsilon \!\!\!\!\pmod{r}}}
	Q_{k}(\mathbf{x}; q)y^{k}
	\]
	and let
	\[
	g(\mathbf{x}, y; q)  = \frac{1}{(1 - x_1) (1 - x_2)}\sum_{\substack{k \ge r \\ k \equiv 0 \!\!\!\!\pmod{r}}}
	Q_{k}(\mathbf{x}; q)y^{k}.
	\]
	Write $g$ as a power series,
	\[
	g(\mathbf{x}, y; q)  = \sum_{k_{1}\!, \, k_{2}, \, l} a(k_{1}, k_{2}, rl; q) x_{1}^{k_{1}} x_{2}^{k_{2}} y^{rl},
	\]
	and for $\varepsilon \in \{0, 1, \ldots, r -1 \}$, define the subseries $g^{(\varepsilon)}$ of $g$ by
	\[
	g^{(\varepsilon)}(\mathbf{x}, y; q)  = \sum_{\substack{k_{1}\!, \, k_{2}, \, l \\
			k_{1} - k_{2} \equiv \varepsilon \!\!\!\!\pmod{r}}} a(k_{1}, k_{2}, rl; q) x_{1}^{k_{1}} x_{2}^{k_{2}} y^{rl}.
	\]
	For an $r$-th power-free ideal $\mathfrak{a} \in I(S)$ and $\kappa=0,1$, we also define
	\[
	f_\ai^{(\kappa)}(\sb;\psi_1,\psi_2)=
	\prod_{\pp_v\mid \ai}f_\kappa^{(\ord_v \ai)}(\psi_1(\pp_v)q_v^{-s_1},\ov{\psi}_2(\pp_v)q_v^{-s_2}, q_v^{-s_3}; q_v)
	\]
	and
	\[g_{\ai,\bi}(\sb;\psi_1,\psi_2)=
	\prod_{\pp_v\mid \bi }g^{(\ord_v \ai)}(\psi_1(\pp_v)q_v^{-s_1},\ov{\psi}_2(\pp_v)q_v^{-s_2}, q_v^{-s_3}; q_v),
	\]
	if $\bi$ is an integral ideal.
\end{para}

\subsection{Decomposition in terms of MDS}
\begin{para}
For $\hi,\ff$ integral ideals, recall that $\hi^{(\ff)}$ denotes the part of $\hi$ coprime to $\ff$.

	\begin{prop}\label{prop: sieving}
	For $\kappa=0,1$, we have the identity
	\[\begin{aligned}
	 Z_\kappa(\sb; \psi_1, \psi_2, \rho; \hi)&=\sum_{\substack{\ff,\gi\mid \hi^{r-1} \\ (\ff,\gi)=1}}
	 f_\ff^{(\kappa)} (\sb; \psi_1, \psi_2)g_{\gi,\hi^{(\ff)}} (\sb; \psi_1, \psi_2)\rho(\ff)\chi_{\ff}(\gi)\\
	&\cdot\sum_{G\in\EE}\psi_G(m_{[\gi]}) \sum_{\eta\in \widehat{R}_{\ci}}  \frac{\ov{\eta}(G) }{|R_{\ci}|}
	 \wZ^{S\cup S_\hi}(\sb;\psi_1\chi_\ff, \psi_2\chi_\ff, \eta \rho\chi_{\gi}).
	\end{aligned}
	 \]
	\end{prop}
\begin{proof}
Write $\ai=\mi\ni$ in the series defining $Z_\kappa(\sb;\psi_1, \psi_2, \rho; \hi)$, with $(\mi,\hi)=(1)$ and $S_\ni=S_\hi$.
One easily checks that the function $Q_\ai$ is twisted multiplicative:
 \[\begin{aligned}
 Q_{\mi \ni}(s_1,s_2;\psi_1,\psi_2)&=Q_{\mi}(s_1,s_2;\psi_1\chi_{\ni_0},\psi_2\chi_{\ni_0})\cdot
 Q_{\ni}(s_1,s_2;\psi_1\chi_{\mi_0},\psi_2\chi_{\mi_0})\\
 &=Q_{\mi}(s_1,s_2;\psi_1\chi_{\ni_0},\psi_2\chi_{\ni_0})\prod_{\substack{\pp_v^k\| \ni\\ \pp_v\mid \ni_0}}
 Q_k(x_{1,v},x_{2,v};q_v  )\\
 &\qquad \qquad \cdot\prod_{\substack{\pp_v^k\| \ni\\ \pp_v\nmid \ni_0}}
 Q_k(\chi_{\mi_0\ni_0}(\pp_v)x_{1,v},\ov{\chi}_{\mi_0\ni_0}(\pp_v)x_{2,v} ;q_v  )
 \end{aligned}
 \]
where we set $x_{1,v}=\psi_1(\pp_v)q_v^{-s_1} $, $x_{2,v}=\ov{\psi}_2(\pp_v) q_v^{-s_2} $, and we used that $Q_k(x_1,x_2;q)$ depends only on $x_1x_2$ if $r\nmid k$.

We now fix $\ff\mid \hi^{r-1}$, and sum over those $\ni$ with $\ni_0=\ff$, $S_\ni=S_\hi$. Using
that the condition $\hi\mid (\ai\slash\ai_\kappa)$ is simply $\ord_v(\ni)>1$ for $v\in S_\hi$ if $\kappa=1$
 and $\ord_v(\ni)\ge r$ for $v\in S_\hi$ if $\kappa=0$, we obtain:
\[\begin{aligned}
	 Z_\kappa(\sb;\psi_1, &\psi_2, \rho; \hi)=\sum_{(\mi,\hi)=1} \sum_{\ff\mid \hi^{r-1}}
	\frac{ L^{S}(s_1, \psi_1\chi_\ff\chi_{\mi_0})L^{S}(s_2 ,\ov{\psi}_2\ov{\chi}_\ff\ov{\chi}_{\mi_0})}{ |\mi|^{s_3} }\rho(\ff\mi)\\
	& \cdot Q_{\mi}(s_1,s_2;\psi_1\chi_\ff,\psi_2\chi_\ff)
	\prod_{v\in S_\ff} f_\kappa^{(\ord_v \ff)} (x_{1,v},x_{2,v},y_{v}; q_v)\\
	& \cdot \prod_{v\in S_{\hi^{(\ff)}}}	L_v(s_1, \psi_1\chi_{\ff\mi_0})^{-1} L_v(s_2, \ov{\psi}_2\ov{\chi}_{\ff\mi_0})^{-1}
	 g(\chi_{\ff\mi_0}(\pp_v)x_{1,v},\ov{\chi}_{\ff\mi_0}(\pp_v)x_{2,v},y_{v}; q_v),
\end{aligned}
\]
where we set $y_{v}=q_v^{-s_3}$. The first product above is $f_\ff^{(\kappa)} (\sb; \psi_1, \psi_2)$, while the second becomes,
after writing $g=\sum_{\eps=0}^{r-1} g^{(\eps)}$:
\[
\frac{ L^{S\cup S_\hi}(s_1, \psi_1\chi_\ff\chi_{\mi_0})L^{S\cup S_\hi}(s_2 ,\ov{\psi}_2\ov{\chi}_\ff\ov{\chi}_{\mi_0})}
{ L^{S}(s_1, \psi_1\chi_\ff\chi_{\mi_0})L^{S}(s_2 ,\ov{\psi}_2\ov{\chi}_\ff\ov{\chi}_{\mi_0})}
\sum_{\gi\mid (\hi^{(\ff)})^{r-1}} \chi_{\ff\mi} (\gi) g_{\gi,\hi^{(\ff)}}(\sb;\psi_1,\psi_2).
\]
By \cref{lem: reciprocity} we have
$\chi_{\mi}(\gi)=\chi_{\gi}(\mi) \psi_{[\mi]}(m_{[\gi]})$. Breaking the series over $\mi$ above
into subseries such that $[\mi]=G\in\EE$, and
observing that the subseries in~\eqref{eq:wZ} with $[\ai]=G$ is given by the
inner sum in the relation to prove (by the orthogonality relation for characters) finishes the proof.
\end{proof}
\end{para}

\section{Convexity bound}\label{sec:convexity}

\begin{para}
In this section we obtain a convexity bound for $\wZ^{S\cup S_\hi}$ in the parameter $\hi$.
We use the functional equations satisfied by $\wZ^{S\cup S_\hi}$ and its associated MDS under the transformations:
\[
\sigma_1\sigma_2 \sb=(1-s_1, 1-s_2, s_1+s_2+s_3-1), \quad
\sigma_3 \sb=(s_1+s_3-1/2, s_2+s_3-1/2, 1-s_3),
\]
where $\sb=(s_1,s_2,s_3)$. The first transformation relates $\wZ^{S\cup S_\hi}$ to a linear combination of functions of the same kind,
via the functional equations of the Hecke $L$-functions in its definition. The second transformation relates $\wZ^{S\cup S_\hi}$ to a linear combination of functions $Z_\aux$, as in \cref{prop: Z funeq}.
\end{para}
\begin{para}\label{par:Zanalytic}
Using these functional equations,
together with the convexity bound for Hecke $L$-functions and for Kubota's generating series of Gauss sums (see \cite[eq. (2.3)]{FHL03})
the same argument as in~\cite[Thm. 3.5]{Di04} shows that
\[
(s-1)^2 (s+w-1-1 \slash r)^2 (w-1) (2s+w-2)
 \wZ^{S\cup S_\hi}(s,s,w; \psi_1\chi_\ff, \psi_2\chi_\ff, \rho\chi_\gi )
\]
has analytic continuation to $\C^2$.
An argument similar to that in \cite[Prop. 4.11]{DGH03} then shows
that the function above is entire of order 1 when viewed as a function of $w$, with $s$ fixed.
\end{para}

\subsection{Functional equations}

\begin{para} We will apply the functional equation
$\sigma_3\cdot \sigma_1\sigma_2 \cdot \sigma_3$, which maps
$\wZ^{S\cup S_\hi}(\sb;\psi_1\chi_\ff, \psi_2\chi_\ff, \rho\chi_\gi )$ to a linear combination of functions
of the same type, evaluated at
\[
\sigma_3\cdot \sigma_1\sigma_2 \cdot \sigma_3\sb = (s_2,s_1, 2-s_1-s_2-s_3).
\]
We specialize $s_1=s_2=1/2$, and apply a convexity bound coming from the Phragm\'en-Lindel\"of principle
on the strip $-\eps<\Re(s_3)<1+\eps$ for small $\eps>0$.
\end{para}

\begin{para}
In preparation for \cref{prop:convexity}, in \S\ref{eq:step1}-\ref{eq:step5} we collect the various functional equations in previous sections, ignoring the terms that grow more slowly than $|\hi|^\eps$ for any $\eps>0$ fixed.

Let $T=S\cup S_\hi$. For an $(r+1)$-th power-free integral ideal $\gi\in I(S)$,
we denote $\ov{\gi}= \prod_{v\in S_\gi} \pp_v^{r-\ord_v(\gi)} $
so that we have $\ov{\chi}_\gi=\chi_{\ov{\gi}}$.
\end{para}
\begin{para} \label{eq:step1}
We assume throughout that $s_1=s_2=1/2$ and $\Re(s_3)=-\eps$ for a small $\eps>0$.
By the functional equation \cref{prop: Z funeq} we have
\[\begin{aligned}
   \wZ^{T}(\sb;\psi_1\chi_\ff, &\psi_2\chi_\ff, \rho\chi_\gi ) \lle
    |\gi_1|^{1/2-\Re(s_3)}\prod_{v\in S_\infty}\frac{|\zeta_{F,v}(1-s_3)|}{|\zeta_{F,v}(s_3)|}\\
  & \cdot
   \sum_{\ti\mid(\hi^{(\gi)})^r }\sum_{\vp,\eta\in\widehat{R}_\ci}
  | L^{T}(s_1+s_2,\psi_1\ov{\psi}_2 )\wZ_\aux^T (\sigma_3 \sb; \psi_1\vp\chi_{\ff \ti\gi_1}, \psi_2\vp\chi_{\ff \ti\gi_1}, \eta\chi_\gi ) |.
\end{aligned}
\]
Note that the factor $L^{T}(s_1+s_2,\psi_1\ov{\psi}_2 )$  above has a pole at $s_1=s_2=1/2$ when $\psi_1=\psi_2$, but the product with $\wZ_\aux^T $ is holomorphic there, so this factor can be
absorbed in the error term.
\end{para}
\begin{para}\label{eq:step2}
For an ideal $\ji$ with $S_\ji\subset S_\hi$, let
$$\ji^*=\ji_0\prod_{\substack{\pp_v\mid\, \hi\\ \pp_v\nmid\, \ji_0 }}\pp_v^r,$$
so that the characters $\chi_\ji$ and $\chi_{\ji^*}$ agree on ideals coprime to $\hi$.
Expressing $\wZ_\aux^T$ as in \eqref{eq: Zaux} in terms of Kubota's Dirichlet series $\wD^T$
and using \cref{corollary: transition-places} we obtain
\[\begin{aligned}
\wZ_\aux^T (\sigma_3 \sb;\psi_1\chi_{\ff \ti\gi_1}, &\psi_2\chi_{\ff \ti\gi_1}, \rho\chi_\gi )
\approx
\sum_{\substack{\eb\in E_{(\ff \ti\gi \gi_1)^*} \\ \eb'\in E_{(\ff \ti\ov{\gi} \gi_1)^*}  }  }
\prod_{v\in S_\eb } q_v^{\frac{r-\ord_v ((\ff \ti\gi \gi_1)^*)}{2}}
\prod_{v'\in S_{\eb'}}q_{v'}^{\frac{r-\ord_{v'} ((\ff \ti\ov{\gi} \gi_1)^*)  }{2}}\\
&\cdot\sum_{\vp_1,\vp_2\in \widehat{R}_\ci}
\sum_{\ai\in I(T)} \frac{\wD^S(s_1', \gi_\eb\ai,\psi_1\vp_1\rho)
		\ov{\wD^S(\ov{s_2'}, \gi_{\eb'}\ai,\psi_2\vp_2\ov{\rho}) }}{|\ai|^{s_3'}}
		\ov{\rho}\ov{\chi}_\gi\chi_{\ff_\eb}\ov{\chi}_{\ff_{\eb'}}(\ai)
\end{aligned}
\]
where $\sigma_3 \sb=\sb'=(s_1', s_2',s_3')$. For brevity, here and in the next two paragraphs we use the symbol ``$\approx$'' to indicate that we ignored the quantities that are not
relevant for the estimates depending on $q_v^{k_v}$ for $v\in S_\hi$ and $k_v\ge 1$ in the actual expression.
\end{para}
\begin{para}\label{eq:step3}
Next we apply the functional equation of Kubota's Dirichlet series, \cref{theorem:Kubota-func-eq}:
\[
\begin{aligned}
  \sum_{\ai\in I(T)} \frac{\wD^S(s_1', \gi_\eb\ai,\psi_1\rho)
		\ov{\wD^S(\ov{s_2'}, \gi_{\eb'}\ai,\psi_2\ov{\rho}) }}{|\ai|^{s_3'}}& \ov{\rho}(\ai)
		\approx
	  |\gi_\eb|^{1/2-s_1'} |\gi_{\eb'}|^{1/2-s_2'}\\
	  \left(\frac{G(1-s_1')G(1-s_2')}{G(s_1')G(s_2')}\right)^{\frac{[F:\mathbb{Q}]}2}
	 \sum_{\vp_1,\vp_2\in \widehat{R}_\ci} \sum_{\ai\in I(T)}&
	\frac{\wD^S(s_1'', \gi_\eb\ai,\ov{\psi_1\vp_1}\ov{\rho})\ov{\wD^S(\ov{s_2''}, \gi_{\eb'}\ai,\ov{\psi_2\vp_2}\rho) }  }
	{|\ai|^{s_3''}}\psi_1\ov{\psi_2}\rho(\ai)
  \end{aligned}
\]
where $\sigma_1\sigma_2 \sb'=\sb''=(s_1'', s_2'',s_3'')=(1-s_1',1-s_2',s_1'+s_2'+s_3'-1)$.
Note that we could ignore the product over $v\in S_f$ in \cref{theorem:Kubota-func-eq} since $\Re(s_1')=\Re(s_2')=-\eps$ under our assumption.
\end{para}
\begin{para} \label{eq:step4}
We now apply \cref{corollary: transition-places2} to express $\wD^S$ back in terms of $\wD^T$:
\[\begin{aligned}
&\wD^S(s_1'', \gi_\eb \ai,\ov{\psi}_1\ov{\rho})\ov{\wD^S(\ov{s_2''}, \gi_{\eb'}\ai,\ov{\psi}_2\rho) }
	\approx \sum_{\substack{\db\in E_{\gi_\eb}\\ \db'\in E_{\gi_{\eb'}} }}
	\prod_{\substack{v\in S_\db \\ k_v= \ord_v (\gi_\eb) }   } q_v^{\frac{k_v}2-(k_v+1)s_1''}\\
	&\cdot \prod_{\substack{v\in S_{\db'} \\ k_v= \ord_v (\gi_{\eb'})} } q_v^{\frac{k_v}2-(k_v+1)s_2''}
	 \sum_{\vp_1,\vp_2\in \widehat{R}_\ci}
	\widetilde{D}^{T}(s_1'', \mathfrak{a}, \ov{\psi}_1\ov{\vp_1}\ov{\rho}\widetilde{\chi}_{\gi_\db})
	\ov{\widetilde{D}^{T}(s_2'', \mathfrak{a}, \ov{\psi}_2\ov{\vp_2}\rho\widetilde{\chi}_{\gi_{\db'}})}
	\ov{\chi}_{\ff_{\db}} \chi_{\ff_{\db'}}(\ai)
\end{aligned}
	\]
where we ignored the product over $v\in T\setminus S$ in \cref{corollary: transition-places2},  since it is bounded as $s_j''=3/2-s_j-s_3$, $j=1,2$ have real part $1+\eps$ under our assumptions.
\end{para}
\begin{para}  \label{eq:step5}
Summing over $\ai \in I(T)$ in \S\ref{eq:step4} and taking into account the characters evaluated at $\ai$
in \S\ref{eq:step3} and \S\ref{eq:step2}, we recover the MDS $\wZ_\aux^T(\sb'', \ov{\psi}_2\chi_{\ji\gi_{\db}}, \ov{\psi}_1\chi_{\ov{\ji}\gi_{\db'}} ,
 \ov{\psi}_1\psi_2\ov{\rho} \ov{\chi}_{\ji})$ where
 \[
 \ji=(\ov{\gi}\ff_\eb\ov{\ff}_{\eb^\prime} \ov{\ff}_{\db} \ff_{\db^\prime})_0
 \]
 (taking $\vp_1=\vp_2=1$ for simplicity).
Using the functional equation in \cref{prop: Z funeq2}, we obtain \\
\resizebox{\textwidth}{!}{
$
L^{T}(s_1''+s_2''+2s_3''-1,
\psi_1\ov{\psi}_2 \chi_{\gi_{\db}}\ov{\chi}_{\gi_{\db'}} \chi_\ji^2 )
\wZ_\aux^T(\sb''; \ov{\psi}_2\chi_{\gi_{\db}}\chi_{\ji}, \ov{\psi}_1\chi_{\gi_{\db'}} \ov{\chi}_{\ji},
 \ov{\psi}_1\psi_2\ov{\rho} \ov{\chi}_{\ji})\lle |\ji_1|^{1/2-\Re(s_3'')}
 $
 }
 \[
\prod_{v\in S_\infty}\frac{|\zeta_{F,v}(1-s_3'')|}{|\zeta_{F,v}(s_3'')|}
\sum_{\substack{\ti\mid\hi^r\\ (\ti,\gi)=1} }
\sum_{\vp,\eta\in\widehat{R}_\ci} |\wZ^T(\sigma_3 \sb'';\ov{\psi}_2\vp\chi_{\gi_\db\ti\ji\ov{\ji}_1}\chi_\ji\ov{\chi}_{\ji_1} ,
\ov{\psi}_1\vp\chi_{\gi_{\db'}\ti\ov{\ji}\ov{\ji}_1}, \ov{\psi}_1\psi_2\ov{\rho}\ov{\chi}_\ji \eta  )|.
\]
The factor $L^T(-)$ above can be absorbed in the error term as in~\S\ref{eq:step1}.
\end{para}

\subsection{Convexity for perfect MDS}
\begin{para}
We are ready to prove a convexity bound for  $\wZ^{S\cup S_\hi}$.

 \begin{prop}[Convexity bound] \label{prop:convexity} Let $\hi\in I(S)$ be square-free and let $\ff, \gi$ be integral ideals dividing $\hi^{r-1}$
such that $(\ff,\gi)=1$.
For $\sigma=\Re (s_3)\in(0,1)$, we have the estimate
\[\begin{aligned}
(s_3-1)^2(s_3-1/2-1/r)^2 &\wZ^{S\cup S_\hi} \big(\tfrac 12, \tfrac 12, s_3; \psi_1 \chi_{\ff}, \psi_2 \chi_{\ff}, \rho\chi_\gi \big)
\ll_\eps |\hi|^{(r-1)(1-\sigma)+\eps}  \\
&\cdot(1+|s_3|)^{4+rd(1-\sigma)+\eps}\max_{\psi\in \widehat{R}_\ci}\sum_{\ti|\hi^{r-1}}
\sum_{\substack{\ai \in I(S)_0\\ (\ai,\ti)=1}}\frac{|L(1/2, \psi\chi_{\ai\ti})|^2}{|\ai|^{1+\eps}},
\end{aligned}
\]
where $d=[F:\Q]$.
\end{prop}
\begin{proof}Fix $s_1=s_2=1/2$. We apply the functional equations in the previous paragraphs
to express the values of the function to estimate on the line $\sigma=-\eps'$ in terms of its values at $\sigma=1+\eps'$,
where we use \cref{prop:estimate at1}. For the dependence on $|s_3|$, the archimedean gamma factors enter in the functional equation in
\S\ref{eq:step1}, \S\ref{eq:step3}, and  \S\ref{eq:step5}, giving a combined contribution of
\[\bigg( \prod_{v\in S_\infty}
\frac{\zeta_{F,v}(1-s_3)}{\zeta_{F,v}(s_3)} \bigg)^2  \bigg(\frac{G(1-s_3)}{G(s_3)}\bigg)^d,
\]
with $G$ defined in \S\ref{par:G}. Applying Stirling's approximation and the Phragm\'en-Lindel\"of principle yields the bound in
$|s_3|$ stated above, with the factor of $4$ in the exponent due to the degree in $s_3$ of the factor in front.

For the $|\hi|$-dependence, for each place $v\in S_\hi$, we have to collect the contributions involving $q_v=|\pp_v|$
in each of the transformations \S\ref{eq:step1}-\ref{eq:step5}. We consider three cases.

Case I. $\pp_v^{l_v}\| \ff$.  Then $\pp_v\nmid \gi$, so there is no contribution from the functional equation in~\S\ref{eq:step1}.
In~\S\ref{eq:step2} we may have $\pp_v^{k_v}\| (\ff\ti)^*$, for any $1\le k_v\le r$. We record in Table~\ref{table1} the
contributions from  the functional equations \S\ref{eq:step1}-\ref{eq:step5}, depending on whether $v\in S_\eb$ or
$v\in S_{\eb'}$ in the functional equation in \S\ref{eq:step2}.

\begin{table}[!ht]
\begin{center}
 \begin{tabular}{c|cccc}
 \  & \S\ref{eq:step2} & \S\ref{eq:step3} & \S\ref{eq:step4}-\ref{eq:step5}& Total contribution \\ \hline
Ia.  \makecell{$v\in S_\eb$\\
 $v\in S_{\eb'} $} &   $q_v^{r-k_v}$  &  $q_v^{k_v-2+\eps}$ &    1  &  $q_v^{r-2+\eps}$ \\ \hline
 Ib.  \makecell{$v\in S_\eb$\\
 $v\notin S_{\eb'} $} &   $q_v^{\frac{r-k_v}2}$  &  $q_v^{\frac{r-2}2+\eps}$ &    $q_v^{\frac 12}$  &  $q_v^{r-\frac{k_v+1}2+\eps}$\\
 \hline
 Ic. \makecell{$v\notin S_\eb$\\
 $v\notin S_{\eb'} $} &  1 & $q_v^{r-k_v+\eps}$ & 1 & $q_v^{r-k_v+\eps}$
 \end{tabular}
\end{center}
\caption{\label{table1}Contributions from the functional equations \S\ref{eq:step1}-\ref{eq:step5} in case I. }
\end{table}
The case $v\notin S_\eb$, $v\in S_{\eb'} $ is identical to Ib.
For the contribution from \S\ref{eq:step5}, we used that $\ord_v\ff_\eb=k_v-1$ if $v\in S_\eb$
and $\ord_v\ff_\eb=0$ if $v\notin S_\eb$. We do not get any contribution from
\S\ref{eq:step4} or from $\pp_v|\widetilde{\ff}_{\db} \ff_{\db'}$ in \S\ref{eq:step5} because
$\Re (s_1'')=\Re (s_2'') =1+\eps$.

Note that $k_v\ge 2$ in cases Ia. and Ib., while $k_v\ge 1$ in Ic., so the largest contribution is $q_v^{r-1+\eps}$, occurring for $k_v=1$ in case Ic.

Case II. $\pp_v^{l_v}\| \gi$ with $1\le l_v\le r-1$. Then $\ord_v(\gi\gi_1)^*=l_v+1$,
$\ord_v(\ov{\gi}\gi_1)^*=r-l_v+1$ are both at least 2.
We consider three cases.

\begin{table}[!ht]
\begin{center}
 \begin{tabular}{c|ccccc}
 \ & \S\ref{eq:step1} & \S\ref{eq:step2} & \S\ref{eq:step3} & \S\ref{eq:step4}-\ref{eq:step5}& Total contribution \\ \hline
IIa.  \makecell{$v\in S_\eb$\\
 $v\in S_{\eb'} $} &    $q_v^{\frac 12}$ & $q_v^{\frac{r-2}2}$  & $q_v^{\frac{r-2}2+\eps}$ & $ q_v^{\frac 12}$ & $q_v^{r-1+\eps}$ \\ \hline
  IIb.  \makecell{$v\in S_\eb$\\
 $v\notin S_{\eb'} $} &$q_v^{\frac 12}$ & $q_v^{\frac{r-l_v-1}2}$  & $q_v^{l_v-1+\eps}$ & $1$ &  $q_v^{\frac{r+l_v-2}2+\eps}$\\ \hline
 IIc. \makecell{$v\notin S_\eb$\\
 $v\notin S_{\eb'} $} & $q_v^{\frac 12}$ & 1 & $q_v^{\frac{r-2}2+\eps}$ & $ q_v^{\frac 12}$ & $q_v^{\frac{r}2+\eps}$
 \end{tabular}
\end{center}
\caption{\label{table2}Contributions from the functional equations \S\ref{eq:step1}-\ref{eq:step5} in case II.}
\end{table}
The case $v\notin S_\eb$, $v\in S_{\eb'} $ is similar to IIb, with $l_v$ replaced by $r-l_v$.
Note that in case IIb., the contribution $q_v^{1/2}$ from  \S\ref{eq:step5} is offset
by the negative power of $q_v$ coming from \S\ref{eq:step4}.

we have $\ord_v \ff_\eb=l_v$ so $\ord_v \ji=0$ in \S\ref{eq:step5},
so we do not get any contribution in this case.

Case III. $\pp_v\nmid \ff\gi$. Then either we get no contribution at all, or if $\pp_v^{k_v}\|\ti$ in~\S\ref{eq:step2} the total contribution
can be computed as in case I, being at most  $q_v^{r-1+\eps}$.

In conclusion, the largest overall contribution of a place $v$ is $q_v^{r-1+\eps}$, and we obtain by the Phragm\'en-Lindel\"of principle:
\[\begin{aligned}
(1-s_3)^2&(s_3-1/2-1/r)^2 \wZ^{S\cup S_\hi} \big(\tfrac 12,\tfrac 12,s_3; \psi_1 \chi_{\ff}, \psi_2 \chi_{\ff}, \rho\chi_\gi \big)
\ll_\eps |\hi|^{(r-1)(1-\sigma)+\eps}\\
&\cdot(1+|s_3|)^{4+rd(1-\sigma)+\eps} \sum_{\substack{\ti,\ti'| \hi^{r-1}\\ \psi_1,\psi_2,\eta\in\widehat{R}_\ci}}
\left|\wZ^{S\cup S_\hi}  \big(\tfrac 12,\tfrac 12,1+\eps; \psi_1 \chi_{\ti}, \psi_2 \chi_{\ti'}, \rho\eta\chi_\gi \big) \right|.
\end{aligned}
\]

Let $T=S\cup S_\hi$.
By obvious bounds for the polynomials $Q_\ai(\tfrac 12,\tfrac 12; \psi_1,\psi_2)$ in the definition of  $\wZ^{T}$ we have
\[\wZ^{T} (\tfrac 12,\tfrac 12,1+\eps;\psi_1 \chi_{\ti}, \psi_2 \chi_{\ti'}, \rho\eta\chi_\gi )\ll \sum_{\ai \in I(T)}
\frac{|L^{T}(1/2, \psi_1\chi_{\ai_0\ti})|\cdot |L^{T}(1/2, \ov{\psi}_2\ov{\chi}_{\ai_0\ti'})|}{|\ai|^{1+\eps}}.
\]
By Cauchy-Schwarz, the sum of the right-hand side over all $\psi_1, \psi_2\in\widehat{R}_\ci$
and all $\ti,\ti'|\hi^{r-1}$ is
\[\lle \max_{\psi\in\widehat{R}_\ci} \sum_{ \ti|\hi^{r-1} }
\sum_{\ai\in I(T)}\frac{|L^{T}(1/2, \psi\chi_{\ai_0\ti})|^2}{|\ai|^{1+\eps}}.
\]
We can replace $L^T$ by the complete $L$-series, as the omitted Euler factors contribute $\lle$. Writing $\ai=\ai_0 \bi^r$ and using that
$\sum_{\bi\in I(T)} |\bi|^{-r(1+\eps)}$ converges finishes the proof.
\end{proof}
\end{para}

\begin{para} The following estimate is a standard consequence of the large sieve
for $r$-th order characters developed in~\cite{BGL14}.
\begin{prop}\label{prop:estimate at1} Let $\psi$ be
a Hecke character unramified outside $S$ of order dividing $r$. Then we have
\[
\sum_{\substack{\ai\in I(S)_0\\ (\ai,\ff)=1}}\frac{|L(1/2, \psi\chi_{\ai\ff})|^2}{|\ai|^{1+\eps}}\ll |\ff_1|^{1/2+\eps}
\]
as $\ff\in I(S)$ runs over $r$-th power-free ideals.
\end{prop}
Conjecturally, the exponent $1/2$ above should be 0 (Lindel\"of's hypothesis for this family). For $r=3$ we will see that
one can replace the exponent $1/2$ by $1/3$, after further averaging over $\hi$.
\begin{proof}
This follows from the large sieve for $r$-th order characters~\cite[Thm. 1.3]{BGL14}. The case $r=3$
is worked out in~\cite[Prop. 4.2]{DFDS24}, while the general case is proved in the Appendix.
\end{proof}
Combining the previous estimate with \cref{prop:convexity} yields the following explicit convexity bound.
\begin{cor}\label{cor:convexity}
For $\Re(s_3)=\sigma\in (0, 1)$ we have:
 \[\begin{aligned}
(1-s_3)^2(s_3-\tfrac12-\tfrac1r)^2&  \wZ^{S\cup S_\hi} \big(\tfrac 12,\tfrac 12,s_3; \psi_1 \chi_{\ff}, \psi_2 \chi_{\ff}, \rho\chi_\gi \big)\ll_\eps \\
& \ll_\eps |\hi|^{1/2+(r-1)(1-\sigma)+\eps} (1+|s_3|)^{4+rd(1-\sigma)+\eps} .
   \end{aligned}
\]
\end{cor}
\end{para}

\subsection{An improved convexity bound}\label{sec:refined_convexity}
\begin{para}\label{par:Zrecursion}
We can improve the bound in \cref{cor:convexity} by using an inductive method first used in \cite[Sec. 6]{Di19}.
Let
\[\wg(\xx,y;q) =
\frac{\sum_{k \equiv 0 \!\!\!\!\pmod{r}}
	Q_{k}(\mathbf{x}; q)y^{k}}{(1 - x_1) (1 - x_2)}=g(\xx,y;q)+ {(1 - x_1)^{-1} (1 - x_2)^{-1}}
\]
and for $0\le  e <r $ let $\wg^{( e )}$ be the part of $\wg$ containing monomials $x_1^a x_2^b y^l$
with $a-b\equiv  e  \pmod r$ when expanded as a power series in $x_1,x_2, y$.
For $0< e <r$ let
\[
\wf^{( e )}(\xx,y;q)=\sum_{k \equiv  e  \!\!\!\!\pmod{r}}
	Q_{k}(\mathbf{x}; q)y^{k}
\]
which agrees with $f_1^{( e )}$ for $ e >1$, while   $\wf^{(1)}=f_1^{(1)}+y$.
By an argument similar to the proof of \cref{prop: sieving} we can express the perfect MDS
$\wZ^{S\cup S_\hi}$ in terms of $\wZ^{S\cup S_{\hi\pp}}$ for some prime $\pp=\pp_v\in I(S\cup S_\hi)$
with $q_v=|\pp_v|$:
\[\begin{aligned}
   &\wZ^{S\cup S_\hi}(\sb;\psi_1\chi_\ff, \psi_2\chi_\ff, \rho\chi_{\gi})=\wg^{(0)}(x_{1,v},x_{2,v},y_v;q_v ) \wZ^{S\cup S_{\hi\pp}}(\sb;\psi_1\chi_\ff, \psi_2\chi_\ff, \rho\chi_{\gi})\\
 &+  \sum_{ e =1}^{r-1} \rho\chi_\gi(\pp^ e ) \wf^{( e )}(x_{1,v},x_{2,v},y_v;q_v )
   \wZ^{S\cup S_{\hi\pp}}(\sb;\psi_1\chi_{\ff\pp^ e }, \psi_2\chi_{\ff\pp^{ e }}, \rho\chi_{\gi})\\
&+ \sum_{ e =1}^{r-1}\chi_\ff(\pp^{ e }) \wg^{( e )}(x_{1,v},x_{2,v},y_v;q_v )
\sum_{\substack{G\in\EE\\\eta\in \widehat{R}_{\ci} }} \frac{1}{|R_{\ci}|}  \ov{\eta}(G) \psi_G(m_{[\pp]}^ e )
	 \wZ^{S\cup S_{\hi\pp}}(\sb;\psi_1\chi_\ff, \psi_2\chi_\ff, \eta \rho\chi_{\gi\pp^ e }),
   \end{aligned}
\]
where $x_{1,v}=\psi_1(\pp_v)q_v^{-s_1} $, $x_{2,v}=\ov{\psi}_2(\pp_v) q_v^{-s_2} $, $y_v=q_v^{-s_3}$.

\begin{prop}\label{prop:refined-convexity}
Let $1/2<\sigma=\Re(w)<1$, let $\hi$ be square-free, and let
$\ff,\gi\mid\hi^{r-1}$ with $(\ff,\gi)=1$.  Then
\[
\begin{aligned}
(w-1)^2(w-1/2-1/r)^2&
\wZ^{S\cup S_\hi}\big(\tfrac12,\tfrac12,w;
 \psi_1\chi_\ff,\psi_2\chi_\ff,\rho\chi_\gi\big)\\
&\ll_ e  |\hi|^{(r-1)(1-\sigma)+ e }|\rad(\ff\gi)|^{1/2}
(1+|w|)^{4+rd(1-\sigma)+ e }.
\end{aligned}
\]
In particular, the exponent $1/2$ in \cref{cor:convexity} is needed only at primes dividing
$\ff\gi$.
\end{prop}
\begin{proof}
We use the same notation as in \cite[Prop. 6.3]{Di19} to facilitate
comparing with the argument there. Write
\[
 \ci_1=\rad(\ff),\qquad \ci_2=\rad(\gi),\qquad
 \ci_3=\hi^{(\ff\gi)},
\]
so that $\hi=\ci_1\ci_2\ci_3$.  We induct on $\omega(\ci_3)$, simultaneously for all characters
and all exponents at the primes dividing $\ci_1\ci_2$.  The case $\ci_3=1$ is
\cref{cor:convexity}.

Let $\pp\nmid \hi$ a prime ideal and $q=|\pp|$.
From the explicit formulas in \S\ref{sec:v-part} we have the estimates
\[
 |\wf^{( e )}(q^{-1/2},q^{-1/2},q^{-w}; q)|\ll q^{- e \sigma}\quad(1\leq e <r),
\]
and
\[
 |\wg^{( e )}(q^{-1/2},q^{-1/2},q^{-w}; q)|\ll q^{-1/2}\quad(0< e <r),
 \qquad |\wg^{(0)}(q^{-1/2},q^{-1/2},q^{-w}; q)|^{-1}\ll 1.
\]

Let $a:=(r-1)(1-\sigma)+ e $. We use the relation before the proposition, specialized to $\sb_w=(1/2,1/2,w)$,
to express the first term on the right in terms of the rest.
The induction hypothesis
together with the estimates above give:
\[
\begin{aligned}
(w-1)^2(w-1/2-1/r)^2&\wZ^{S\cup S_{\hi\pp}}(\sb_w;\psi_1\chi_\ff, \psi_2\chi_\ff,  \rho\chi_{\gi})
\lle \\
&\cdot \bigg(|\ff_1\gi_1|^{\tfrac12} |\hi|^a+
|\ff_1\gi_1\pp|^{\tfrac12} |\hi\pp|^a\cdot |\pp|^{-\sigma}+
|\ff_1\gi_1\pp|^{\tfrac12} |\hi\pp|^a\cdot |\pp|^{-\tfrac12}\bigg),
\end{aligned}
\]
where the terms on the right-hand side come from the terms containing
\[\wZ^{S\cup S_\hi}(\sb_w;\psi_1\chi_\ff, \psi_2\chi_\ff, \rho\chi_{\gi}),
\wZ^{S\cup S_{\hi\pp}}(\sb_w;\psi_1\chi_{\ff\pp^ e }, \psi_2\chi_{\ff\pp^{ e }}, \rho\chi_{\gi}),
\wZ^{S\cup S_{\hi\pp}}(\sb_w;\psi_1\chi_\ff, \psi_2\chi_\ff, \eta \rho\chi_{\gi\pp^ e }),\]
respectively
($0< e <r)$. Since $\sigma>1/2$, we obtain that the right-hand side is
$\ll |\ff_1 \gi_1|^{1/2} |\hi\pp|^{a}$, finishing the induction step.

The dependence on $w$ is unchanged from \cref{cor:convexity}.
\end{proof}
\end{para}

\section{Analytic continuation of sieved MDS}\label{sec:analytic}
\begin{para}Our goal is to extend meromorphically the MDS $Z_\kappa^S(\sb):= Z_\kappa^S(\sb;1, 1, 1)$, $\kappa=0,1$,
in \S\ref{sec:sieving} to a region
containing the points $\sb_w=(1/2, 1/2, w)$ with $\Re (w)<1$ as small as possible.
In view of~\eqref{eq:Zkappa}, we need to bound  $Z(\sb_w;\hi):=Z(\sb_w;1,1,1; \hi)$ uniformly in $|\hi|$ in such a region,
for $\hi\in I(S)$ square-free. Because of the relation in \cref{prop: sieving},
we can use the convexity bounds for $\wZ^{S\cup S_\hi}(\sb_w; \psi_1 \chi_{\ff}, \psi_2 \chi_{\ff}, \rho\chi_\gi)$ in the previous section to provide such a bound.
\end{para}

\subsection{Convexity for sieved MDS}
\begin{para}\label{sec:fg_estimates}
In order to apply \cref{prop: sieving}, we need estimates on the functions $f_\kappa^{(\eps)}$ and $g^{(\eps)}$
from \S\ref{sec:sieve_fg}. We take $x_1=x_2=q^{-1/2}$ and $y=q^{-w}$,
where $1/2 <\sigma=\Re(w)<1$, and $q=|\pp|$ for a prime ideal $\pp$.
Noting that the factor $q^{r-1} x_1^r y^r$ in the denominator $D$ of $f$ equals $q^{r/2-1-r\sigma}$,
we easily obtain for $0<\eps<r$:
\[
f_\kappa^{(\eps)}(q^{-1/2},q^{-1/2},q^{-w}; q)\ll
\begin{cases}q^{-2\sigma} & \text{if $\kappa=1$} \\
              q^{r/2-1-(r+1)\sigma} &  \text{if $\kappa=0$}
        \end{cases}
\]
\[
g^{(\eps)}(q^{-1/2},q^{-1/2},q^{-w}; q)\ll
\begin{cases}
 q^{r/2-3/2-r\sigma} & \text{if $0<\eps<r$}\\
 q^{r/2-1-r\sigma} & \text{if $\eps=0$}.
\end{cases}
\]
Note that $q^{-2\sigma}$ is a uniform bound for both $f_1^{(\eps)}$ and $g^{(\eps)}$, since we are assuming $\sigma>1/2$.
\end{para}

\begin{para}\label{par:conv}
We now apply \cref{prop: sieving} and~\cref{cor:convexity}, taking into account the estimates in \S\ref{sec:fg_estimates}.
For $1/2<\sigma=\Re(w)<1$ we have (we ignore the convexity bound in the $|w|$-aspect for simplicity):
\[
(w-1)^2(w-\tfrac12-\tfrac1r)^2 Z_\kappa(\tfrac 12,\tfrac 12,w; \hi) \ll_\eps
|\hi|^{1/2+(r-1)(1-\sigma) + \eps }\cdot
\begin{cases}
 |\hi|^{-2\sigma} & \text{if } \kappa=1 \\
 |\hi|^{r/2-1-r\sigma} & \text{if } \kappa=0 .
\end{cases}
\]
\end{para}

\begin{para}\label{par:convergence}
For $\kappa=1$, the right-hand side of \eqref{eq:Zkappa} has meromorphic continuation to the region
\[\tfrac12+(r-1)(1-\sigma)-2\sigma +\eps<-1, \quad  \sigma>\frac{r+1/2}{r+1}+\eps'.
\]
This region contains the pole at $w=1$, so we get an asymptotic formula for all $r$ with
an explicit error term determined by the abscissa of convergence above.
For $r=3$, we will show in \cref{prop:conv_req3} that we can replace $1/2$ by $1/3$
in the bound for $\sigma$. Therefore we obtain a meromorphic continuation of $Z_1^S$
to the region $\sigma>\frac 56 +\eps'$, matching the error in the asymptotic for the second moment from~\cite{DFDS24}.
\end{para}
\begin{para}\label{par:kappa0region}
For $\kappa=0$, we can improve the region of convergence by using the refined estimate in \cref{prop:refined-convexity}.
 Namely, the contribution of a pair $(\ff,\gi)$ in
\cref{prop: sieving} is bounded, apart from the common factor in $s_3=w$, by
\[
 |\hi|^{(r-1)(1-\sigma)+\eps}|\rad(\ff\gi)|^{1/2}
 |f_\ff^{(\kappa)}(\tfrac12,\tfrac12,w)
 g_{\gi,\hi^{(\ff)}}(\tfrac12,\tfrac12,w)|.
\]
Using the estimates in \S\ref{sec:fg_estimates} and taking into account \cref{prop:refined-convexity}, we check that the worst local exponent occurs at primes $\pp\mid\hi$ such that
$\pp\nmid \ff\gi $ or $\pp\|\gi$. We obtain:
\be\label{eq:refined-sieved-convexity}
(w-1)^2(w-\tfrac12-\tfrac1r)^2 Z_0(\tfrac12,\tfrac12,w;\hi)\ll_\eps |\hi|^{1/2+(r-1)(1-\sigma)+r/2-3/2-r\sigma+\eps}
 \ee
 (for $\kappa=1$ the refined convexity estimate gives no improvement).

Therefore the right-hand side of \eqref{eq:Zkappa} has meromorphic continuation to the region
\[(r-1)(1-\sigma)+r/2-1-r\sigma+\eps<-1, \quad  \sigma>\frac{3r-2}{4r-2}+\eps'.
\]
For $r=3$ and $r=4$, this region includes both poles at $w=1$ and $w=\frac12 +\frac1r$.
\end{para}

\subsection{An improved region of convergence for \texorpdfstring{$r=3$}~}
\begin{para}
We now take $r=3$ and improve the region of meromorphic continuation in \S\ref{par:convergence}.
We assume $\kappa=1$, as in the case $\kappa=0$ the region of convergence obtained above
is still better than if applying the next proposition.

\begin{prop} \label{prop:conv_req3}
For $\kappa=1$,
  the Dirichlet series
  $$ Z_\kappa^S\big(\tfrac 12, \tfrac 12, w\big) $$
  has meromorphic continuation for $\Re(w)>5/6$ with a double pole at $w=1$.
\end{prop}
\begin{proof}
By \cref{prop:convexity}, it is enough to show that for $\psi\in \widehat{R}_\ci$ the following series converges:
\[
S_\psi(\sigma)=\sum_{\hi=\hi_1\in I(S)} |\hi|^{2-4\sigma}\sum_{\ff|\hi^2}
\sum_{\substack{\ai\in I(S)_0\\ (\ai,\ff)=1}} \frac{|L(1/2, \psi\chi_{\ai\ff})|^2}{|\ai|^{1+\eps}},
\]
for $\sigma>5/6$. Recall that $I(S)_0$ consists of the $r$-th power-free ideals in $I(S)$. Writing $\hi=\ff_1 \gi$, the sum over $\gi$ converges (for $\sigma>3/4$), and we have
\[
S_\psi(\sigma)\ll_\eps \sum_{\substack{\ai,\ff \in I(S)_0\\ (\ai,\ff)=1}}|\ff_1|^{2-4\sigma}
\frac{|L(1/2, \psi\chi_{\ai\ff})|^2}{|\ai|^{1+\eps}}.
\]
We write $\ff=\bi\ei^2$ with $\bi$, $\ei$ square-free
and coprime, so that $\ff_1=\bi \ei$. Replacing $\ai\bi$ by $\ai$ and changing $\eps$ if necessary we have
\be\label{eq:Sbound}
S_\psi(\sigma)\ll_\eps \sum_{\substack{\ai,\ei \in I(S)_0\\ (\ai,\ei)=1}}|\ei|^{2-4\sigma}
\frac{|L(1/2, \psi\chi_{\ai\ei^2})|^2}{|\ai|^{1+\eps}}.
\ee
Let $\ai=\qi_1\qi_2^2$, with $\qi_1,\qi_2$ square-free and coprime. Note that
\[
\frac{|L( 1/2, \psi\chi_{\qi_1\qi_2^2\ei^2} )|}{|\qi_1\qi_2^2|^{1+\eps}}=
\frac{|L( 1/2, \ov{\psi}\chi_{\qi_1^2\qi_2\ei} )|}{|\qi_1^2 \qi_2|^{1+\eps}}
|\qi_1/\qi_2|^{1+\eps},
\]
and we consider first the subseries of~\eqref{eq:Sbound} for which $|\qi_1/\qi_2|<|\ei|^{1/3}$.
Since $2-4\sigma+(1+\eps)/3<-1-\eps$, namely $\sigma>5/6+\eps/3$,
this subseries is bounded by
\[\sum_{\substack{\ai,\ei \in I(S)_0\\ (\ai,\ei)=1}}
\frac{|L(1/2, \psi\chi_{\ai\ei})|^2}{|\ai\ei|^{1+\eps}}\ll \sum_{\ai \in I(S)_0}
\frac{|L(1/2, \psi\chi_{\ai})|^2}{|\ai|^{1+2\eps}}Q_\ai,
\]
with $Q_\ai=Q_\ai(1/2,1/2;\psi,\ov{\psi})$ as in \cref{Main-Thm2-r-th-power-free-alt} for $\ai\in I(S)_0$. The last series converges because
$\wZ^{S}(1/2,1/2,w; \psi,\psi,1)$ is a Dirichlet series in $w$ with nonnegative coefficients with
its first pole at $w=1$, so by Landau's lemma it converges absolutely at $w=1+2\eps$.

Therefore we can restrict to $|\qi_1/\qi_2|>|\ei|^{1/3}$ in~\eqref{eq:Sbound},
and using dyadic summation for the series over $\ai$ yields
\[
\sum_{\substack{Q_1,Q_2\text{ dyadic}\\ Q_1/Q_2>|\ei|^{1/3} }} \frac{1}{(Q_1 Q_2^2)^{1+\eps}}
\sum_{\substack{\qi_1 \asymp Q_1 \\ \qi_2\asymp Q_2 } }
|L(1/2, \psi\chi_{\qi_1 \qi_2^2\ei^2})|^2.
\]
The inner sum is
$\ll_{\eps'}(|\ei|^{1/3}Q_1Q_2^{4/3})^{1+\eps'}$ by \cite[Prop. 4.2]{DFDS24}
(or its version for a general field $F$ in \cref{secondmoment-cor}), as this term
dominates the one containing $|\ei|^{1/2}$ in loc. cit. because of the
restriction $Q_1/Q_2>|\ei|^{1/3}$.
Taking $\eps'<\eps$, the last display is $\ll_{\eps'} |\ei|^{1/3+\eps'}$, and we conclude that
$S_\psi(\sigma)$ converges for $\sigma>5/6+\eps'$.
\end{proof}
\end{para}
\begin{para}
 Putting together the results of this section we obtain:
\begin{prop}\label{prop:convergence_sievedMDS}
The sieved MDS $ Z_\kappa^S\big(\tfrac 12, \tfrac 12, w\big) $ has meromorphic continuation to the region
$\Re(w)>\dd_\kappa$ with a double pole at $w=1$, where
 \[
 \dd_\kappa=\begin{cases} \frac{r+A_r}{r+1}  & \text{if } \kappa=1 \\
		  \frac{3r-2}{4r-2}  & \text{if } \kappa=0
         \end{cases}
\]
with $A_r=1/2$ if $r>3$ and $A_3=1/3$. For $\kappa=0$ and $r=3,4$, this region also contains the second double pole at $w=1/2+1/r$. For
$1>\Re(w)>\dd_\kappa$, we have the convexity bound
\[
(w-1)^2(w-\tfrac12-\tfrac1r)^2
Z_\kappa^S\big(\tfrac 12, \tfrac 12, w\big)\ll_\eps  (1+|w|)^{rd+\eps}.
\]
\end{prop}
\end{para}

\section{Residues}\label{sec:residues}
\begin{para}
The goal of this section is to compute the principal part of the MDS
$ Z_\kappa^S\big(\tfrac 12, \tfrac 12, w\big) $ at $w=1$.
Using \eqref{eq:Zkappa} and \cref{prop: sieving}, this will follow
from the corresponding principal parts of the meromorphic functions
$Z_\kappa(s, s, w; \hi)$ and
$\wZ^{S\cup S_\hi}(s, s, w; \chi_\ff,\chi_\ff,\chi_\gi)$ with
$\ff,\gi|\hi^{r-1}$ coprime. As in the previous section,  we suppress the three trivial characters
and write $Z_\kappa^S(\mathbf s)$ for $Z_\kappa^S(\mathbf s;1,1,1)$ and $Z_\kappa(\sb; \hi)$.
\end{para}

\begin{para}
\label{par:prinpart}
To recover the principal part of a multiple Dirichlet series at the point where
the polar hyperplanes $w=1$ and $w=2-2s$ meet, we use the following elementary observation.

Let $f(s,w)$ be a meromorphic function such that $(w-1)(w+2s-2)f(s,w)$ is holomorphic near $(1/2,1)$.
Then the functions
\[
\Ac(s)=(2s-1)\res_{w=1} f(s,w), \quad \Bc(s)=(2s-1)\res_{w= 2-2s} f(s,w)
\]
are holomorphic at $s=1/2$, and
$f(1/2, w)$ has at most a double pole at $w=1$ with
\be\label{eq:doubleres}
\lim_{w\rightarrow 1} (w-1)^2 f(1/2,w) =\Ac(1/2)=-\Bc(1/2),
\ee
as it is seen by expressing the limit as a double limit in two ways.
A short computation then shows that the principal part of $f(1/2,w)$ at $w=1$ is
\[\frac{\Ac(1/2)}{(w-1)^2}+\frac{1}{w-1} \frac{\Ac'(1/2)+\Bc'(1/2)}{2}.
\]
\end{para}

\subsection{Principal part of the perfect MDS} \label{par:Zresidue}
\begin{para}
By the elementary observation in \S\ref{par:prinpart}, all the computations in this section
rely on determining the residues of the perfect MDS $\wZ^{S\cup S_\hi}(s,s,w;-)$ at $w=1$, and $w=2-2s$
computed in the next proposition.
Recall that
\[
 \wZ^T(\mathbf s;\psi_1,\psi_2,\rho;G)
 :=\frac{1}{\lvert R_{\mathfrak c}\rvert}
 \sum_{\eta\in\widehat R_{\mathfrak c}}
 \overline\eta(G)
 \wZ^T(\mathbf s;\psi_1,\psi_2,\rho\eta),
\]
is the subseries of $\wZ^T$ in~\S\ref{sec:wZ} with $[\ai]=G\in \EE$.
\begin{prop}\label{prop:Zresidue}
Let $\hi\in I(S)$ be square-free and let $\ff,\gi|\hi^{r-1}$ with $(\ff,\gi)=1$.
Set $T=S\cup S_\hi$, and fix a representative $G\in \EE$.
Assume $s\ne 1/2$ is in a neighborhood of $1/2$.

(a) The function
$\wZ^{T}(s,s,w; \chi_\ff, \chi_\ff, \chi_\gi )$ has a pole at $w=1$ only if $\chi_\gi=1$, and
we have
\[\res_{w=1} \wZ^{T}(s,s,w;\chi_\ff, \chi_\ff, 1;G)=\frac{1}{|R_\ci|}\Rc^T(s),
\]
where $\Rc^T(s)=\zeta_F^T(2s)\zeta_F^T(rs)^2  \res_{w=1}\zeta_F^T(w)$.

(b) We have
\[
\begin{aligned}
\res_{w=2-2s} \wZ^{T}(s,s,w; \chi_\ff, \chi_\ff, \chi_\gi;G )=
\frac{1}{|R_\ci|}\Rc^T(1-s)\ov{\chi}_\ff(\gi)\ov{\psi}_G(m_{[\gi]}) \\
\cdot |\ff_1|^{1-2s} A_{[\ff G]}(s)B_{\gi,\hi^{(\ff)}}(s),
\end{aligned}
\]
where
\[
A_{H}(s)=|D_F \cond(\psi_{H})|^{1-2s}\prod_{v\in S}
\frac{L_v(1-s,\chi_{H})L_v(1-s,\ov{\chi}_{H})}{L_v(s,\chi_{H})L_v(s,\ov{\chi}_{H})}
\]
for $H\in\EE$, and
\[
B_{\gi,\ai}(s)=\prod_{v\in S_{\ai}}\bigg(
\frac 1r \sum_{\xi\in\mu_r} \xi^{\ord_v(\gi)}\frac{(1-q_v^{-s}\xi)(1-q_v^{-s}\ov{\xi})}{(1-q_v^{s-1}\xi)(1-q_v^{s-1}\ov{\xi})}\bigg).
  \]
Note that $A_{H}(1/2)=1$, while $B_{\gi,\hi^{(\ff)}}(1/2)$ is 1 for
$\gi=1$ and 0 otherwise.
\end{prop}
The local factor in the product above can be computed explicitly:
\be\label{eq:mur_sum}
\frac 1r \sum_{\xi\in\mu_r} \xi^{k}\tfrac{(1-x\xi)(1-x\ov{\xi})}{(1-y\xi)(1-y\ov{\xi})}=
\begin{cases}
\tfrac{(1-x/y)(1-xy)(y^k+y^{r-k})}{(1-y^2)(1-y^r)}  & \text{ if $k\ne 0$}\\
 \tfrac{(1+x^2)(1+y^r)-2x(y+y^{r-1})}{(1-y^2)(1-y^r)}  & \text{ if $k= 0$}
\end{cases}
\ee
for $0\le k< r$.
\begin{proof}
(a) Using the expression for $\wZ^{T}(\sb;\chi_\ff, \chi_\ff, \chi_\gi )$ in terms of $\wZ^{T,(3)}$ and
$\wZ^{T,(2)}$ in \S\ref{sec:2ndmoment}, we conclude from
the expansion of $\wZ^{T,(2)}$ in \cref{prop: Zaux2} that $\wZ^{T}(\sb;\chi_\ff, \chi_\ff, \chi_\gi )$
has a pole at $s_3=1$ if the character $\chi_\mi\ov{\chi}_{\mi'}\chi_\gi$ appearing there is trivial. That is, $\gi=1$ and
$\mi$, $\mi'$ are $r$-th powers in \cref{prop: Zaux2}, and it follows that the product
in the series over the places dividing $\ei^{(1)}$ is zero unless $\ei^{(1)}=1$, namely $\ei$ is an $r$-power as well.
We conclude:
\[\begin{aligned}
\res_{w=1} \wZ^{T}(s,s,w; \chi_\ff, \chi_\ff, 1;G)&=\zeta_F^T(2s)\zeta_F^T(rs+1)^2\res_{w=1}\zeta_F^T(w)\frac{1}{|R_\ci|}
\sum_{\bi,\bi'\in I(T)}
\frac{1}{|\bi\bi'|^{rs}}\\
&\cdot\prod_{v\in S_{\mi\mi'}} (1-q_v^{-1})\prod_{v\in S_{\ei}\setminus S_{\mi\mi'}} (1-q_v^{-1})^2
\prod_{v\in S_{\ei}\cap S_{\mi\mi'}} (1-q_v^{-1}),
\end{aligned}
\]
where in the sum we write $\bi=\ei \mi$, $\bi'=\ei \mi'$ with $(\mi,\mi')=1$. The Dirichlet series above equals
$\zeta_F^T(rs)^2 / \zeta_F^T(rs+1)^2$, finishing the proof.

(b) Using the definition of $\wZ^{T}(\sb;\chi_\ff, \chi_\ff, \chi_\gi )$  in \S\ref{sec:wZ},
we apply the functional
equations to the two $L$-functions, as well as to the polynomials $Q_\ai$. Setting $s'=1-s$, $w'=w+2s-1$ we obtain:
\[\begin{aligned}
\wZ^{T}(s,s,w; \chi_\ff, \chi_\ff, \chi_\gi;G )&=\sum_{\substack{\ai\in I(T)\\ [\ai]=G}} \frac{L^{T}(s', \chi_{\ff\ai_0})
L^{T}(s', \ov{\chi}_{\ff\ai_0})}{|\ai|^{w'}}
	\chi_\gi(\ai)Q_\ai(s',s'; 1,1)\cdot \\
	&\cdot |\ff_1 D_F \cond(\psi_{[\ff G]})|^{1-2s} \prod_{v\in T}
	 \frac{L_v(1-s,\chi_{\ff \ai_0})L_v(1-s,\ov{\chi}_{\ff \ai_0})}{L_v(s,\chi_{\ff \ai_0})L_v(s,\ov{\chi}_{\ff \ai_0})}.
\end{aligned}
	\]
For $v\in S_f$, the quantity $\chi_{\ai_0}(\pp_v)=\chi_{\pp_v}(G)\psi_G(p_v)=\chi_{G}(\pp_v)$ is independent of $\ai$ with $[\ai]=G$,
where we used the definition of the character $\chi_{\pp_v}\in\widehat{R}_\ci$ in \S\ref{sec:cleaning}.
This gives the factor $A_{[\ff G]}(s)$ in the formula.

Setting $\hi'=\hi^{(\ff)}$, we can write the product of the local factors at $v\in S_\hi$ as follows:
\[
\prod_{v\in S_{\hi'}} \frac{1}{(1-q_v^{-rs'})^2} \sum_{\substack{\di,\di'|\hi'\\ \ei,\ei'|\hi'^{r-1} }}
\mu(\di) \mu(\di')\chi_{\ai_0\ff}(\di\ei)  \ov{\chi}_{\ai_0\ff}(\di'\ei')|\di\di'|^{-s}  |\ei\ei'|^{-s'}.
\]
Using the reciprocity law $\chi_\ai(\di\ei)=\chi_{\di\ei}(\ai)\psi_{G}(m_{[\di\ei]})$ we obtain
\[\begin{aligned}
\wZ^{T}(s,s,w; \chi_\ff, \chi_\ff, \chi_\gi;G )&=|\ff_1|^{1-2s} A_{[\ff G]}(s)
\prod_{v\in S_{\hi'}} \tfrac{1}{(1-q_v^{-rs'})^2}\sum_{\substack{\di,\di'|\hi'\\ \ei,\ei'|\hi'^{r-1} }}
\chi_\ff(\di\ei)\ov{\chi}_\ff(\di'\ei') \psi_G\big(\tfrac{m_{[\di\ei]}}{m_{[\di'\ei']}} \big)\cdot\\
&\cdot
\mu(\di) \mu(\di') |\di\di'|^{-s}  |\ei\ei'|^{-s'}
\wZ^{T}(s',s',w';\chi_\ff, \chi_\ff, \chi_{\gi\di\ei}\ov{\chi}_{\di'\ei'};G ).
\end{aligned}
\]
By part (a), the right-hand side has a pole at $w'=1$ only if $\gi\di\ei/(\di'\ei')$ is an $r$-th power, and the residue at $w'=1$ of the sum
above gives
\[
\frac{1}{|R_\ci|}\ov{\chi}_\ff(\gi)\ov{\psi}_G(m_{[\gi]})
\Rc^T(s') \sum_{\substack{\di,\di'|\hi',\ \ei,\ei'|\hi'^{r-1} \\  \gi\di\ei/(\di'\ei') \text{ $r$-th power} }}
\mu(\di) \mu(\di') |\di\di'|^{-s}  |\ei\ei'|^{-s'}.
\]
The remaining sum is an Euler product, and combining it with $\prod_{v\in S_{\hi'}} (1-q_v^{-rs'})^{-2}$ yields the factor $B_{\gi,\hi^{(\ff)}}(s)$ and finishes the proof of (b).
\end{proof}
\end{para}

\begin{para}\label{par:prinpartZh}
We now turn to the determination of the residues of $Z_\kappa(s,s,w;\hi)$.
Using the decomposition in \cref{prop: sieving} and part (a) of \cref{prop:Zresidue},
we conclude that $Z_\kappa(s,s,w;\hi)$ has a pole at $w=1$ with residue
\be\label{eq:resZhw1}
\res_{w=1}Z_\kappa(s,s,w; \hi)=\Rc^{S\cup S_{\hi}}(s)
\sum_{\ff|\hi^{r-1}} f_{\ff}^{(\kappa)}(s,s,1) g_{1,\hi^{(\ff)}}(s,s,1),
\ee
where we suppress the two trivial characters from the arguments of
$f_{\mathfrak f}^{(\kappa)}$ and $g_{1,\mathfrak h^{(\mathfrak f)}}$.
Setting $s'=1-s$, we similarly have
\[
\res_{w=2-2s} Z_\kappa (s, s,w; \hi)=\Rc^{S\cup S_{\hi}}(s')A(s)
\!\!\!  \sum_{\substack{\ff\mid \hi^{r-1} \\ \gi\mid (\hi^{(\ff)})^{r-1} }}
|\ff_1|^{1-2s}	 f_{\ff}^{(\kappa)}(s,s,2s') g_{\gi,\hi^{(\ff)}}(s,s, 2s')B_{\gi,\hi^{(\ff)}}(s),
\]
where the function
\be\label{eq:Afun}
A(s)=\frac{1}{|R_\ci|}\sum_{G\in \EE} A_{[\ff G]}(s)
\ee
is independent of $\ff$ and satisfies $A(1/2)=1$ (see the definition of $A_H$ in \cref{prop:Zresidue}).
As in~\S\ref{par:prinpart}, let
\[
\Ac_\kappa(s;\hi)=(2s-1)\res_{w=1}Z_\kappa(s,s,w; \hi),\quad
\Bc_\kappa(s;\hi)=(2s-1)\res_{w=2-2s}Z_\kappa(s,s,w; \hi).
\]
We will show in \cref{lem:funeq} that $\Bc_\kappa(s;\hi)=-\Ac_\kappa(1-s;\hi)A(s)$,
and this relation and the observation in~\S\ref{par:prinpart} prove the following result.
\end{para}

\begin{prop}\label{prop:prinpartZh}
The principal part of $Z_\kappa(\tfrac 12,\tfrac 12,w;\hi)$ at $w=1$
is given by
\[\frac{\Ac_\kappa(\tfrac 12;\hi)}{(w-1)^2}+\frac{\Ac_\kappa'(\tfrac12;\hi )-\tfrac12  \Ac_\kappa(\tfrac12 ;\hi) A'(\tfrac12)}{w-1},
\]
with $A(s)$ given in \eqref{eq:Afun} and
\[\Ac_\kappa(\tfrac 12;\hi)=\zeta_F^{S\cup S_{\hi}}(r/2)^2  [\res_{w=1}\zeta_F^{S\cup S_{\hi}}(w)]^2
\sum_{\ff|\hi^{r-1}} f_{\ff}^{(\kappa)}(\tfrac12,\tfrac12,1) g_{1,\hi^{(\ff)}}(\tfrac12,\tfrac12,1).
\]
\end{prop}

\begin{lem} \label{lem:funeq}
Setting $s'=1-s$ we have
\[
\sum_{\substack{\ff\mid \hi^{r-1} \\ \gi\mid (\hi^{(\ff)})^{r-1} }}
	|\ff_1|^{1-2s} f_{\ff}^{(\kappa)}(s,s,2s') g_{\gi,\hi^{(\ff)}}(s,s, 2s')B_{\gi,\hi^{(\ff)}}(s)=
	\sum_{\ff\mid \hi^{r-1}} f_{\ff}^{(\kappa)}(s',s',1) g_{1,\hi^{(\ff)}}(s',s', 1).
\]
\end{lem}
\begin{proof} Both sides are Euler products, so we have to prove the equality of their $v$-parts for $v\in S_\hi$:
\[
\sum_{\eps=1}^{r-1} q x^2 f_\kappa^{(\eps)} (x,x,y^2) +
\sum_{\eps=0}^{r-1} g^{(\eps)} (x,x,y^2) B^{(\eps)}(x,y)=
\sum_{\eps=1}^{r-1} f_\kappa^{(\eps)} (y,y,q^{-1})+ g^{(0)} (y,y,q^{-1}),
\]
where $q=q_v$, $x=q^{-s}$, $y=q^{-s'}=(qx)^{-1}$ and
\[
B^{(\eps)}(x,y)=\frac 1r \sum_{\xi\in\mu_r} \xi^{\eps}\frac{(1-x\xi)(1-x\ov{\xi})}{(1-y\xi)(1-y\ov{\xi})}.
\]
We omit the last argument $q=q_v$ from the notation of the $v$-part and derived functions, as $v$ is fixed.

Decompose the $v$-part $f$ in \S\ref{sec:v-part} as $f=f^++f^-$, where $f^+$ is the part of $f$ containing the powers
$y^k$ with $k\equiv 0 \pmod r$. The polynomials $Q_k$ satisfy the functional equation
\[
Q_k(x_1, x_2;q) =Q_k((qx_1)^{-1}, (qx_2)^{-1};q) \cdot
\begin{cases}
(qx_1x_2)^{k-1} & \text{ if $k\not\equiv 0 \pmod r$}\\
(qx_1x_2)^{k} & \text{ if $k\equiv 0 \pmod r$},
\end{cases}
 \]
which yields the following functional equations for $f^\pm$:
\[
f^-\big(\tfrac 1{qx_1}, \tfrac 1{qx_2}, qx_1x_2y\big)=qx_1 x_2 f^-(x_1,x_2,y),
\]
\[
f^+\big(\tfrac 1{qx_1}, \tfrac 1{qx_2}, qx_1x_2y\big)=\frac{(1-x_1)(1-x_2)}{(1-(qx_1)^{-1})(1-(qx_2)^{-1})} f^+(x_1,x_2,y).
\]
Note that
\[
\sum_{\eps=1}^{r-1} f_\kappa^{(\eps)} (x_1,x_2,y)=f^-(x_1,x_2,y)-
\begin{cases}
 y & \text{ if $\kappa=1$} \\
  \sum_{k=1}^{r-1} P_k(x_1,x_2;q)y^k & \text{ if $\kappa=0$}.
\end{cases}
\]
so $\sum_{\eps=1}^{r-1} f_\kappa^{(\eps)}$ satisfies the same functional equation as $f^-$.
Therefore the first sums on both sides in the $v$-part identity match.

To check that the terms containing $g$ also match,  note that $g=f^+-\tfrac{1}{(1-x_1)(1-x_2)}$ also satisfies the functional
equation satisfied by $f^+$. We have
\[\begin{aligned}
\sum_{\eps=0}^{r-1} g^{(\eps)} (x,x,y^2)B^{(\eps)}(x,y)
&=\sum_{\eps=0}^{r-1}\frac{1}{r^2}\sum_{\xi,\xi'\in \mu_r}
\xi^{-\eps}\xi'^{\eps}g (\xi x,\ov{\xi}x,y^2)\frac{(1-x\xi')(1-x\ov{\xi}')}{(1-y\xi')(1-y\ov{\xi}')} \\
&=\sum_{\eps=0}^{r-1}g^{(\eps)} (y,y,q^{-1}) B^{(\eps)}(y,y),
\end{aligned}
\]
where in the second sum only the terms with $\xi=\xi'$ survive, and the second equality follows from the functional equation for $g$. Since $B^{(\eps)}(y,y)$ vanishes unless $\eps=0$, when it equals 1, the last sum equals $g^{(0)} (y,y,q^{-1})$
finishing the proof.
\end{proof}

\subsection{Principal part of the sieved MDS}\label{par:res_Zsieved}
\begin{para}
We combine~\eqref{eq:Zkappa} and \cref{prop:prinpartZh} to compute the principal part of the sieved
second moment MDS $Z_\kappa^S(\tfrac12,\tfrac12,w)$ over $r$-th power-free ideals ($\kappa=0$),
or over square-free ideals ($\kappa=1$). Recall the residue
$\Rc^S(s)=\zeta_F^S(2s)\zeta_F^S(rs)^2  \res_{w=1}\zeta_F^S(w)$ appearing in \cref{prop:Zresidue}.
\end{para}
\begin{prop}\label{prop:prinpartZsieved}
 The principal part of $Z_\kappa^S(\tfrac12,\tfrac12,w)$ at $w=1$ is given by
 \[
 \frac{\Ac_\kappa(\tfrac 12)}{(w-1)^2}+\frac{\Ac_\kappa'(\tfrac12)-\tfrac12  \Ac_\kappa(\tfrac12) A'(\tfrac12)}{w-1},
 \]
 where $A(s)$ is given by~\eqref{eq:Afun} and
 \[
 \Ac_\kappa(s)=(2s-1)\Rc^{S}(s) \prod_{v\notin S}(1-q_v^{-rs})(1-q_v^{-1}) P_r^{(\kappa)}(q_v^{-s};q_v),
 \]
 for polynomials $P_r^{(\kappa)}(x;q)$ given by
 \[
P_r^{(\kappa)}(x;q)=\begin{cases}
(1-x^r)\left(\frac{1-q^{-r}}{1-q^{-1}}-q^{1-r}x^2\frac{1-(qx^2)^{r-1}}{1-qx^2}\right)+2x^r & \text{ if $\kappa=0$}\\
               1+q^{-1}-x^2q^{-1}+x^r(1-q^{-1}+x^2q^{-1})      & \text{ if $\kappa=1$}.
                    \end{cases}
\]
Moreover, the Euler product in the definition of $\Ac_\kappa(s)$ converges absolutely for $\Re(s)\ge 1/2$.
\end{prop}
\begin{proof}
Define
\[
\Ac_\kappa(s)= \sum_{\hi\in I(S)} \mu(\hi)\Ac_\kappa(s;\hi)=(2s-1)\sum_{\hi\in I(S)} \mu(\hi) \res_{w=1} Z_\kappa(s,s,w;\hi) ,
\]
with the residues given by~\eqref{eq:resZhw1}; we will show shortly that this agrees with the definition in the statement of the proposition.
The decomposition in~\eqref{eq:Zkappa} and the meromorphic
continuation of $Z_\kappa^S(\tfrac12,\tfrac12,w)$ established in Section~\ref{sec:analytic} prove the formula
for the principal part, taking into account the formula for the principal part of
$Z_\kappa(\tfrac12,\tfrac12,w;\hi)$ in \cref{prop:prinpartZh}.

To prove the Euler product formula for $\Ac_\kappa(s)$, we use~\eqref{eq:resZhw1} to write
 \[\begin{aligned}
\Ac_\kappa(s)=(2s-1)\Rc^{S}(s)
\sum_{\hi\in I(S)} \mu(\hi)&\prod_{v\in S_\hi}(1-q_v^{-1})(1-q_v^{-2s}) (1-q_v^{-rs})^2 \\
&\cdot\sum_{\ff\mid \hi^{r-1}} f_{\ff}^{(\kappa)}(s, s, 1)
g_{1,\hi^{(\ff)}}(s, s, 1),
\end{aligned}
\]
The sum over $\hi\in I(S)$ equals an Euler product with local factor at $v\notin S$:
\[
1-(1-q_v^{-1})(1-q_v^{-2s}) (1-q_v^{-rs})^2\left(\sum_{\eps=1}^{r-1}f_\kappa^{(\eps)}(q_v^{-s},q_v^{-s},q_v^{-1};q_v)+
g^{(0)}(q_v^{-s},q_v^{-s},q_v^{-1};q_v)   \right)
\]
Using the formula for the $v$-part $f$ in \S\ref{sec:v-part}, this local factor equals
\[
(1-q_v^{-rs})(1-q_v^{-1}) P_r^{(\kappa)}(q_v^{-s};q_v),
\]
with $P_r^{(\kappa)}(x;q)$ as given above.

The convergence of $\Ac_\kappa(s)$ follows from the explicit formula for its Euler factors.
\end{proof}

\begin{para}
By the Tauberian argument discussed in more detail in the next section,
we obtain that the Euler product in the leading term in Theorems~\ref{Main-Thm1-alt}
and~\ref{Main-Thm2-r-th-power-free-alt} is given by
\be\label{eq:eulerprod}
\lim_{w\rightarrow 1} (w-1)^2 Z_\kappa^S\big(\tfrac 12,\tfrac 12,w\big)=\zeta_F^S(r/2) \big[\res_{w=1} \zeta_F^S(w)\big]^2
\prod_{v\notin S} (1-q_v^{-1})P_r^{(\kappa)}\big(q_v^{-1/2}\big),
\ee
where $P_r^{(\kappa)}\big(q_v^{-1/2}\big):= P_r^{(\kappa)}(q_v^{-1/2};q_v)$, with the polynomials $P_r^{(\kappa)}(x)$ given by
\be \label{eq:euler_factors_2ndmom}
P_r^{(\kappa)}(x)=\begin{cases}(1-x^{r})\left(\frac{1-x^{2r}}{1-x^2}-(r-1)x^{2r}\right)+2x^r & \text{ if $\kappa=0$}\\
               1+x^2-x^4+x^r(1-x^2+x^4)      & \text{ if $\kappa=1$}.
                    \end{cases}
\ee
Note that $P_r^{(\kappa)}(x)>0$ for $0<x<1$, and that
 $(1-q^{-1})P_r^{(\kappa)}(q^{-1/2};q)
 =1+O\bigl(q^{-\min(2,r/2)}\bigr)$, so the Euler products above converge and are positive.

The entire principal part will be computed in more generality in \cref{prop:Zbres} (taking $\bi=1$ there).
\begin{rem}
For $r=2$, we have $P_2^{(0)}(x)=P_2^{(1)}(x)=1+2x^2-2x^4+x^6$, and the Euler product
above--completed with a local factor $(1-1/q_v)$ to account for the factor $\zeta(r/2)$ removed above--matches
the leading term in the asymptotic formula for the second moment of quadratic Dirichlet $L$-functions over $\Q$
\cite[eq. (1.3.5)]{CFKRS}.

For $r=3,4$, $P_r^{(1)}(x)$ matches the local Dirichlet polynomial appearing in the asymptotic formulas in the cubic case~\cite{DFDS24}
and in the quartic case~\cite{CFD26}.
\end{rem}
\end{para}

\section{Mollification}\label{sec:mollification}
\begin{para}
Fix $M$ to be chosen later in terms of $X$.
For $\ai\in I(S)$ we consider a mollifier
\[
M_\ai=\sum_{\substack{\bi\in I(S)_0 \\ |\bi|\le M}} \lambda_\bi|\bi|^{1/2}\chi_\ai(\bi),
\]
where $\lam_\bi\in \mathbb{R}$ will be chosen such that $\lambda_\bi\ll |\bi|^{-1+\eps}$.
We start with a mollifier supported on $r$-th power-free ideals, which in principle would lead to
a better proportion of non-vanishing.
Later on we will specialize as in~\cite{DFDS24} to a mollifier supported
on square-free ideals, as the technical complications required by the more general setting
do not seem to justify the modest gains in the non-vanishing proportion.

Recall that $I(S)_\kappa\subset I(S)$ denotes the subset of square-free ideals for $\kappa=1$,
and $r$-th power-free ideals for $\kappa=0$.
For $\kappa=0,1$, we are interested in the mollified first and second moments
\[
\Sc_1^{(\kappa)}(M;W)=\sum_{\ai \in I(S)_\kappa} L^S(\tfrac 12,\chi_\ai) M_\ai W\big(\tfrac{|\ai|}{X}\big),\]
\[
\Sc_2^{(\kappa)}(M;W)=\sum_{\ai \in I(S)_\kappa} |L^S(\tfrac 12,\chi_\ai) M_\ai|^2
Q_\ai W\big(\tfrac{|\ai|}{X}\big),
\]
where $Q_\ai = \prod_{\pp^{n_\pp}\|\ai} n_\pp$,
and $0\le W(x)\le 1$ is a Schwartz function supported on $[1,2]$. The mollifier $M_\ai$,
its length $M$ and the coefficients $\lam_\bi$ will depend on $\kappa$, but for simplicity
we omit this dependence until we choose the mollifier in \S\ref{sec:sqfreemol}.

We start with the second mollified moment, as the analysis for the first moment is very similar,
requiring a specialization of the perfect MDS from the second moment.
\end{para}
\subsection{Second mollified moment}
\begin{para}
Denoting by $\wW(s)=\int_0^\infty W(x)x^{s-1}dx$ the Mellin transform of $W$, by Mellin inversion we have
\be\label{eq:mol2ndmom}
\Sc_2^{(\kappa)}(M;W)=\sum_{\substack{\bi,\bi'\in I(S)_0 \\ |\bi|,|\bi'|\le M}}
\lam_\bi\lam_{\bi'}|\bi\bi'|^{1/2}
\frac{1}{2\pi i}\int_{(2)}\wW(w)Z_\kappa^S(\tfrac 12, \tfrac 12, w;\chi_{\bi\ov{\bi'}})X^w dw
\ee
where we define a twisted version of the MDS $Z_\kappa^S(\sb):=Z_\kappa^S(\sb;1,1,1)$ from~\eqref{eq:Zsieved}:
\[
Z_\kappa^S(\sb;\chi_\bi):=
	\sum_{\ai\in I(S)_\kappa}\frac{
	L^{S}(s_1,  \chi_{\mathfrak{a}})
	L^{S}(s_2, \overline{\chi}_{\mathfrak{a}})}
	{|\mathfrak{a}|^{s_3}} \chi_\ai(\bi)Q_\ai(s_1,s_2),
\]
for $\bi\in I(S)$ an $(r+1)$-th power-free ideal, and
$Q_\ai(s_1,s_2)=Q_\ai(s_1,s_2;1,1)$. Note that $Q_\ai=Q_\ai(1/2,1/2)$ in the definition
of~$\Sc_2^{(\kappa)}$,
and recall that $\ov{\bi}=(\rad \bi)^r/\bi$, so that $\ov{\chi}_\bi=\chi_{\ov{\bi}}$.

Here and in the sequel we replace the ideal $\ff=\bi \ov{\bi'}$ by the $(r+1)$-th power-free ideal
\be\label{eq:fstar}
\ff^*:=\ff_0 \prod_{v\in S_\ff \setminus S_{\ff_0}} \pp_v^r,
\ee
so that $\chi_{\ff}=\chi_{\ff^*}$ and $\chi_\ai(\ff)=\chi_\ai(\ff^*)$ for all $\ai\in I(S)$
(note that this differs from the definition $\ji^*$ in \S\ref{eq:step2}).
\end{para}

\begin{para} \label{par:Zbsieving}
Next we study the analytic properties of $Z_\kappa^S(\sb;\chi_\bi)$, following
\S\ref{sec:sieving}. For $\hi\in I(S)$ square-free with $(\hi,\bi)=1$, we define
	\[
	Z_\kappa(\sb; \chi_\bi;  \mathfrak{h}) = \sum_{\substack{\mathfrak{a} \in I(S\cup S_\bi) \\ \mathfrak{h} \mid  \ai\ai_\kappa^{-1}}}
	\frac{L^{S}(s_1, \chi_{\mathfrak{a}_{0}}) L^{S}(s_2, \overline{\chi}_{\mathfrak{a}_{0}})}{|\mathfrak{a}|^{s_3}}
	\chi_\ai(\bi)Q_\ai(s_1, s_2),
	\]
so that we have
\be\label{eq: sieving_mol}
Z_\kappa^S(\sb;\chi_\bi)=
	\sum_{\mathfrak{h} \in I(S\cup S_\bi)}\mu(\mathfrak{h}) Z_\kappa(\sb; \chi_\bi;  \mathfrak{h}).
\ee
We have the following decomposition $Z_\kappa(\sb; \chi_\bi;  \hi)$
in terms of the perfect MDS $\wZ^{S\cup S_{\bi\hi}}$, extending \cref{prop: sieving}.
With the notation in \S\ref{sec:sieving}, we set  $f_\ai^{(\kappa)} (\sb):=f_\ai^{(\kappa)} (\sb;1,1)$,
	 $g_{\ai,\bi} (\sb):=g_{\ai,\bi} (\sb;1,1)$ to simplify notation
\begin{prop}\label{prop: sieving2}
Let $\hi\in I(S)$ be square-free and $\bi\in I(S)$ be $(r+1)$-th power-free with $(\hi,\bi)=1$.
For $\kappa=0,1$, we have the identity
	\[\begin{aligned}
	 Z_\kappa(\sb; \chi_\bi;  \mathfrak{h})&=\prod_{v\in S_\bi}\frac{1}{(1-q_v^{-rs_1})(1-q_v^{-rs_2})}
	 \sum_{\di,\di'| (\rad \bi)^{r-1}} \frac{1}{|\di|^{s_1}|\di'|^{s_2}}
	 \sum_{\substack{\ff,\gi\mid \hi^{r-1} \\ (\ff,\gi)=1}}f_\ff^{(\kappa)} (\sb)g_{\gi,\hi^{(\ff)}} (\sb)\cdot \\
	 &\cdot\frac{1}{|R_\ci|}\sum_{\substack{G\in \EE\\\eta\in\widehat{R}_\ci }} \chi_\ff(\bi\gi\di)\ov{\chi}_\ff(\di') \psi_G\big(\tfrac{m_{[\bi\gi\di]}}{m_{[\di']}}\big)\ov{\eta}(G)
	 \wZ^{S\cup S_{\bi\hi}} (\sb,\chi_\ff,\chi_\ff, \eta\chi_{\bi\gi\di}\ov{\chi}_{\di'}).
\end{aligned}
	 \]
\end{prop}
\begin{proof}
 The proof is similar to that of \cref{prop: sieving}, replacing the character $\rho(\ai)$ there  by
 $\chi_\ai(\bi)=\chi_\bi(\ai)\psi_{[\ai]}(m_{[\bi]})$.
\end{proof}
\end{para}
\subsection{Convexity bound}\label{sec:convexity_bh}
\begin{para}

We first record the uniform version of the convexity estimate for the perfect
MDS needed in \cref{prop: sieving2}.  Let $\bi$ be $(r+1)$-th power-free, to be thought of as
the $(r+1)$-th power-free part of $\bi\di\ov{\di'}$ in the proposition.
Let $\hi\in I(S\cup S_\bi)$ be square-free, and let
$\ff,\gi\mid \hi^{r-1}$ with $(\ff,\gi)=1$.

Letting $\sb_w=\big(\tfrac12,\tfrac12,w\big)$, we have
uniformly for $1/2<\sigma=\Re(w)<1$:
\begin{equation}\label{eq:twisted-perfect-convexity}
\begin{aligned}
(w-1)^2&\big(w-\tfrac12-\tfrac1r\big)^2\,
 \wZ^{S\cup S_{\bi\hi}}
 \big(\sb_w;\chi_\ff,\chi_\ff,
       \eta\chi_{\bi\gi}\big)  \\
 &\ll_\eps
 |\rad\bi|^{(1/2+r-1)(1-\sigma)+\eps}
 |\hi|^{1/2+(r-1)(1-\sigma)+\eps}
 (1+|w|)^{4+rd(1-\sigma)+\eps} .
\end{aligned}
\end{equation}
For the case $\kappa=0$ we also need the refined estimate,
and the same argument in \S\ref{sec:refined_convexity} gives
\begin{equation}\label{eq:twisted-perfect-refined}
\begin{aligned}
 (w-1)^2&\big(w-\tfrac12-\tfrac1r\big)^2\,
 \wZ^{S\cup S_{\bi\hi}}
 \big(\sb_w;\chi_\ff,\chi_\ff,
       \eta\chi_{\bi\gi}\big)  \\
 &\ll_\eps
 |\rad\bi|^{(1/2+r-1)(1-\sigma)+\eps}
 |\hi|^{(r-1)(1-\sigma)+\eps}|\rad(\ff\gi)|^{1/2}
 (1+|w|)^{4+rd(1-\sigma)+\eps} .
\end{aligned}
\end{equation}
To prove the first estimate,
note that on the line of
absolute convergence $\Re(w)=1+\eta$, the Euler factors at $S_\bi$ contribute
$|\rad\bi|^\eta$, and the remaining Dirichlet series is bounded independently of $\bi$.
On the line  $\Re(w)=-\eta$, we have a bound of
$|\hi \rad\bi|^{1/2+r-1+\eta} $ as in the proof of  \cref{prop:convexity}, combined with
\cref{cor:convexity}. The Phragm\'en-Lindel\"of principle therefore gives the bound in
\eqref{eq:twisted-perfect-convexity}.

This also proves the
refined form \eqref{eq:twisted-perfect-refined}, since the induction in
\cref{prop:refined-convexity} is unchanged, involving
the primes dividing $\hi$ and not $\ff,\gi$.
\end{para}

\subsection{Analytic continuation}

\begin{para}
Using \cref{prop: sieving2} and the convexity bounds for the perfect MDS we obtain the following
twisted version of \cref{prop:convergence_sievedMDS}, with the same region of convergence.
\begin{prop}\label{prop:convergence_sievedMDS_twisted}
Let $\bi\in I(S)$ be $(r+1)$-th power-free.
The MDS $Z_\kappa^S (\tfrac 12, \tfrac 12, w;\chi_\bi)$ has meromorphic continuation
to the region
\[
\Re (w)>\dd_\kappa,\qquad
\dd_\kappa=\begin{cases} \frac{r+A_r}{r+1}  & \text{if } \kappa=1 \\[3pt]
		  \frac{3r-2}{4r-2}  & \text{if } \kappa=0,
         \end{cases}
\]
with a double pole at $w=1$ and a possible second double pole at $w=1/2+1/r$.

Moreover, for small $\eps>0$ we have the uniform estimate for  $\dd_\kappa+\eps\le \Re (w)=\sigma\le 1$:
 \be\label{eq:convexZb}
(w-1)^2(w-\tfrac12-\tfrac1r)^2 Z_\kappa^S (\tfrac 12, \tfrac 12, w; \chi_\bi)
\ll_\eps |\rad{\bi}|^{(1/2+r-1)(1-\sigma)+\eps}(1+|w|)^{4+ rd(1-\sigma)+\eps}.
 \ee
\end{prop}
Note that the region above contains the second double pole only
for $\kappa=0$ and $r=3,4$.
\begin{proof}
Let $\sb_w=\big(\tfrac12,\tfrac12,w\big)$.
We apply \cref{prop: sieving2} and the estimates in subsection~\ref{sec:convexity_bh}.
For $\hi\in I(S\cup S_\bi)$ square-free, the factors depending only on $\bi$ are
 \[
  \prod_{v\in S_\bi}
  \frac{1}{\big(1-q_v^{-r/2}\big)^2}
  \sum_{\di,\di'\mid (\rad\bi)^{r-1}} |\di\di'|^{-1/2}
  \ll_\eps |\rad(\bi)|^\eps .
\]

For $\kappa=1$, we
use \eqref{eq:twisted-perfect-convexity}, and the
local estimates for $f_\ff^{(\kappa)}$ and $g_{\gi,\hi^{(\ff)}}$ from
\S\ref{sec:fg_estimates} to obtain,
\begin{equation}\label{eq:Zbh-kappa1-bound}
\begin{aligned}
 (w-1)^2&\big(w-\tfrac12-\tfrac1r\big)^2 Z_1(\sb_w;\chi_\bi;\hi)
 \ll_\eps\\
 &|\rad\bi|^{(1/2+r-1)(1-\sigma)+\eps}
 |\hi|^{1/2+(r-1)(1-\sigma)-2\sigma+\eps}
 (1+|w|)^{4+rd(1-\sigma)+\eps}.
\end{aligned}
\end{equation}
For $r=3$, the averaging argument used in \cref{prop:conv_req3}
applies verbatim to yield the larger region of convergence involving $A_3=1/3$
for the sieved series in~\eqref{eq: sieving_mol}.

For $\kappa=0$, using instead the refined perfect MDS estimate
\eqref{eq:twisted-perfect-refined}, exactly as in the proof of the untwisted
case in \S\ref{par:kappa0region}, yields the same region of convergence as there.

The convexity bound in $\bi$ and $|w|$ follows from~\eqref{eq:twisted-perfect-convexity}
and~\eqref{eq:twisted-perfect-refined}, finishing the proof.
\end{proof}
In the next section, we will see  that the pole at $w=1$ is genuinely of order two,
and we will compute its principal part.
\end{para}

\subsection{Residues}

\begin{para}\label{par:resw1}
We now extend the principal part computations from Section~\ref{sec:residues} to the twisted MDS considered here.

Proposition \ref{prop: sieving2} implies that $Z_\kappa(s,s,w; \chi_\bi;\hi)$
has a pole at $w=1$, coming from the terms $\wZ^{S\cup S_{\bi\hi}}$ for which $\gi=1$ and $\bi\di/\di'$ an $r$-th power.
We obtain
\be\label{eq:resw1}
\res_{w=1}Z_\kappa(s,s,w; \chi_\bi;\hi)=\Rc^{S\cup S_{\bi\hi}}(s)B_\bi(s)
\sum_{\ff|\hi^{r-1}} f_{\ff}^{(\kappa)}(s,s,1) g_{1,\hi^{(\ff)}}(s,s,1),
\ee
where we recall that $\Rc^T(s)=\zeta_F^T(2s)\zeta_F^T(rs)^2  \res_{w=1}\zeta_F^T(w)$
and we define
\be\label{eq:Bbs}
B_\bi(s):=\prod_{v\in S_\bi}(1-q_v^{-rs})^{-2}\sum_{\substack{\di,\di'| \rad(\bi)^{r-1}\\\bi\di/\di' \text{ $r$-th power} }} |\di\di'|^{-s}=
\prod_{v\in S_\bi}\frac{q_v^{-su_v}+q_v^{-s(r-u_v)}}{(1-q_v^{-rs})(1-q_v^{-2s})},
\ee
with $u_v=\ord_v(\bi)$ (the identity is a consequence of~\eqref{eq:mur_sum} for $x=0$).
\end{para}

\begin{para}\label{par:resw'1}
By \cref{prop: sieving2} and part (b) of \cref{prop:Zresidue}, we also have (setting $s'=1-s$)
\[\begin{aligned}
\res_{w=2-2s} Z_\kappa &(s, s,w; \chi_\bi;\hi)=\Rc^{S\cup S_{\bi\hi}}(s')
\prod_{v\in S_\bi}(1-q_v^{-rs})^{-2}\sum_{\di,\di'| \rad(\bi)^{r-1}} |\di\di'|^{-s}\\
	 &\cdot  \sum_{\substack{\ff\mid \hi^{r-1} \\ \gi\mid (\hi^{(\ff)})^{r-1} }}
|\ff_1|^{1-2s}	 f_{\ff}^{(\kappa)}(s,s,2s') g_{\gi,\hi^{(\ff)}}(s,s, 2s')A(s)B_{\bi\di\ov{\di'}\gi,\bi\hi^{(\ff)}}(s),
\end{aligned}
\]
where $A(s)$ is given in~\eqref{eq:Afun} and $B_{\ti,\ai}(s)=B_{\ti_0,\ai}(s)$ is defined as in \cref{prop:Zresidue}
for $\ti$ not $r$-th power free. The function $B_{\bullet,\bullet}(s)$ is multiplicative and
\[\prod_{v\in S_\bi}(1-q_v^{-rs})^{-2}
	 \sum_{\di,\di'| \rad(\bi)^{r-1}} |\di\di'|^{-s} B_{\bi\di\ov{\di'},\bi}(s)=B_\bi(s'),\]
with $B_\bi(s)$ defined in~\eqref{eq:Bbs}. We obtain
\be\label{eq:resw'1}
\begin{aligned}
\res_{w=2-2s} Z_\kappa & (s, s,w; \chi_\bi;\hi)=\Rc^{S\cup S_{\bi\hi}}(s')B_\bi(s') A(s)\\
&\cdot  \sum_{\substack{\ff\mid \hi^{r-1} \\ \gi\mid (\hi^{(\ff)})^{r-1} }}
	|\ff_1|^{1-2s} f_{\ff}^{(\kappa)}(s,s,2s') g_{\gi,\hi^{(\ff)}}(s,s, 2s')B_{\gi,\hi^{(\ff)}}(s).
\end{aligned}
\ee
As in~\S\ref{par:prinpart}, define
\[
\Ac_\kappa(s;\bi,\hi)=(2s-1)\res_{w=1}Z_\kappa(s,s,w;\chi_\bi; \hi),\quad
\Bc_\kappa(s;\bi,\hi)=(2s-1)\res_{w=2-2s}Z_\kappa(s,s,w; \chi_\bi;\hi).
\]
We note that by \cref{lem:funeq} we have as before $\Bc_\kappa(s;\bi,\hi)=-\Ac_\kappa(1-s;\bi,\hi)A(s)$.
We obtain as in \S\ref{par:prinpartZh} that the principal part of $Z_\kappa(\tfrac 12, \tfrac 12,w; \chi_\bi;\hi)$
at $w=1$ is given by
\[
\frac{\Ac_\kappa(\tfrac 12;\bi,\hi)}{(w-1)^2}+\frac{\Ac_\kappa'(\tfrac12;\bi,\hi )-\tfrac12  \Ac_\kappa(\tfrac12; \bi,\hi) A'(\tfrac12)}{w-1}.
\]
\end{para}

\begin{para}
We now determine the principal part of $Z_\kappa^S (\tfrac 12, \tfrac 12,w; \chi_\bi)$
at $w=1$. For an $(r+1)$-th power-free ideal $\bi\in I(S)$, let
\[
\Ac_\kappa^\bi(s)= \sum_{\hi\in I(S\cup S_\bi)} \mu(\hi)\Ac_\kappa(s;\bi,\hi)
=(2s-1)\sum_{\hi\in I(S\cup S_\bi)} \mu(\hi) \res_{w=1} Z_\kappa(s,s,w;\bi,\hi) .
\]
We obtain from~\eqref{eq:resw1} and~\eqref{eq:resZhw1} that
\[
\Ac_\kappa^\bi(s)=\Ac_\kappa(s)\prod_{v\in S_\bi} \tfrac{q_v^{-\ord_v(\bi)s}+q_v^{-(r-\ord_v(\bi))s}}{P_r^{(\kappa)}(q_v^{-s};q_v)},
\]
where $\Ac_\kappa(s)$ is given in~\cref{prop:prinpartZsieved}.

We obtain the following generalization of \cref{prop:prinpartZsieved}, whose proof is identical using the computations above.
Let $\dlog f=f'/f$ denote the logarithmic derivative of a function $f(s)$.
\begin{prop} \label{prop:Zbres}
(a) Let $\bi\in I(S)$ be $(r+1)$-th power-free.
With the notation above, the principal part of $Z_\kappa^S (\tfrac12,\tfrac12,w; \chi_\bi)$ at $w=1$ is
\[
\frac{\Ac_\kappa^\bi(\tfrac12)}{(w-1)^2}+\frac{\Ac_\kappa^\bi(\tfrac12)}{w-1}
\big(\dlog\Ac_\kappa^\bi(\tfrac12)  + K_0  \big)
\]
where $K_0=-A'(1/2)/2$.

(b) Assume $\bi=\ei \mi, \bi' =\ei\mi'$ are $r$-th power-free with $\mi$, $\mi'$ coprime. For
$\ff=(\bi\ov{\bi'})^*$ defined in~\eqref{eq:fstar} we have
\[
\Ac_\kappa^\ff(\tfrac 12)= D_\kappa g_\kappa(\ei^{(\mi\mi')}) \frac{h_\kappa(\mi\mi')}{|\mi\mi'|^{1/2}},
\]
where
\[
D_\kappa=\Ac_\kappa(\tfrac12) =
\zeta_F^S(r/2) \big[\res_{w=1} \zeta_F^S(w)\big]^2
\prod_{v\notin S} (1-q_v^{-1})P_r^{(\kappa)}\big(q_v^{-1/2}\big),
\]
\[
g_\kappa(\ei)=\prod_{v\in S_\ei} \frac{1+q_v^{-r/2}}{P_r^{(\kappa)}\big(q_v^{-1/2}\big)}, \quad
h_\kappa(\mi)=\prod_{v\in S_\mi} \frac{1+q_v^{-(r-2\ord_v(\mi))/2}}{P_r^{(\kappa)}\big(q_v^{-1/2}\big)},
\]
and
\[
\dlog \Ac_\kappa^\ff(\tfrac 12)=-\log \prod_{v\in S_{\mi\mi'}} q_v^{\min (u_v,r-u_v)}+K_1+\sum_{v\in S_{\ei^{(\mi\mi')}}} K_v \frac{\log q_v}{q_v} +
\sum_{v\in S_{\mi\mi'}} K_v' \frac{\log q_v}{q_v^{1/2}}
\]
where $K_1$ is a constant, $K_v, K_v'$ are uniformly bounded and $u_v=\ord_v(\mi\mi')$.

\end{prop}
\end{para}
\subsection{Twisted second moment asymptotics}
\begin{para} From the region of meromorphic continuation in \cref{prop:convergence_sievedMDS_twisted} and the residue
computation in~\cref{prop:Zbres}, we obtain the following asymptotic formula. It captures the secondary term only for $\kappa=0$ and $r=3,4$.

\begin{thm} \label{thm:twisted_2ndmom}
Fix $\kappa\in\{0,1 \}$. Let $\bi\in I(S)$ be $r$-th power-free and
let $W(u)$ be a Schwartz function with compact support on $(1,2)$ satisfying $0\le W(u) \le 1$.
Then, for every $\varepsilon > 0$, we have
\[\begin{aligned}
\sum_{\ai \in I(S)_\kappa}  \big|L^S(\tfrac{1}{2},& \chi_{\ai})\big|^2 \chi_\ai(\bi)Q_\ai W\big(\tfrac{|\ai|}{X}\big)
=X R_{W,\bi}^{(\kappa)}(\log X)  \\
&+X^{\frac{1}{2}+\frac{1}{r}} S_{W,\bi}^{(\kappa)}(\log X)
+O\Big(X^{\dd_\kappa+\varepsilon} |\rad{\bi}|^{(1/2+r-1)(1-\dd_\kappa)+\varepsilon} \Big),
\end{aligned}
\]
for some explicit linear polynomials $R_{W,\bi}^{(\kappa)}$, $S _{W,\bi}^{(\kappa)}$.
The leading coefficient of $R_{W,\bi}^{(\kappa)}$ is
\[
\wW(1)\zeta_F^S(r/2) \big[\res_{w=1} \zeta_F^S(w)\big]^2
\prod_{v\notin S} (1-q_v^{-1})P_r^{(\kappa)}\big(q_v^{-1/2}\big)
\prod_{v\in S_\bi} \tfrac{q_v^{-\ord_v(\bi)/2}+q_v^{-(r-\ord_v(\bi))/2}}{P_r^{(\kappa)}(q_v^{-1/2})}.
\]
Here $\dd_\kappa$ is given explicitly in \cref{prop:convergence_sievedMDS_twisted},
$P_r^{(\kappa)}$ is given in~\eqref{eq:euler_factors_2ndmom}, and  $Q_\ai:=\prod_{v\in S_\ai} \ord_v(\ai)$ (so $Q_\ai=1$ if $\kappa=1$).
\end{thm}
\begin{proof}
By Mellin inversion, the moment sum equals
\[
\frac{1}{2\pi i}\int_{(2)}\wW(w)Z_\kappa^S(\tfrac 12, \tfrac12, w;\chi_{\bi})X^w dw.
\]
By the meromorphic continuation of $Z_\kappa^S(\tfrac 12, \tfrac12,w;\chi_{\bi})$ in \cref{prop:convergence_sievedMDS_twisted},
its polynomial growth on vertical strips, and the fast decay of the Mellin transform on vertical lines,
we can move the line of integration to $\Re(w)=\dd_\kappa+\eps$, picking up the residue(s) at the pole(s) in between. The integral
on the new line is bounded by the error stated in the theorem, finishing the proof.
\end{proof}

\end{para}

\subsection{Mollifier length}

\begin{para}
Coming back to the mollified second moment in~\eqref{eq:mol2ndmom}, we
use the region of convergence and the estimates in \cref{prop:convergence_sievedMDS_twisted}
to move the line of integration to $\Re (w)=\wdd_\kappa+\eps$,
where we set $\wdd_\kappa=\max(\dd_\kappa, 1/2+1/r)$ to avoid picking up the residue at the second pole.
We obtain
\[\begin{aligned}
\frac{1}{2\pi i}\int_{(2)} \wW(w)Z_\kappa^S(\tfrac 12, \tfrac 12, w;\chi_{\bi\ov{\bi'}})X^w dw=&
\res_{w=1} \wW(w) Z_\kappa^S(\tfrac 12, \tfrac 12, w;\chi_{\bi\ov{\bi'}})X^w +\\
&+O(X^{\wdd_\kappa+\eps}|\bi\bi'|^{(1/2+r-1)(1-\wdd_\kappa)+\eps}).
\end{aligned}\]
 Explicitly
$\wdd_\kappa$ differs from the boundary of the region of convergence $\dd_\kappa$ only for $\kappa=0$ and $r=3,4$:
\[
\wdd_1=\dd_1=\frac{r+A_r}{r+1},\qquad \wdd_0=\begin{cases}\tfrac12+\tfrac1r & \text{$r=3,4$ }\\
                          \tfrac{3r-2}{4r-2} & \text{$r>4$ }.
                         \end{cases}
\]
We require the error in~\eqref{eq:mol2ndmom} to be  $o(X)$, and, using that $\lam_\bi\ll |\bi|^{-1+\eps}$,
this happens if
\[ \sum_{\substack{\bi,\bi'\in I(S)_0 \\ |\bi|,|\bi'|\le M}}|\lam_\bi\lam_{\bi'}|
|\bi\bi'|^{1/2 +(1/2+r-1)(1-\wdd_\kappa)}\ll M^{2(1/2+(1/2+r-1)(1-\wdd_\kappa))}=o(X^{1-\wdd_\kappa}).
\]
It follows that the error in ~\eqref{eq:mol2ndmom} has size $o(X)$ if $M=X^{\theta_\kappa-\eps}$ for
\be\label{eq:thetakappa}
\theta_\kappa=\frac{1-\wdd_\kappa}{1+(2r-1)(1-\wdd_\kappa)}.
\ee

For $r=3$, $\kappa=1$, we have $\wdd_1=5/6$, and $\theta_1=1/11$.
\end{para}
\begin{para}
Putting together the results of this section we obtain:
\begin{prop}\label{prop:mol2ndmom}
For $M=X^{\theta_\kappa-\eps}$ with $\theta_\kappa$ defined in \eqref{eq:thetakappa} we have:
\[\begin{aligned}
\Sc_2^{(\kappa)}(M;W)&=D_\kappa \wW(1) X
\sum_{\substack{\bi,\bi'\in I(S)_0 \\ |\bi|,|\bi'|\le M}}
\lam_\bi\lam_{\bi'}|\ei| g_\kappa(\ei^{(\mi\mi')}) h_\kappa(\mi\mi')\\
&\cdot \bigg(\log X-\log\prod_{v\in S_{\mi\mi'}} q_v^{\min (u_v,r-u_v)} + R(\chi_{\bi\ov{\bi'}}) \bigg)
+o(X)
\end{aligned}
\]
where we write $\bi=\ei\mi$, $\bi'=\ei\mi'$ with $\mi$, $\mi'$ coprime, $u_v=\ord_v(\mi\mi')$,
$D_\kappa$, $g_\kappa$, $h_\kappa$ are defined in \cref{prop:Zbres}, and
$$
R(\chi_{\bi\ov{\bi'}} )=\tfrac{\wW'}{\wW}(1)+K+\sum_{v\in S_{\ei^{(\mi\mi')}}} K_v \frac{\log q_v}{q_v} +
\sum_{v\in S_{\mi\mi'}} K_v' \frac{\log q_v}{q_v^{1/2}}
$$
for $K,K_v,K_v'\ll 1$.
\end{prop}

\end{para}

\subsection{First mollified moment}

\begin{para}After applying Mellin inversion, the first mollified moment becomes:
\be\label{eq:mol1stmom}
\Sc_1^{(\kappa)}(M;W)=\sum_{\substack{\bi\in I(S)_0 \\ |\bi|\le M}}
\lam_\bi|\bi|^{1/2}
\frac{1}{2\pi i}\int_{(2)}\wW(w)Z_\kappa^S(\tfrac 12, w;\chi_{\bi})X^w dw
\ee
where for $\kappa=0,1$ we define for $\bi\in I(S)_0$:
\[
Z_\kappa^S(s,w;\chi_\bi):=
	\sum_{\ai\in I(S)_\kappa}
	\frac{L^{S}(s,  \chi_{\mathfrak{a}})}{|\mathfrak{a}|^{w}}
	\chi_\ai(\bi).
\]
We are led to consider the two-variable MDS obtained by letting $s_2\rightarrow \infty$ in the second moment MDS:
\[
\wZ^S(s,w; \psi,\rho)= \sum_{\ai\in I(S)}
\frac{L^{S}(s,  \psi \chi_{\ai_0})}{|\ai|^{w}}
\rho(\ai)Q_\ai(s,\infty;\psi,1),
\]
using that $\lim_{s_2\rightarrow\infty} L(s_2,\chi_{\ai})=1$.
Setting the second variable $\infty$ in $Q_\ai$ is equivalent to setting $x_2=0$ in $Q_k(x_1,x_2;q)$
and in the $v$-part $f(x_1,x_2,y;q)$ from \S\ref{sec:v-part}. In particular,
note that $Q_\ai(s,\infty;\psi,1)=1$ if $\ai$ is $r$-th power-free.
\end{para}
\begin{para}
Define $Z_\kappa(s,w;\chi_\bi;\hi)$ as in \S\ref{par:Zbsieving} after letting $s_2\rightarrow \infty$,
so that $Z_\kappa^S(s,w;\chi_\bi)$ can be expressed as in \S\ref{par:Zbsieving} in terms of $Z_\kappa(s,w;\chi_\bi;\hi)$.
The analogue of \cref{prop: sieving2} becomes
\be\label{eq:Zsieved_1stmom}
\begin{aligned}
Z_\kappa(s,w;\chi_\bi;\hi)&=
\prod_{v\in S_\bi}\frac{1}{1-q_v^{-rs}}
	 \sum_{\di| \bi_1^{r-1}} \frac{1}{|\di|^{s}}
	 \sum_{\substack{\ff,\gi\mid \hi^{r-1} \\ (\ff,\gi)=1}}
	 f_\ff^{(\kappa)} (s,\infty, w)g_{\gi,\hi^{(\ff)}} (s,\infty, w)\cdot \\
	 &\cdot\frac{1}{|R_\ci|}\sum_{\substack{G\in \EE\\\eta\in \widehat{R}_\ci }}
	 \chi_\ff(\bi\gi\di)\psi_G\big( m_{[\bi\gi\di]}\big)\ov{\eta}(G)
	 \wZ^{S\cup S_{\bi\hi}} (s,w;\chi_\ff, \eta\chi_{\bi\gi\di}).
\end{aligned}
\ee

\end{para}

\subsection{Convexity bound}\label{sec:convexity_1stmom}
\begin{para}
The perfect MDS for the first moment was first studied by
Fisher and Friedberg in the quadratic case over function fields~\cite{FF04}, and by Friedberg, Hoffstein
and Lieman over number fields \cite{FHL03}. They proved that
\[
(s-1) (s+w-1-1 \slash r) (w-1)
 \wZ^{S\cup S_\hi}(s,w; \psi\chi_\ff, \rho\chi_\gi )
\]
has analytic continuation to $\C^2$, and it is entire of order 1 as a function of $w$ for fixed $s$.
\end{para}
\begin{para}\label{par:convexity1stmom}
Let $\bi\in I(S)$ be an $(r+1)$-th power-free ideal, which will later be taken to be the
$(r+1)$-th power-free ideal $(\bi\di)^*$ in~\eqref{eq:Zsieved_1stmom}.
As in Section \ref{sec:convexity}, we can apply the functional equation
$\sigma_1\sigma_3\sigma_1$ to obtain the following convexity bound:
\[
\begin{aligned}
 (w-1)&\big(w-\tfrac12-\tfrac1r\big) \,
 \wZ^{S\cup S_{\bi\hi}} \big(1/2,w;\chi_\ff, \chi_{\bi\gi}\big)  \\
 &\ll_\eps
 |\rad\bi|^{(\frac{1}{4}+\frac{r}{2})(1-\sigma)+\eps}
 |\hi|^{\frac{1}{4} + \frac{r-1}{2}(1-\sigma)+\eps}\prod_{\pp_v\| \gi} q_v^{\frac{1-\sigma}{2}}
 (1+|w|)^{2+\tfrac{rd}{2}(1-\sigma)+\eps} .
\end{aligned}
\]
The exponent $1/4$ above  comes from an analogue of \cref{prop:convexity},
where we use
\cref{prop:estimate at1}, and Cauchy-Schwarz
\[
\bigg(\sum_{\ai\in I(S\cup S_\ff)_0}\frac{|L(1/2, \psi\chi_{\ai\ff})|}{|\ai|^{1+\eps}}\bigg)^2\ll
\sum_{\ai\in I(S\cup S_\ff)_0}\frac{|L(1/2, \psi\chi_{\ai\ff})|^2}{|\ai|^{1+\eps}}\ll
|\ff_1|^{1/2+\eps}.
\]
\end{para}
\begin{para}\label{par:refined_convexity1stmom}
 We also have a refined estimate, using the recursive method in \S\ref{sec:refined_convexity}:
 \[
 \begin{aligned}
 (w-1)&\big(w-\tfrac12-\tfrac1r\big) \,
 \wZ^{S\cup S_{\bi\hi}} \big(1/2,w;\chi_\ff, \chi_{\gi\bi}\big)  \\
 &\ll_\eps
 |\rad\bi|^{(\frac{1}{4}+\frac{r}{2})(1-\sigma)+\eps}
 |\hi|^{\frac{r-1}{2}(1-\sigma)+\eps}|\rad\ff\gi|^{\frac{1}{4}}
 \prod_{\pp_v\| \gi} q_v^{\frac{1-\sigma}{2}}
 (1+|w|)^{2+\frac{rd}{2}(1-\sigma)+\eps} .
\end{aligned}
 \]

\end{para}

\subsection{Analytic continuation}\label{sec:analytic_1stmom}
\begin{para}
Using the decomposition~\eqref{eq:Zsieved_1stmom} and the convexity bounds in the previous subsection
we obtain the following region of meromorphic continuation for the twisted sieved MDS. Recall
 that $A_r=1/2$ for $r\ge 4$ while  $A_3=1/3$.
 \begin{prop}\label{prop:convergence_sievedMDS_twisted1st}
  The MDS $Z_\kappa^S (\tfrac 12, w;\chi_\bi)$ has meromorphic continuation
to the region
$$\Re (w)>\dd_\kappa'+\eps,\qquad  \dd_\kappa'=\begin{cases} \frac{r+1+A_r}{r+3}  & \text{if } \kappa=1 \\
		  \frac{2r-1}{3r-1}  & \text{if } \kappa=0,
         \end{cases}
$$
with simple poles at $w=1$ and possibly at $w=\tfrac12+\tfrac1r$.

Moreover, we have the uniform estimate for  $\dd_\kappa'<\Re (w)=\sigma<1$:
 \be\label{eq:convexZb1stmom}
(w-1)(w-\tfrac12-\tfrac1r) Z_\kappa^S (\tfrac 12, w; \chi_\bi)
\ll_\eps |\rad{\bi}|^{\big(\tfrac{1}{4}+\tfrac{r}{2}\big)(1-\sigma)+\eps}(1+|w|)^{2+ \frac{rd}2(1-\sigma)+\eps}.
 \ee
 \end{prop}
 Note that the region of meromorphic continuation contains the pole at $w=\tfrac12+\tfrac1r$ for $r=3$ if $\kappa=1$, and for $3\le r\le 6$ if $\kappa=0$.
\begin{proof}
We use the first-moment sieving identity
\[
        Z_\kappa^S(\tfrac12,w;\chi_\bi)
        =
        \sum_{\hi\in I(S\cup S_\bi)}\mu(\hi)
        Z_\kappa(\tfrac12,w;\chi_\bi;\hi),
\]
multiplied by the factor $(w-1)(w-\tfrac12-\tfrac1r)$ to kill possible poles.
We bound each term $Z_\kappa(\tfrac12,w;\chi_\bi;\hi)$ using the decomposition~\eqref{eq:Zsieved_1stmom} in terms of perfect MDS.
A typical term in the decomposition is
\[f_\ff^{(\kappa)}(\tfrac12,\infty,w)\,
 g_{\gi,\hi^{(\ff)}}(\tfrac12,\infty,w)
 \wZ^{S\cup S_{\bi\hi}}(\tfrac12,w;\chi_\ff,\eta\chi_{\bi\gi\di})
\]
with $\ff,\gi,\di$ as in~\eqref{eq:Zsieved_1stmom}.

For a prime $\pp| \hi$ with $q=|\pp|$, the local factors in the coefficients  $f_\ff^{(\kappa)}$, $ g_{\gi,\hi^{(\ff)}}$, namely
$$f_\kappa^{(\eps)}(q^{-1/2}, 0 ,q^{-w};q),\qquad g^{(\eps)}(q^{-1/2}, 0 ,q^{-w};q),$$
satisfy the same bounds as their evaluation at $(q^{-1/2}, q^{-1/2} ,q^{-w})$ in \S\ref{sec:fg_estimates}.
In addition, the refined convexity bound for $\wZ^{S\cup S_{\bi\hi}}$
in \S\ref{par:refined_convexity1stmom} contributes the common factor
\[
        q^{\frac{r-1}{2}(1-\sigma)}
\]
at this prime.  If $\pp\mid \ff \gi$, it also contributes
the localized large-sieve factor $q^{1/4}$; if $\pp\|\gi$,
it contributes the additional factor $q^{(1-\sigma)/2}$.
We next determine the largest such exponent.

\noindent\textbf{The square-free case $\kappa=1$.}
The maximal local exponent occurs if $\pp\mid \ff$ and it equals
\[
E(\sigma)=\frac{r-1}{2}(1-\sigma)+\frac{1}{4}-2\sigma;
\]
note that in this case we do not need the refined convexity bound,
as the largest contribution is the same in both convexity bounds for $\pp\mid \ff$.
Therefore $Z_\kappa(\tfrac12,w;\chi_\bi;\hi)\ll_\eps |\hi|^{E(\sigma)+\eps}$, and
the condition $E(\sigma)<-1$ gives
$$\sigma >   \frac{r+1+ 1/2}{r+3}.$$
When $r=3$, we can improve the fraction $1/2$ to $1/3$ in the bound for $\sigma$,
as in the second-moment argument.
Using the first moment analogue of \cref{prop:convexity}
we obtain that the sum over $\hi$ in the sieving identity converges if the following series converges
\[
T_\psi(\sigma)
:=
\sum_{\hi=\hi_1}
|\hi|^{1-3\sigma+\eps}
\sum_{\ff\mid\hi^2}
\sum_{\ai\in I(S\cup S_{\bi\hi})_0}
\frac{|L(1/2,\psi\chi_{\ai\ff})|}{|\ai|^{1+\eps}}.
\]
Applying Cauchy-Schwarz to the outer sum gives
\[T_\psi(\sigma)^2
\ll_\eps  \bigg(\sum_{\hi=\hi_1}\sum_{\ff\mid\hi^2}|\hi|^{-1-\eps}\bigg)
\sum_{\hi=\hi_1}
|\hi|^{3-6\sigma+\eps}
\sum_{\ff\mid\hi^2}
\sum_{\ai\in I(S\cup S_{\bi\hi})_0}
\frac{|L(1/2,\psi\chi_{\ai\ff})|^2}{|\ai|^{1+\eps}}.
\]
Comparing to the first equation in the  proof of \cref{prop:conv_req3},
the last series converges if $3-6\sigma<-4/3$, namely for $\sigma>13/18$.
That is the same as replacing $1/2$ by $1/3$ in the earlier bound.

\noindent\textbf{The $r$-th power-free case $\kappa=0$.}
The dominant local term is when $\pp\nmid \ff\gi$ (the $g^{(0)}$ term),
for which the local exponent is
\[
        E_0(\sigma)
        =
        \frac{r-1}{2}(1-\sigma)+\frac r2-1-r\sigma .
\]
It is here that the refined convexity bound saves a factor $1/4$.
The terms with $\pp\mid \ff$ are smaller by the factor
$q^{1/4-\sigma}$, and the terms with $\pp\|\gi$ are smaller by the factor $q^{1/4-\sigma/2}$;
this is $<1$ because $\sigma>1/2$.  The remaining $g^{(\epsilon)}$, $\eps>1$, terms are smaller
still.  Hence absolute convergence in the sieving identity happens for
\[
        E_0(\sigma)<-1, \text{ that is \ }   \sigma>\frac{2r-1}{3r-1}.
\]
Including the uniform bounds in $|\bi|$ and $|w|$ finishes the proof.
\end{proof}
 \end{para}

\subsection{Residues}\label{sec:residues_1stmom}
\begin{para}
As in \cref{prop:Zresidue}, the function
$\wZ^{S\cup S_{\bi\hi}} (s,w;\chi_\ff, \chi_{\bi\gi\di};G)$ has a pole
at $w=1$ only if $\chi_{\bi\gi\di}=1$, and for $T= S\cup S_{\bi\hi}$ we have
\[
\res_{w=1} \wZ^{T}(s,w; \chi_\ff, 1; G)= \frac{1}{|R_\ci|} \zeta_F^T(rs)\res_{w=1}\zeta_F^T(w).
\]
We conclude from~\eqref{eq:Zsieved_1stmom} that
\[\res_{w=1} Z_\kappa(s,w;\chi_\bi;\hi) =\zeta_F^T(rs)\res_{w=1}\zeta_F^T(w)
\prod_{v\in S_\bi}\frac{q_v^{-s(r-\ord_v\bi )}}{1-q_v^{-rs}}\sum_{\ff|\hi^{r-1}}
f_\ff^{(\kappa)} (s,\infty, 1)g_{1,\hi^{(\ff)}} (s,\infty, 1),
\]
and the MDS $Z_\kappa^S(s,w;\chi_\bi)$ has a simple pole at $w=1$,
with:
\[
\res_{w=1}Z_\kappa^S(s,w;\chi_\bi)=\zeta_F^S(rs) \res_{w=1}\zeta_F^S(w)
\prod_{v\notin S}(1-q_v^{-1}) Q_r^{(\kappa)}(q_v^{-s};q_v)
\prod_{v\in S_\bi} \frac{q_v^{-(r-\ord_v(\bi))s}}{Q_r^{(\kappa)}(q_v^{-s};q_v)}.
\]
The polynomial $Q_r^{(\kappa)}(x;q)$ is determined from the $v$-part $f(x,0, y;q)$ as follows:
\[
Q_r^{(\kappa)}(q_v^{-s};q_v)=\frac{1-(1-q_v^{-1})(1-q_v^{-rs})\bigg(g^{(0)}(q_v^{-s},0,q_v^{-1};q_v)+
\sum_{\eps=1}^{r-1}f_\kappa^{(\eps)}(q_v^{-s},0,q_v^{-1};q_v)
   \bigg)}{1-q_v^{-1}}.
\]
The explicit formula for the $v$-part $f$ in \S\ref{sec:v-part} gives:
\be
\label{eq:euler_factors_1stmom}
Q_r^{(\kappa)}(x;q)=\begin{cases}
\displaystyle \frac{1-q^{-r}}{1-q^{-1}}-x^r\frac{q^{-1}-q^{-r}}{1-q^{-1}} & \text{ if $\kappa=0$}\\
               1+q^{-1}-x^rq^{-1}      & \text{ if $\kappa=1$}.
                    \end{cases}
\ee
Let $Q_r^{(\kappa)}(x)=Q_r^{(\kappa)}(x;1/x^2)$. Note that $Q_r^{(\kappa)}(x;q)$ and
$Q_r^{(\kappa)}(x)$ are positive for $0<x<1$, so the Euler product above is positive.
\end{para}

\subsection{Twisted first moment asymptotics}

\begin{para}
Putting together the results of Subsections \ref{sec:convexity_1stmom}-\ref{sec:residues_1stmom} gives, as in the second-moment case, the following theorem.
\begin{thm} \label{thm:twisted_first_moment}
Fix $\kappa\in\{0,1 \}$.
Let $W(u)$ be a Schwartz function with compact support on $(1,2)$ satisfying $0\le W(u) \le 1$. Let $\bi\in I(S)$ be $r$-th power-free.
Then, for every $\varepsilon > 0$, we have
\[\begin{aligned}
\sum_{\ai \in I(S)_\kappa}  L^S(\tfrac{1}{2}, \chi_{\ai}) \chi_\ai(\bi) W\big(\tfrac{|\ai|}{X}\big)
&=X \wW(1) C_{\kappa,\bi} +X^{\frac{1}{2}+\frac{1}{r}}\wW(\tfrac 12+\tfrac 1r) E_{\kappa,\bi}\\
&+O\Big(X^{\dd_\kappa'+\varepsilon} |\rad{\bi}|^{\big(\tfrac{1}{4}+\tfrac{r}{2}\big)(1-\dd_\kappa')+\eps} \Big),
\end{aligned}
\]
where $E_{\kappa,\bi}$ can be computed from the residue at $\tfrac12+\tfrac1r$ of the Kubota Dirichlet series, and
\[
C_{\kappa,\bi}=\zeta_F^S(r/2) \res_{w=1} \zeta_F^S(w)
\prod_{v\notin S} (1-q_v^{-1})Q_r^{(\kappa)}\big(q_v^{-1/2}\big)
\prod_{v\in S_\bi} \tfrac{q_v^{-(r-\ord_v(\bi))/2}}{Q_r^{(\kappa)}(q_v^{-1/2})}.
\]
Here $\dd_\kappa'$ is given in \cref{prop:convergence_sievedMDS_twisted1st}, and
$Q_r^{(\kappa)}(x)=Q_r^{(\kappa)}(x;1/x^2)$ with $Q_r^{(\kappa)}(x;q)$
given in~\eqref{eq:euler_factors_1stmom}.
\end{thm}
This asymptotic formula captures the second-order term for  $r=3$ if $\kappa=1$,
and $r\le 6$ if $\kappa=0$.
\begin{proof}
By Mellin inversion, the moment sum equals
\[
\frac{1}{2\pi i}\int_{(2)}\wW(w)Z_\kappa^S(\tfrac 12, w;\chi_{\bi})X^w dw.
\]
By the meromorphic continuation of $Z_\kappa^S(\tfrac 12, w;\chi_{\bi})$ in \cref{prop:convergence_sievedMDS_twisted1st},
its polynomial growth on vertical strips, and the fast decay of the Mellin transform on vertical lines,
we can move the line of integration to $\Re(w)=\dd_\kappa'+\eps$, picking up the residue(s) at the pole(s) in between. The integral
on the new line is bounded by the error stated in the theorem, finishing the proof.
\end{proof}
\end{para}

\begin{para}
Using the previous theorem, we can bound the error in the first mollified moment,
with the same mollifier length $M$ as for the second moment.
\begin{prop}\label{prop:mol1stmom}
 For $M=X^{\theta_\kappa-\eps}$ with $\theta_\kappa$ defined in \eqref{eq:thetakappa} we have:
\[
\Sc_1^{(\kappa)}(M;W)=C_\kappa\wW(1)  X
\sum_{\substack{\bi\in I(S)_0 \\ |\bi|\le M}}\lam_\bi r_\kappa(\bi) +o(X)
\]
where
\[C_\kappa:=\zeta_F^S(r/2) \res_{w=1}\zeta_F^S(w)\prod_{v\notin S} (1-q_v^{-1})Q_r^{(\kappa)}\big(q_v^{-1/2}\big),\quad
r_\kappa(\bi)=\prod_{v\in S_\bi} \frac{q_v^{u_v-r/2}}{Q_r^{(\kappa)}\big(q_v^{-1/2}\big)}
\]
with $u_v=\ord_v\bi$.
\end{prop}
\begin{proof}
Define $\wdd_\kappa'=\max \big(\dd_\kappa', \tfrac12+\tfrac1r\big) $. Moving the line of integration in the integral in
\eqref{eq:mol1stmom} to $\Re(w)=\wdd_\kappa'+\eps$ picks up only the pole at $w=1$, and by the same argument in
 \cref{thm:twisted_first_moment} and \eqref{eq:mol1stmom} we obtain
 \[
 \begin{aligned}
    \Sc_1^{(\kappa)}(M;W)-C_\kappa\wW(1)  X \sum_{\substack{\bi\in I(S)_0 \\ |\bi|\le M}}\lam_\bi r_\kappa(\bi)
&\ll_\eps X^{\wdd_\kappa'+\eps} \sum_{\substack{\bi\in I(S)_0 \\ |\bi|\le M}}
|\lam_\bi| |\bi|^{1/2} |\bi|^{\big(\tfrac{1}{4}+\tfrac{r}{2}\big)(1-\wdd_\kappa')+\eps}\\
&\ll_\eps X^{\wdd_\kappa'+\eps} M^{\tfrac12+\big(\tfrac{1}{4}+\tfrac{r}{2}\big)(1-\wdd_\kappa')+\eps},
   \end{aligned}
\]
where we used that $\lam_\bi\ll |\bi|^{-1+\eps}$.
Given that $M=X^{\theta_\kappa-\eps}$, the last expression is $o(X)$ provided
\[
\theta_\kappa<\theta_\kappa':= \frac{2(1-\wdd_\kappa') }{1+(r+1/2)(1-\wdd_\kappa')}.
\]
This is easily verified in all cases, finishing the proof.
\end{proof}
\end{para}

\subsection{Choosing the mollifier}\label{sec:sqfreemol}
\begin{para}
Comparing propositions \ref{prop:mol1stmom} and \ref{prop:mol2ndmom} to the similar estimate
for the case of cubic characters in \cite[Sec. 9]{DFDS24}, we note that we can choose
a similar mollifier, supported on square-free ideals, to obtain a positive proportion of non-vanishing for $L(1/2,\chi_\ai)$, both in the family of square-free $\ai$ and in the family of $r$-th power-free $\ai$. Thus we now set
\[
M_\ai=\sum_{\substack{\bi\in I(S)_1 \\ |\bi|\le M_\kappa}}
\lam_\kappa(\bi)|\bi|^{1/2}\chi_\ai(\bi),
\]
where we write $\lam_\bi=\lam_\kappa(\bi)$, and $M_\kappa=X^{\theta_\kappa-\eta}$ with $\theta_\kappa$ defined in~\eqref{eq:thetakappa} and $\eta>0$ small.
\end{para}
\begin{para}
Let $C_\kappa$, $D_\kappa$ be the Euler products, and $r_\kappa$, $g_\kappa$, $h_\kappa$ be the multiplicative
functions defined in \cref{prop:mol1stmom} and in \cref{prop:Zbres}, restricted to square-free ideals.
Define two multiplicative functions on square-free ideals by their values on primes:
\be\label{eq:GH}
G_\kappa(\pp_v)=r_\kappa(\pp_v)-h_\kappa(\pp_v), \quad
H_\kappa(\pp_v)=g_\kappa(\pp_v)-q_v^{-1}h_\kappa(\pp_v)^2,
\ee
and let
\[
E_\kappa= \prod_{v\not\in S} (1-q_v^{-1})\big(1+q_v^{-1}\tfrac{G_\kappa(\pp_v)^2}{H_\kappa(\pp_v)}\big).
\]
As in \cite[Sec. 9]{DFDS24}, one checks that
\[G_\kappa(\pp_v)=-1+O(q_v^{-1}), \qquad H_\kappa(\pp_v)=1+O(q_v^{-1}),
\]
and a short computation using the formulas for $P_r^{(\kappa)}$
shows that $H_\kappa(\pp_v)>0$ for all $v$. Therefore the Euler product $E_\kappa$ is convergent
and positive.

As in \cite[Sec. 9]{DFDS24}, for $\bi$ square-free with $|\bi|<M_\kappa$ we define
\[
\xi_\kappa(\bi)=\frac{C_\kappa}{D_\kappa \log M_\kappa} \cdot \frac{G_\kappa(\bi)}{|\bi|H_\kappa(\bi)},
\]
and extend $\xi_\kappa(\bi)$ by 0 for other $\bi$.
We choose the mollifier supported on $\bi \in I(S)_1$ with $|\bi|<M_\kappa$
defined by
\[
\lam_\kappa(\bi)=\sum_{\substack{\ai\in I(S)_1\\|\ai\bi|<M_\kappa }} \mu(\ai)h_\kappa(\ai)\xi_\kappa(\ai\bi).
\]

When $r=3$ and $\kappa=1$, the arithmetic functions defined here agree with the ones defined in \cite[Sec. 9]{DFDS24}, and in general one easily checks that they satisfy the same growth estimates.
\end{para}
\subsection{Proof of non-vanishing theorems}\label{sec:non-vanishing}
\begin{para}
We now prove Theorems~\ref{Nonvanishing-sq-free-alt} and~\ref{Nonvanishing-r-th-power-free-alt}.
Let $\rho_F^S:=\res_{s=1} \zeta_F^S(s)$. With the notation in the previous subsection,
the same argument as in \cite[Sec. 9]{DFDS24}  gives
\[
\Sc_1^{(\kappa)}(M;W)\sim X \wW(1) \frac{C_\kappa^2E_\kappa}{D_\kappa}\rho_F^S , \ \  \quad
\Sc_2^{(\kappa)}(M;W)\sim X \wW(1) \frac{C_\kappa^2E_\kappa}{D_\kappa}\rho_F^S \bigg(1+\frac{1}{\theta_\kappa-\eta}\bigg).
\]

Assume $W\ne 0$. Since $0\le W\le 1$ and $Q_\ai\ge 1$ we have
\[\begin{aligned}
\sum_{\substack{\ai\in I(S)_\kappa \\ L(1/2,\chi_\ai)\ne 0\\X<|\ai|<2X}} 1&\ge
\sum_{\substack{\ai\in I(S)_\kappa \\ L(1/2,\chi_\ai)\ne 0\\X<|\ai|<2X}} W\big(\tfrac{|\ai|}{X}\big) Q_\ai^{-1} \\
&\ge
\frac{ | \Sc_1^{(\kappa)}(M;W)|^2 }{\Sc_2^{(\kappa)}(M;W)}=
X \left(\wW(1) \frac{C^2_\kappa E_\kappa}{D_\kappa} \rho_F^S \frac{\theta_\kappa}{\theta_\kappa+1}-\eps'\right) +o(X),
\end{aligned}
\]
for some $\eps'>0$ depending on $\eta$ linearly. Note that
\[\sum_{\substack{\ai\in I(S)_\kappa \\|\ai|<X}} 1 \sim X \frac{\rho_F^S}{\zeta_F^S(n_\kappa)},
\]
where $n_1=2$, $n_0=r$. Letting $W$ approach the indicator function of the interval $(1,2)$ and using dyadic summation,
we conclude that the proportion of non-vanishing is at least
\[
\frac{C_\kappa^2E_\kappa}{D_\kappa}\zeta_F^S(n_\kappa) \frac{\theta_\kappa}{\theta_\kappa+1}-\eps'.
\]
\end{para}
\begin{para}
For $\kappa=1$, direct computation shows that we have as in  \cite[Sec. 9.3.1]{DFDS24}
\[
\frac{C_1^2 E_1}{D_1} =1/\zeta_F^S(2),
\]
so the proportion of non-vanishing is $\ge \frac{\theta_1}{\theta_1+1}-\eps$. This proves \cref{Nonvanishing-sq-free-alt}.
\end{para}
\begin{para} \label{par:non-vanish0}
For $\kappa=0$, an explicit computation shows that
\[
\frac{C_0^2 E_0}{D_0}=\prod_{v\notin S} \bigg(1-q_v^{-2}-q_v^{-3} \frac{a_r(q_v^{-1})}{b_r(q_v^{-1})}\bigg),
\]
where
\[
b_r(y)=\sum_{j=0}^{r-2} \big(1+\tfrac{j(j+1)}{2} \big) y^j +\tfrac{r(r-1)}{2} y^{r-1}+ \sum_{j=r}^{2r-2} \tfrac{j(2r-1-j)}{2} y^j,
\]
is a polynomial with positive coefficients,
$a_3(y)=y^4-y$, and for $r\ge 4$:
\[
a_r(y)=\sum_{j=0}^{r-4} (j+1)y^j-(r-2)\sum_{j=r-2}^{2r-5} y^j + \sum_{j=2r-2}^{3r-5} (3r-4-j) y^j.
\]
Denoting by
$$
\alpha_r=\frac{C_0^2 E_0}{D_0}\zeta_F^S(r), \quad u_v(r)=  \frac{a_r(q_v^{-1})}{b_r(q_v^{-1})(1-q_v^{-2})},
$$
we obtain the formula for $\alpha_r$ in the statement in \cref{Nonvanishing-r-th-power-free-alt}.
It remains to show the claimed bounds on $u_v(r)$, namely $u_v(3)<0$ and $0<u_v(r)<1$ for $r\ge 4$

\end{para}

\begin{para}For $r=3$ it is clear that $a_3(y)<0$ for $0<y<1$, and since $b_3$ has positive coefficients we have $u_v(3)<0$.

For $r\ge 4$, and $0<y\le 1/2$ we have:
\[a_r(y)\geq 1-(r-2)\sum_{j=r-2}^{2r-5}y^j
\geq 1-\frac{(r-2)y^{r-2}}{1-y}\geq 0,
\]
and the inequality is strict in the range under consideration. To prove the upper bound, one checks that
$$\frac{b_r(y)(1-y^2)-a_r(y)}{1-y}$$
is a nonzero polynomial with nonnegative coefficients, so $u_v(r)<1$.
\end{para}

\begin{para} Finally, let $F=\Q(\mu_{2r})$ (which is the same as $\Q(\mu_{r})$ if $r$ odd). We use only the bounds on $u_v$ in the previous paragraph
to prove the statement in \cref{rem:non-vanishing-alt}, namely that:
\[
\prod_{v\notin S} (1-q_v^{-3} u_v(r) )> \begin{cases}
                                          1 & \text{ for $r=3$}\\
                                          2/3 & \text{ for $r=4$}\\
                                          \dfrac{r^2+2r-1}{r(2r+1)} & \text{ for $r>4$}.
                                         \end{cases}
\]
For $r=3$ we have $u_v(3)<0$ and the statement is trivial.

For $r\ge 4$, all the terms in the product are between 0 and 1, so it is enough to prove the stronger inequality
with the product over all finite primes of $F$ unramified
over $\Q$. Let $S_r$ be the set of ramified places, namely those above the places in $r$.
Since $u_v<1$ we have
\[\begin{aligned}
\prod_{v\notin S_r} (1-q_v^{-3} u_v(r) ) &\ge 1-\sum_{v\notin S_r } q_v^{-3}\\
&\ge 1-  \varphi(2r)\sum_{m\geq1}(1+mr)^{-3}
> 1- \frac{\varphi(2r)\zeta(3)}{r^3}> 1-\frac{\zeta(3)}{r^2},
  \end{aligned}
\]
where in the second inequality we used that for $v\notin S_r$ we have $q_v\equiv 1 \pmod {r}$, and the
number of places of any fixed norm is at most $\varphi(2r)$.
Since $\zeta(3)=1.202...<5/4$, for $r\ge 4$
the last number is greater than $11/12$,
which is greater than the lower bound claimed above.
\end{para}

\appendix

\section{Some consequences of the large sieve inequality for \texorpdfstring{$r$}~-th order characters}
In this appendix we give a proof of \cref{prop:estimate at1}.
We need straightforward generalizations of Propositions 4.1 and 4.2 from \cite{DFDS24}
to the case of $r$-th order characters, which are based on the large sieve inequality
of Blomer, Goldmakher and Louvel.

\begin{thm}\emph{\cite[Thm.~1.3]{BGL14}}
    \label{thm:thm3}
    Let $M, N \geq 1/2$, and let $\lambda_{\ai} \in \C$ be a sequence of complex numbers indexed by $\ai \in I(S)$. Then,  for any $\varepsilon > 0$,
    \[
    \sum_{ \substack{\ai \in I(S)\\ |\ai|\leq M}} \mu(\ai)^2 \,
    \bigl| \sum_{ \substack{\b \in I(S)\\ |\b|\leq N}} \mu(\b)^2 \lambda_{\b} \chi_{\ai}(\b) \bigr|^2
    \ll
    (MN)^\varepsilon (M+N +(MN)^{2/3})
    \sum_{  \substack{\b \in I(S)\\ |\b|\leq N}} \mu(\b)^2|\lambda_{\b} |^2.
    \]

\end{thm}

We first collect some standard facts. Let $F$ be a totally imaginary number field of degree $d$ with ring of integers $\OO$ containing the $r$-th roots of unity.
Let $S$ be the finite set of places fixed in Section~\ref{sec:notation}, and
denote by $I$ the set of all ideals in $\OO$.

\subsection{Hecke \texorpdfstring{$L$}~-functions}

Let $\mathfrak{m} \subset \OO$ be an ideal and let $\psi \pmod{\mathfrak{m}}$ be a Hecke character of conductor
$\mathfrak{c}_{\psi} \in I$ and trivial infinity type. The Hecke $L$-function attached to $\psi$ is given by
\begin{equation} \label{Lspsi}
L(s,\psi):=\sum_{0 \neq \mathfrak{n} \in I} \frac{\psi(\mathfrak{n})}{|\mathfrak{n}|^{s}}, \qquad \Re(s)>1.
\end{equation}
We set $\psi(\mathfrak{n})=0$ whenever $\mathfrak{n}$ and $\mathfrak{m}$ are not coprime.  The gamma factor from \cite[(5.3)]{IK} is
$$
\gamma(s)=\left(\pi^{-s} \Gamma\left(\frac{s}{2}\right)\Gamma\left(\frac{s+1}{2}\right)\right)^{d/2}=\left(2\pi^{1/2}(2\pi)^{-s}\Gamma(s)\right)^{d/2}.
$$

The completed Hecke $L$-function of $\psi$ is
\begin{equation} \label{completed}
\Lambda(s,\psi):= \left(\frac{\Gamma(s)}{(2 \pi)^{s}}\right)^{d/2} (|d_{F}| |\mathfrak{c}_\psi|)^{s/2}  L(s,\psi), \qquad s \in \C.
\end{equation}
The completed $L$-function $\Lambda(s,\psi)$ is entire, provided that $\psi$ is primitive and
$\mathfrak{c}_{\psi} \neq \OO$. Furthermore, it satisfies the functional equation
\begin{equation*}
\Lambda(s,\psi)=G(\psi) \Lambda(1-s,\overline{\psi}).
\end{equation*}


\subsection{Approximate functional equations} Let $A\in \Z_{\geq3}$, and  $\displaystyle C(u)=\left(\cos\frac{\pi u}{4A}\right)^{-4dA}$. The function $C(u)$ is even, holomorphic in $|\Re(u)|<4$, and  satisfies
\begin{enumerate}[label=(\alph*)]
\item $C$ is bounded in $|\Re(u)|<4$;
\item $\overline{C(u)}=C(\bar{u})$;
\item $C(0)=1$.
\end{enumerate}
For  $\Re(s)\geq -3$ let $V_s: \R_{>0}\to \C$ be defined by
\begin{equation} \label{Vsdef}
V_s(y):=\frac{1}{2 \pi i} \int_{(3)} (2 \pi)^{-dw/2} y^{-w}  C(w) \left(\frac{\Gamma(s+w)}{\Gamma(s)}\right)^{d/2} \frac{dw}{w}.
\end{equation}
Note that we have $\overline{V_s(y)}=V_{\overline{s}}(y)$.

\begin{prop}\label{decay_lemma} \emph{\cite[Prop.~5.4]{IK}} Let $q_\infty=\left((|s|+3)(|s+1|+3)\right)^{d/2}$ and assume $\Re(s)\geq 3\alpha>0$. Then, for all $y>0$,
\begin{enumerate}[label=(\roman*)]
\item $\displaystyle y^aV_s^{(a)}(y)\ll_{A,\alpha,a}  \left(1+\frac{y}{q_\infty^{1/2}}\right)^{-A} \lesssim \left(1+\frac{y}{(|s|+1)^{d/2}} \right)^{-A}$,
\item $\displaystyle y^aV_s^{(a)}(y)=\delta_a +O_{A,\alpha,a} \left(\frac{y^\alpha}{q_\infty^{\alpha/2}}\right)=\delta_a +O_{A,\alpha,a} \left( \frac{y^\alpha}{(|s|+1)^{d\alpha/2}} \right)$,
\end{enumerate}
where $\delta_0=1$, $\delta_a=0$, if $a>0$.
\end{prop}
\begin{thm} \emph{\cite[Thm.~5.3]{IK}} \label{afe}
Let   $X>0$ and $\OO\neq \q\in I(S)_0$. For $s$ such that $0\leq \Re(s)\leq 1$ we have
\begin{equation*}
L(s,\chi_\q)=\sum_{0 \neq \mathfrak{n} \in I} \frac{\chi_\q(\mathfrak{n})}{|\mathfrak{n}|^s} V_s \left( \frac{|\mathfrak{n}|}{X \sqrt{|d_F| |\ci_{\chi_\q}|}}   \right) +\varepsilon(\q,s) \sum_{0 \neq \mathfrak{n} \in I} \frac{\overline{\chi_\q(\mathfrak{n})}}{|\mathfrak{n}|^{1-s}}
V_{1-s} \left( \frac{X |\mathfrak{n}|}{\sqrt{|d_F| |\ci_{\chi_\q}|}} \right),
\end{equation*}
where
$$\varepsilon(\q,s)= (|d_F| |\ci_{\chi_\q}|)^{1/2-s} (2 \pi)^{d(2s-1)/2} \left(\frac{\Gamma(1-s)}{\Gamma(s)}\right)^{d/2}G(\chi_\q).$$
\end{thm}

Let
$d(\mathfrak{n})$, $0 \neq \mathfrak{n} \in I$,
be the divisor function on ideals.
For  $\varepsilon>0$, we have
$
d(\mathfrak{n}) \ll_{\varepsilon} |\mathfrak{n}|^{\varepsilon} $ for all $\mathfrak{n}$.
\begin{lem} \label{device}  \emph{\cite[Lemma~3.9]{DFDS24}}
Let $0\neq \q \in I(S)_0$,
$U>0$, and $s\in \C$ with $\Re(s) \in (0,1]$. Then, for any $0 \leq \alpha < \Re(s) \leq 1$ we have

\begin{align*}
L(s,\chi_\q)^2
 =\sum_{0 \neq \mathfrak{n} \in I} d(\mathfrak{n}) \chi_\q(\mathfrak{n}) |\mathfrak{n}|^{-s} e^{-|\mathfrak{n}|/U}
-\frac{1}{2 \pi i} \int_{(\alpha)} L(w,\chi_\q)^2 \Gamma(w-s) U^{w-s} dw.
 \end{align*}
 \end{lem}


\subsection{Second moment bounds} \label{second_moment_section}

We next introduce the notation used in this subsection in connection with the large sieve.
\begin{notation}
Let $S$ be the finite set of places fixed in Section~\ref{sec:notation}.
Denote by $I^S\subset I$ the ideals supported on $S$ and by $I(S)$ the ideals coprime
to the finite places in $S$. For $R>0$, let $I_R:= \big \{\b\in I ~\big|~ |\b| \leq R \}$ and recall from Section~\ref{sec:notation} that $I(S)_0=\{\ai\in I(S)~|~\ai=\ai_0\}$. We decompose $\b\in I$ as  $\b=\b_S\bt$, where  $\b_S\in I^S$,  $\bt\in I(S)$. We write $\bt=\bt_{\underline{1}}\bt_{\underline{2}}^2$, i.e. $\bt_{\underline{1}}$ is the square-free part of $\bt$. If $\b\in I(S)$, then $\bt=\b$, and if $\b\in I(S)_0$, then $\bt=\b=\b_0$. Note that the notation  $\bt_{\underline{1}}$ differs from $\bt_1$, which denotes the square-free kernel of $\bt_0$.

We use $\b\sim B$ to  mean $B<|\b|\leq 2B$, and $x \asymp N$ to mean that there exist constants $c_1,c_2>0$ such that
$c_1 N \leq x \leq c_2 N$.

We decompose $\q\in I(S)_0$ as  $\q=\q_\sf\q_\full$, with  $\q_\sf, \q_\full\in I(S)_0$ relatively prime, $\q_\sf$ square-free and $\q_\full$ square-full. 
For $Q_1, Q_2 \geq 1/2$, we denote $Q=Q_1 Q_2$. Let
\begin{equation*}
\Eu{F}(Q_1, Q_2):= \big \{\q=\q_\sf \q_\full  \in I(S)_0~\Big|~ |\q_\sf| \asymp Q_1, ~ |(\q_\full)_1|\asymp Q_2 \big \}.
\end{equation*}
\end{notation}

\begin{prop}\label{cubic_Dirichlet_prop_var} \emph{\cite[Prop.~4.1]{DFDS24}}
     Let $\varepsilon, R>0$, $Q_1, Q_2>1/2$ and  $\varphi: I\to \C$, such that  $|\varphi(\mathfrak{b})| \ll |\mathfrak{b}|^\varepsilon$, and $\supp(\varphi)\subseteq I_R$.
     For $ \q \in I(S)_0$, denote
    \begin{equation*}
       \displaystyle \Eu{R}(\q) := \sum_{0 \neq \b \in I_R } \frac{\varphi(\mathfrak{b}) \chi_\q(\mathfrak{b})}{|\mathfrak{b}|^{1/2}}.
    \end{equation*}
Then,  we have
    \begin{equation} \label{R_cubic_large_sieve}
    \sum_{\q \in \Eu{F}(Q_1, Q_2)} |\Eu{R}(\q)|^2 \ll_{\varepsilon} (QR)^{\varepsilon}Q_2 \left(Q_1 + R + (Q_1R)^{2/3} \right).
    \end{equation}
\end{prop}
\begin{proof} The statement is trivial for $R<1$ and we henceforth assume that $R\geq 1$.
    Write $\b=\b_S\bt_{\underline{1}}\bt_{\underline{2}}^2$. Dyadically partition  $\bt_{\underline{1}} \sim B_1$ and $\bt_{\underline{2}} \sim B_2$, and write
    \begin{align} \label{R_dyadic}
        \Eu{R}(\q) = \sum_{\substack{B_1, B_2 \text{ dyadic} \\ B_1,B_2\gg 1 \\ B_1 B_2^2 \ll R}} \  \sum_{\b_S\in I^S} \frac{1}{|\b_S|^{1/2}}  \mathop{\sum \sum}_{\substack{ \bt_{\underline{1}},\bt_{\underline{2}} \in I(S) \\ |\bt_{\underline{1}}| \sim B_1\\ \ |\bt_{\underline{2}}| \sim B_2}}  \mu^2(\bt_{\underline{1}}) \frac{\varphi(\b_S \bt_{\underline{1}} \bt_{\underline{2}}^2) \chi_\q(\b_S \bt_{\underline{1}} \bt_{\underline{2}}^2)}{\left|  \bt_{\underline{1}} \right|^{1/2} \left| \bt_{\underline{2}}\right| }.
    \end{align}

Then, $\Eu{R}(\q)$ is the following sum over data $\lambda=(\b_S,B_1,B_2, \bt_{\underline{2}})$, where $\b_S\in I^S$, $B_1,B_2$  dyadic,  $B_1,B_2\gg 1$, $B_1 B_2^2 \ll R$, $|\bt_{\underline{2}}| \sim B_2$
\begin{equation*}
 \sum_{\lambda}     \frac{\chi_\q(\b_S\bt_{\underline{2}}^2)}{\big|\b_S\big|^{1/4} \left|\bt_{\underline{2}}\right|^{1/2}}  \cdot   \sum_{ |\bt_{\underline{1}}| \sim B_1}  \frac{\mu^2(\bt_{\underline{1}})\varphi(\b_S \bt_{\underline{1}} \bt_{\underline{2}}^2)\chi_\q(\bt_{\underline{1}})}{\left|\b_S\right|^{1/4} \left|\bt_{\underline{1}}\right|^{1/2}\left| \bt_{\underline{2}}\right|^{1/2}}\  .
\end{equation*}
Applying the (Hermitian version) Cauchy-Schwarz inequality we obtain
$$
|\Eu{R}(\q)|^2\leq \sum_{\lambda}    \frac{1}{\left|\b_S\right|^{1/2} \left|\bt_{\underline{2}}\right|}\cdot
\sum_{\lambda}     \frac{1}{\left|\b_S\right|^{1/2} \left|\bt_{\underline{2}}\right|}  \cdot    \left| \sum_{\substack{ |\bt_{\underline{1}}| \sim B_1 }} \frac{\mu^2(\bt_{\underline{1}}) \varphi(\b_S \bt_{\underline{1}}\bt_{\underline{2}}^2) \chi_\q(\bt_{\underline{1}})}{|\bt_{\underline{1}}|^{1/2}} \right|^2 .
$$
The first sum is $O_\varepsilon(R^\varepsilon)$, and we conclude
\begin{equation*}
\resizebox{\textwidth}{!}{
$\displaystyle
        \sum_{\q \in \Eu{F}(Q_1, Q_2)} |\Eu{R}(\q)|^2 \ll_\varepsilon R^\varepsilon \  \sum_{\substack{B_1,B_2 \text{ dyadic} \\ B_1,B_2\gg 1 \\ B_1 B_2^2 \ll R}} \ \sum_{\b_S\in I^S} \frac{1}{|\b_S|^{1/2}B_2} \sum_{|\bt_{\underline{2}}| \sim B_2} \sum_{\q \in \Eu{F}(Q_1, Q_2)} \left| \sum_{|\bt_{\underline{1}}| \sim B_1 } \frac{\mu^2(\bt_{\underline{1}}) \varphi(\b_S\bt_{\underline{1}}\bt_{\underline{2}}^2) \chi_\q(\bt_{\underline{1}})}{|\bt_{\underline{1}}|^{1/2}} \right|^2 .
$
}
\end{equation*}
Decompose $\q=\q_\sf\q_\full$. For fixed $\q_\full=\mathfrak{v}$, $\b_S$, and $\bt_{\underline{2}}$, we use  the hypothesis on $\varphi$ and Theorem \ref{thm:thm3} to  obtain that for any $\varepsilon> 0$ (after choosing $\varepsilon$ sufficiently small and renaming the final loss)
    \begin{align*}
\displaystyle
\sum_{|\mathfrak{u}| \sim Q_1 } \mu^2(\mathfrak{u}) &
\left| \sum_{|\bt_{\underline{1}}| \sim B_1 } \frac{\mu^2(\bt_{\underline{1}}) \varphi(\b_S\bt_{\underline{1}}\bt_{\underline{2}}^2) \chi_{\mathfrak{v}}(\bt_{\underline{1}})\chi_{\mathfrak{u}}(\bt_{\underline{1}})}{|\bt_{\underline{1}}|^{1/2}} \right|^2 \ll \\
&\ll (Q_1R)^\varepsilon\bigl(Q_1+R+(Q_1R)^{2/3}\bigr)   \sum_{|\bt_{\underline{1}}| \sim B_1 } \mu^2(\bt_{\underline{1}}) \left|  \frac{\chi_{\mathfrak{v}}(\bt_{\underline{1}})}{|\bt_{\underline{1}}|^{1/2}} \right|^2\\
&\ll (Q_1R)^\varepsilon\bigl(Q_1+R+(Q_1R)^{2/3}\bigr)  .
\end{align*}
Therefore, 
\begin{equation} \label{eq:fixed-full}
\begin{aligned}
        \sum_{\substack{|\mathfrak{u}| \sim Q_1,\\ (\mathfrak{u},\mathfrak{v})=1}}  \mu(\mathfrak{u})^2 |\Eu{R}(\mathfrak{u}\mathfrak{v})|^2 & \ll_\varepsilon (Q_1 R)^\varepsilon\bigl(Q_1+R+(Q_1R)^{2/3}\bigr)  \  \sum_{\substack{B_1,B_2 \text{ dyadic} \\ B_1,B_2\gg 1 \\ B_1 B_2^2 \ll R}} \ \sum_{\b_S\in I^S} \frac{1}{|\b_S|^{1/2}} \\
        &  \ll_\varepsilon (Q_1 R)^\varepsilon\bigl(Q_1+R+(Q_1R)^{2/3}\bigr) ,
\end{aligned}
\end{equation}
uniformly in the square-full ideal $\mathfrak{v}\in I(S)_0$.

To quantify the effect of the summation over $\mathfrak{v}$, note that if
\(\mathfrak{a}=\mathfrak{v}_1\) is fixed, then each prime dividing
\(\mathfrak{a}\) occurs in \(\mathfrak{v}\) with one of the exponents
\(2,\ldots,r-1\).  There are at most \((r-2)^{\omega(\mathfrak{a})}\) possibilities, where $\omega(\ai)$ denotes, as usual, the number of prime divisors of $\ai$.  The standard bound
\(
 (r-2)^{\omega(\mathfrak a)}
 \ll_{\varepsilon,r}| \mathfrak{a}|^{\varepsilon}
\)
and the ideal-counting estimate give
\[
 \#\left\{\mathfrak v:
 \begin{array}{l}
 \mathfrak v\ {\rm square\mbox{-}full\ and}\
 r{\rm\mbox{-}th\ power\mbox{-}free},\\[-2pt]
 |\mathfrak{v}_1|\asymp Q_2
 \end{array}\right\}
 \ll_{\varepsilon,r}Q_2^{1+\varepsilon}.
\]
After summing \eqref{eq:fixed-full} over \(\mathfrak{v}\), and  renaming the loss, we obtain \eqref{R_cubic_large_sieve}.
\end{proof}

\begin{prop} \label{secondmoment} \emph{\cite[Prop.~4.2]{DFDS24}}
 Let $\h \in I(S)_0$ and denote  $H=|\mathfrak{c}_{\chi_\h}|/|(\ci_{\chi_\h})_S| $.
 Let $\varepsilon>0$, $Q_1, Q_2>1/2$, and $\frac{1}{2} \leq R \leq (QH)^{100}$. Let  $\varphi: I\to \C$ such that
  $|\varphi(\mathfrak{b})| \ll |\mathfrak{b}|^\varepsilon$, and $\supp(\varphi)\subseteq I_R$.
 For $\q \in I(S)_0$, denote
    \begin{equation*}
       \displaystyle \Eu{R}(\q) := \sum_{0 \neq \b \in I_R } \frac{\varphi(\b) \chi_\q(\b)}{|\b|^{1/2}}.
    \end{equation*}
Then, for any $s\in \C$ with $\Re(s)\in [1/2,1]$, and with the notation $T=1+|\Im(s)|$, we have
    \begin{equation} \label{second}
        \sum_{\substack{\q \in \Eu{F}(Q_1, Q_2) \\ (\q, \h) = 1 \\ \q\h\neq \OO}} |L(s,\chi_{\q\h}) \cdot \Eu{R}(\q\h)|^2
\ll_{\varepsilon,\sigma} (QHT )^\varepsilon Q_2 \left((QHT^d)^{1/2}  R + Q_1 (Q_2 HT^dR^2)^{1/3}\right).
    \end{equation}
\end{prop}

\begin{proof}
Let \begin{equation*}
\Eu{T}(Q_1, Q_2,\h,s):=\sum_{\substack{\q \in \Eu{F}(Q_1, Q_2) \\ (\q, \h) = 1\\ \q\h\neq \OO}} |L(s,\chi_{\q\h}) \cdot \Eu{R}(\q\h)|^2 .
\end{equation*}
Write $\h = \h_\sf\h_\full$, $\q = \q_\sf\q_\full$, and $H_1=|\h_\sf|$, $H_2=|(\h_\full)_1|$, $|\q_\sf| \asymp Q_1$, $|(\q_\full)_1| \asymp Q_2$. Then,
\begin{align*}
|\ci_{\chi_\h}| / |(\ci_{\chi_\h})_S|  = &~|\h_\sf(\h_\full)_1| =   H, \quad
|\ci_{\chi_\q}| / |(\ci_{\chi_\q})_S| =  |\q_\sf(\q_\full)_1| \asymp Q,\\
&\quad \quad \text{ and }  |\ci_{\chi_{\q\h}}| / |(\ci_{\chi_{\q\h}})_S| \asymp  QH.
\end{align*}
We first consider  $s$ for which $\Re(s)=1/2$.  We apply \cref{afe} for $X=1$, and use $|\varepsilon(\q\h,s)|=1$ and $V_{1-s}(y)=\overline{V_s(y)}$, to obtain
\begin{equation}\label{AFEabsbound}
\Eu{T}(Q_1, Q_2,\h,s) \quad \lesssim  \sum_{\substack{\q \in \Eu{F}(Q_1, Q_2) \\ (\q, \h) = 1,\, \q\h\neq \OO}} \left| \sum_{0 \neq \mathfrak{n} \in I} \frac{\chi_{\h\q}(\mathfrak{n})}{|\mathfrak{n}|^{s}}
V_{s} \left(  \frac{|\mathfrak{n}|}{ \sqrt{|d_F||\ci_{\chi_{\q\h}}|} } \right) \right|^2 \cdot |\Eu{R}(\q\h)|^2.
\end{equation}

Denote $Z=(QHT)^{\varepsilon}(QHT^d)^{1/2}$.
Using the decay of $V_{s}(\cdot)$ as in Proposition \ref{decay_lemma} (for any fixed $A\in \Z_{\geq 3}$),  the ideals with $|\mathfrak{n}| \gg Z$ in \eqref{AFEabsbound} can be removed at the cost of an error term $O_A(Z^{-A})$. Multiplying by \(| \Eu{R}(\q\h) |^2\) and summing over \(\mathfrak q\) still produces a negligible quantity because \(R\leq(QH)^{100}\) and the decay order can be chosen to be arbitrarily large.

 Using the integral representation \eqref{Vsdef} for $V_{s}(\cdot)$,
separate variables, move the contour to $\Re(w)=\varepsilon$,
and truncate the integral to $[\varepsilon-i Z^{\varepsilon},\varepsilon+i Z^{\varepsilon}]$
up to negligible error $O_A(Z^{-A})$. Apply the Cauchy--Schwarz inequality on the integral to conclude that the right side of \eqref{AFEabsbound} is
\begin{align} \label{intermed}
\ll &\ Z^{-A}+ Z^{\varepsilon}
 \int_{\varepsilon-i Z^{\varepsilon}}^{\varepsilon+i Z^{\varepsilon}}
 \sum_{\substack{\q \in \Eu{F}(Q_1, Q_2) \\ (\q, \h) = 1,\,\q\h\neq \OO}}
\left| \sum_{\substack{0 \neq \mathfrak{n} \in I  \\ |\mathfrak{n}| \ll Z }} \frac{\chi_{\q\h}(\mathfrak{n})}{|\mathfrak{n}|^{s+w}} \right|^2\cdot |\Eu{R}(\q\h)|^2\ dw.
\end{align}
For a fixed $w$, apply Proposition \ref{cubic_Dirichlet_prop_var}, for
$$
\Eu{R}_{w,\h}(\q)= \sum_{\substack{0 \neq \mathfrak{n} \in I  \\ |\mathfrak{n}| \ll Z }} \frac{\chi_{\q\h}(\mathfrak{n})}{|\mathfrak{n}|^{s+w}} \sum_{0\neq \b\in I_{R}} \frac{\varphi(\b) \chi_{\q\h}(\b)}{|\b|^{1/2}}=\sum_{0\neq \m\in I_{ZR}} \frac{\varphi_{w,\h}(\m) \chi_{\q}(\m)}{|\m|^{1/2}},
$$
where $\displaystyle \varphi_{w,\h}(\m)= \chi_{\h}(\m)\sum_{\substack{\m=\b\mathfrak{n}\\0\neq \b\in I_{R}\\ 0\neq \mathfrak{n}\in I_{Z}} }\frac{\varphi(\b)}{|\mathfrak{n}|^{w+\sqrt{-1}\Im(s)}}$. We
conclude that
\begin{align*}
     \Eu{T}(Q_1, Q_2, \h, s)
     & \ll (ZRQ)^\varepsilon Q_2 \left( Q_1 + ZR + (Q_1 Z R)^{2/3} \right) \\
    &\ll (QHT)^\varepsilon Q_2 \left( Q_1 +  (QHT^d)^{1/2}  R + Q_1(Q_2HR^2T^d)^{1/3} \right),
\end{align*}
which establishes \eqref{second} when $\Re(s)=1/2$.

Now consider \eqref{second} in the case $1/2< \Re(s) \leq 1$.  Lemma \ref{device} (with $\alpha=1/2$ and $U=1$) gives
\begin{equation*}
| L(s,\chi_{\q\h}) |^2
 \leq
 \left|
 \sum_{\mathfrak n}
 \frac{d(\mathfrak n)\chi_{\q\h}(\mathfrak n)}
      {| \mathfrak n |^{s}}e^{-| \mathfrak n|}
 \right| +
 \frac{1}{2\pi}\int_{(1/2)}
 | L(w,\chi_{\q\h}) |^2
| \Gamma(w-s) |\,dw.
\end{equation*}
The first series is \(O_{\Re(s)}(1)\) absolutely.  After multiplication
by \(| \Eu{R}(\mathfrak q\mathfrak h)| ^2\), its family sum is
controlled by \cref{cubic_Dirichlet_prop_var} and is dominated by the
right side of \eqref{second}.  For the integral, the established case of \eqref{second} when $\Re(s)=1/2$ and the estimate
 \begin{equation*} \label{gammabd}
\Gamma(x+iy) \ll_x e^{-|y|},
 \end{equation*}
 give the claimed powers of $T$. 
\end{proof}

\begin{cor} \label{secondmoment-cor}
Let $\psi$ be a Hecke character of order dividing $r$, unramified outside $S$.
 Let $\h \in I(S)_0$, and set $H=|\mathfrak{c}_{\chi_\h}|/ |(\ci_{\chi_\h})_S| $.
 Fix $\varepsilon, Q_1, Q_2>0$. Then,
    \begin{equation} \label{second-cor}
        \sum_{\substack{\q \in \Eu{F}(Q_1, Q_2) \\ (\q, \h) = 1}} |L(1/2, \psi\chi_{\q\h}) |^2
\ll_{\varepsilon} (QH )^\varepsilon Q_2 \left((QH)^{1/2} + Q_1 (Q_2 H)^{1/3}\right).
    \end{equation}
\end{cor}

\begin{proof} We can safely ignore the potential term for which $\psi\chi_{\q\h}$ is the trivial character.
We proceed as in the proof of Proposition \ref{secondmoment}, but with $R=1$ and $\Eu{R}(\q)=1$, for all $\q\in I(S)_0$. Let \begin{equation*}
\Eu{T}(Q_1, Q_2,\h):=\sum_{\substack{\q \in \Eu{F}(Q_1, Q_2)\\ (\q, \h) = 1 \\ \psi\chi_{\q\h} \neq \chi_{\OO} }} |L(1/2,\psi\chi_{\q\h}) |^2 .
\end{equation*}
Write $\h = \h_\sf \h_\full $, $\q = \q_\sf \q_\full$, $H_1=|\h_\sf|$, $H_2=|(\h_\full)_1|$, $|\q_\sf| \asymp Q_1$, $|(\q_\full)_1| \asymp Q_2$. Then,
\begin{equation*}
|\ci_{\chi_\h}| / |(\ci_{\chi_\h})_S|  = ~|\h_\sf(\h_\full)_1|=   H, \quad
|\ci_{\chi_\q}| / |(\ci_{\chi_\q})_S| =  |\q_\sf(\q_\full)_1|\asymp Q.  
\end{equation*}
Since $\psi$ is fixed and  unramified outside $S$, we have $|\ci_{\psi\chi_{\q\h}}| / |(\ci_{\psi\chi_{\q\h}})_S| \asymp  |\ci_{\chi_{\q\h}}| / |(\ci_{\chi_{\q\h}})_S| \asymp QH$.

The equation \eqref{AFEabsbound} becomes

\begin{equation}\label{AFE}
\Eu{T}(Q_1, Q_2,\h) \quad \lesssim  \sum_{\substack{\q \in \Eu{F}(Q_1, Q_2) \\ (\q, \h) = 1 }} \left| \sum_{0 \neq \mathfrak{n} \in I} \frac{\psi\chi_{\q\h}(\mathfrak{n})}{|\mathfrak{n}|^{1/2}}
V_{1/2} \left(  \frac{|\mathfrak{n}|}{ \sqrt{|d_F||\ci_{\psi\chi_{\q\h}}|} } \right) \right|^2 .
\end{equation}
Continuing with the argument, with the notation $Z=(QH)^{\varepsilon}(QH)^{1/2}$, the right side of \eqref{AFE} is
\begin{align} \label{intermed2}
\ll &\ Z^{-A}+ Z^{\varepsilon}
 \int_{\varepsilon-i Z^{\varepsilon}}^{\varepsilon+i Z^{\varepsilon}} \sum_{\substack{\q \in \Eu{F}(Q_1, Q_2) \\ (\q, \h) = 1}}
\left| \sum_{\substack{0 \neq \mathfrak{n} \in I  \\ |\mathfrak{n}| \ll Z }} \frac{\psi\chi_{\q\h}(\mathfrak{n})}{|\mathfrak{n}|^{1/2+w}} \right|^2 dw.
\end{align}
For a fixed $w$, apply Proposition \ref{cubic_Dirichlet_prop_var}, for
$$
\Eu{R}_{w,\h}(\q)= \sum_{\substack{0 \neq \mathfrak{n} \in I  \\ |\mathfrak{n}| \ll Z }} \frac{\psi\chi_{\q\h}(\mathfrak{n})}{|\mathfrak{n}|^{1/2+w}} =\sum_{0\neq \mathfrak{n}\in I_{Z}} \frac{\varphi_{w,\h}(\mathfrak{n}) \chi_{\q}(\mathfrak{n})}{|\mathfrak{n}|^{1/2}},
$$
with $\displaystyle \varphi_{w,\h}(\mathfrak{n})= \chi_{\h}(\mathfrak{n})\frac{\psi(\mathfrak{n})}{|\mathfrak{n}|^w}$, supported on $I_Z$. We
conclude that
\begin{align*}
     \Eu{T}(Q_1, Q_2, \h)
     & \ll (ZQ)^\varepsilon Q_2 \left( Q_1 + Z + (Q_1 Z )^{2/3} \right) \\
    &\ll (QH)^\varepsilon Q_2 \left( Q_1 +  (QH)^{1/2}   + Q_1(Q_2 H)^{1/3} \right),
\end{align*}
which is precisely our claim.
\end{proof}

The following proposition restates \cref{prop:estimate at1}
and is the main result of this appendix.

\begin{prop}\label{prop:estimate at11} Let $\psi$ be a Hecke character of order dividing $r$, unramified outside $S$. Let $ \h\in I(S)_0$, $\varepsilon>0$, and denote $H=|\ci_{\chi_\h}| /  |(\ci_{\chi_\h})_S| $. Then,
\[
\sum_{\substack{\q \in I(S)_0 \\ (\q,\h)=1}}  \frac{|L(1/2, \psi\chi_{\q\h})|^2}{|\q|^{1+\varepsilon}} \ll H^{1/2+\varepsilon}.
\]
\end{prop}

\begin{proof}
Write $\q = \q_\sf \q_\full$, and dyadically partition  $|\q_\sf| \sim Q_1$, and $|(\q_\full)_1|\sim Q_2$ . As usual, we denote $Q=Q_1 Q_2$. Then, by using Corollary \ref{secondmoment-cor} for $\varepsilon/2$ we obtain
\begin{align*}
\sum_{\substack{\q \in I(S)_0 \\ (\q,\h)=1}} & \frac{|L(1/2, \psi\chi_{\q\h})|^2}{|\q|^{1+\varepsilon}}=  \sum_{Q_1, Q_2 \text{ dyadic} }
        \sum_{\substack{\q \in \Eu{F}(Q_1, Q_2) \\ (\q, \h) = 1}} \frac{|L(1/2, \psi\chi_{\q\h}) |^2}{|\q|^{1+\varepsilon}} \\
        &\ll_{\varepsilon}H^{1/2+\varepsilon}\sum_{Q_1, Q_2 \text{ dyadic} }
      \left( Q_1^{-\varepsilon/2-1/2}  Q_2^{-3\varepsilon/2-1/2} +
      Q_1^{-\varepsilon/2}   Q_2^{-3\varepsilon/2-2/3} \right).
\end{align*}
Our claim now follows from the convergence of the sum.
\end{proof}

\bibliographystyle{amsalpha}
\bibliography{HOC_references}

@article{BGL14,
  author  = {Blomer, Valentin and Goldmakher, Leo and Louvel, Beno{\^i}t},
  title   = {{$L$}-functions with {$n$}-th order twists},
  journal = {International Mathematics Research Notices},
  year    = {2014},
  number  = {7},
  pages   = {1925--1955}
}

@article{BB06,
  author  = {Brubaker, Ben and Bump, Daniel},
  title   = {On {Kubota}'s {Dirichlet} series},
  journal = {Journal f{\"u}r die reine und angewandte Mathematik},
  volume  = {598},
  year    = {2006},
  pages   = {159--184}
}

@book{CF,
  editor    = {Cassels, J. W. S. and Fr{\"o}hlich, Albrecht},
  title     = {Algebraic Number Theory},
  publisher = {Academic Press},
  year      = {1967}
}

@article{CFKRS,
  author  = {Conrey, J. Brian and Farmer, David W. and Keating, Jonathan P.
             and Rubinstein, Michael O. and Snaith, Nina C.},
  title   = {Integral moments of {$L$}-functions},
  journal = {Proceedings of the London Mathematical Society},
  series  = {3},
  volume  = {91},
  number  = {1},
  year    = {2005},
  pages   = {33--104}
}

@misc{DFDS24,
  author        = {David, Chantal and de Faveri, Alessandro and Dunn, Alexander
                   and Stucky, Joshua},
  title         = {Non-vanishing for cubic {Hecke} {$L$}-functions},
  year          = {2024},
  eprint        = {2410.03048},
  archiveprefix = {arXiv},
  primaryclass  = {math.NT},
  note          = {arXiv: 2410.03048}
}

@misc{CFD26,
  author        = {Castillo, Cruz and de Faveri, Alexandre and Dunn, Alexander},
  title         = {Non-vanishing for quartic {Hecke} {$L$}-functions and ranks
                   of elliptic curves},
  year          = {2026},
  eprint        = {2604.01316},
  archiveprefix = {arXiv},
  primaryclass  = {math.NT},
  note          = {arXiv: 2604.01316}
}

@article{Di04,
  author  = {Diaconu, Adrian},
  title   = {Mean square values of {Hecke} {$L$}-series formed with
             {$r$}-th order characters},
  journal = {Inventiones Mathematicae},
  volume  = {157},
  number  = {3},
  year    = {2004},
  pages   = {635--684}
}

@article{Di19,
  author  = {Diaconu, Adrian},
  title   = {On the third moment of {$L(1/2,\chi_d)$} {I}: the rational
             function field case},
  journal = {Journal of Number Theory},
  volume  = {198},
  year    = {2019},
  pages   = {1--42}
}

@article{Di-Wh21,
  author  = {Diaconu, Adrian and Whitehead, Ian},
  title   = {On the third moment of {$L(1/2,\chi_d)$} {II}: the number
             field case},
  journal = {Journal of the European Mathematical Society},
  volume  = {23},
  number  = {6},
  year    = {2021},
  pages   = {2051--2070}
}

@article{FHL03,
  author  = {Friedberg, Solomon and Hoffstein, Jeffrey and Lieman, Daniel},
  title   = {Double {Dirichlet} series and the {$n$}-th order twists of
             {Hecke} {$L$}-series},
  journal = {Mathematische Annalen},
  volume  = {327},
  number  = {2},
  year    = {2003},
  pages   = {315--338}
}

@article{GL13,
  author  = {Goldmakher, Leo and Louvel, Beno{\^i}t},
  title   = {A quadratic large sieve inequality over number fields},
  journal = {Mathematical Proceedings of the Cambridge Philosophical Society},
  volume  = {154},
  number  = {2},
  year    = {2013},
  pages   = {193--212}
}

@book{I,
  author    = {Iwaniec, Henryk},
  title     = {Topics in Classical Automorphic Forms},
  series    = {Graduate Studies in Mathematics},
  volume    = {17},
  publisher = {American Mathematical Society},
  address   = {Providence, RI},
  year      = {1997}
}

@book{IK,
  author    = {Iwaniec, Henryk and Kowalski, Emmanuel},
  title     = {Analytic Number Theory},
  series    = {American Mathematical Society Colloquium Publications},
  volume    = {53},
  publisher = {American Mathematical Society},
  address   = {Providence, RI},
  year      = {2004}
}

@book{Kub69,
  author    = {Kubota, Tomio},
  title     = {On Automorphic Forms and the Reciprocity Law in a Number Field},
  publisher = {Kinokuniya Book Store Co.},
  address   = {Tokyo},
  year      = {1969}
}

@article{KP84,
  author  = {Kazhdan, David A. and Patterson, S. J.},
  title   = {Metaplectic forms},
  journal = {Publications Math{\'e}matiques de l'IH{\'E}S},
  volume  = {59},
  year    = {1984},
  pages   = {35--142},
  doi     = {10.1007/BF02698770}
}

@book{N,
  author    = {Neukirch, J{\"u}rgen},
  title     = {Algebraic Number Theory},
  series    = {Grundlehren der mathematischen Wissenschaften},
  volume    = {322},
  publisher = {Springer-Verlag},
  address   = {Berlin},
  year      = {1999}
}

@book{R,
  author    = {Rosen, Michael},
  title     = {Number Theory in Function Fields},
  series    = {Graduate Texts in Mathematics},
  volume    = {210},
  publisher = {Springer-Verlag},
  address   = {New York},
  year      = {2002}
}

@article{S00,
  author  = {Soundararajan, Kannan},
  title   = {Nonvanishing of quadratic {Dirichlet} {$L$}-functions at
             {$s=1/2$}},
  journal = {Annals of Mathematics},
  volume  = {152},
  year    = {2000},
  pages   = {447--488}
}

@article{DR,
  author  = {Dunn, Alexander and Radziwi{\l\l}, Maksym},
  title   = {Bias in cubic {Gauss} sums: {Patterson}'s conjecture},
  journal = {Annals of Mathematics},
  series  = {2},
  volume  = {200},
  number  = {3},
  year    = {2024},
  pages   = {967--1057}
}

@article{EP92,
  author  = {Eckhardt, C. and Patterson, S. J.},
  title   = {On the {Fourier} coefficients of biquadratic theta series},
  journal = {Proceedings of the London Mathematical Society},
  series  = {3},
  volume  = {64},
  number  = {2},
  year    = {1992},
  pages   = {225--264},
  doi     = {10.1112/plms/s3-64.2.225}
}

@article{Ham,
    AUTHOR = {Hamdar, Mohammad H.},
     TITLE = {Hecke {$L$}-functions away from the central line},
   JOURNAL = {Math. Ann.},
  FJOURNAL = {Mathematische Annalen},
    VOLUME = {394},
      YEAR = {2026},
    NUMBER = {3},
     PAGES = {Paper No. 74, 46},
      ISSN = {0025-5831,1432-1807},
   MRCLASS = {11M06 (11F30 11L05 11N75 11R04 11R42)},
  MRNUMBER = {5036071},
       DOI = {10.1007/s00208-026-03341-8},
       URL = {https://doi.org/10.1007/s00208-026-03341-8},
}

@misc{DFL25,
  author        = {David, Chantal and Florea, Alexandra M. and Lal{\'i}n,
                   Matilde},
  title         = {Nonvanishing of {$L$}-functions associated to fixed order
                   characters over function fields},
  year          = {2025},
  eprint        = {2506.07815},
  archiveprefix = {arXiv},
  primaryclass  = {math.NT},
  note          = {arXiv: 2506.07815}
}

@article{Flo17a,
  author  = {Florea, Alexandra M.},
  title   = {The fourth moment of quadratic {Dirichlet} {$L$}-functions over
             function fields},
  journal = {Geometric and Functional Analysis},
  volume  = {27},
  number  = {3},
  year    = {2017},
  pages   = {541--595},
  doi     = {10.1007/s00039-017-0409-8},
  url     = {https://doi.org/10.1007/s00039-017-0409-8}
}

@misc{Sh-St24,
  author        = {Shen, Quanli and Stucky, Joshua},
  title         = {The fourth moment of quadratic {Dirichlet} {$L$}-functions
                   {II}},
  year          = {2024},
  eprint        = {2402.01497},
  archiveprefix = {arXiv},
  primaryclass  = {math.NT},
  note          = {arXiv: 2402.01497}
}

@misc{GR,
  author = {Goel, Shivani and Ray, Anwesh},
  title  = {The second moment of cubic {Dirichlet} {$L$}-functions over
            function fields},
  year   = {2025},
  note   = {arXiv: 2505.12015}
}

@article{Flo17b,
  author  = {Florea, Alexandra M.},
  title   = {The second and third moment of {$L(1/2,\chi)$} in the
             hyperelliptic ensemble},
  journal = {Forum Mathematicum},
  volume  = {29},
  number  = {4},
  year    = {2017},
  pages   = {873--892},
  doi     = {10.1515/forum-2015-0152},
  url     = {https://doi.org/10.1515/forum-2015-0152}
}

@article{BF,
  author  = {Bui, H. M. and Florea, Alexandra M.},
  title   = {Zeros of quadratic {Dirichlet} {$L$}-functions in the
             hyperelliptic ensemble},
  journal = {Transactions of the American Mathematical Society},
  volume  = {370},
  number  = {11},
  year    = {2018},
  pages   = {8013--8045},
  doi     = {10.1090/tran/7317},
  url     = {https://doi.org/10.1090/tran/7317}
}

@misc{BDPW23,
  author        = {Bergstr{\"o}m, Jonas and Diaconu, Adrian and Petersen, Dan
                   and Westerland, Craig},
  title         = {Hyperelliptic curves, the scanning map, and moments of
                   families of quadratic {$L$}-functions},
  year          = {2023},
  eprint        = {2302.07664},
  archiveprefix = {arXiv},
  primaryclass  = {math.NT},
  note          = {arXiv: 2302.07664}
}

@article{Flo17c,
  author  = {Florea, Alexandra M.},
  title   = {Improving the error term in the mean value of
             {$L(1/2,\chi)$} in the hyperelliptic ensemble},
  journal = {International Mathematics Research Notices},
  year    = {2017},
  number  = {20},
  pages   = {6119--6148},
  doi     = {10.1093/imrn/rnv387},
  url     = {https://doi.org/10.1093/imrn/rnv387}
}

@article{BY10,
  author  = {Baier, Stephan and Young, Matthew P.},
  title   = {Mean values with cubic characters},
  journal = {Journal of Number Theory},
  volume  = {130},
  number  = {4},
  year    = {2010},
  pages   = {879--903},
  doi     = {10.1016/j.jnt.2009.11.007},
  url     = {https://doi.org/10.1016/j.jnt.2009.11.007}
}

@book{KS99a,
  author    = {Katz, Nicholas M. and Sarnak, Peter},
  title     = {Random Matrices, {Frobenius} Eigenvalues, and Monodromy},
  series    = {American Mathematical Society Colloquium Publications},
  volume    = {45},
  publisher = {American Mathematical Society},
  address   = {Providence, RI},
  year      = {1999},
  doi       = {10.1090/coll/045},
  url       = {https://doi.org/10.1090/coll/045}
}

@article{Jut81,
  author  = {Jutila, Matti},
  title   = {On the mean value of {$L(1/2,\chi)$} for real characters},
  journal = {Analysis},
  volume  = {1},
  number  = {2},
  year    = {1981},
  pages   = {149--161},
  doi     = {10.1524/anly.1981.1.2.149},
  url     = {https://doi.org/10.1524/anly.1981.1.2.149}
}

@article{ELS,
  author  = {Ellenberg, Jordan S. and Li, Wanlin and Shusterman, Mark},
  title   = {Nonvanishing of hyperelliptic zeta functions over finite
             fields},
  journal = {Algebra \& Number Theory},
  volume  = {14},
  number  = {7},
  year    = {2020},
  pages   = {1895--1909},
  doi     = {10.2140/ant.2020.14.1895},
  url     = {https://doi.org/10.2140/ant.2020.14.1895}
}

@article{DFL22,
  author  = {David, Chantal and Florea, Alexandra M. and Lal{\'i}n, Matilde},
  title   = {The mean values of cubic {$L$}-functions over function fields},
  journal = {Algebra \& Number Theory},
  volume  = {16},
  number  = {5},
  year    = {2022},
  pages   = {1259--1326},
  doi     = {10.2140/ant.2022.16.1259},
  url     = {https://doi.org/10.2140/ant.2022.16.1259}
}

@article{DFL21,
  author  = {David, Chantal and Florea, Alexandra M. and Lal{\'i}n, Matilde},
  title   = {Nonvanishing for cubic {$L$}-functions},
  journal = {Forum of Mathematics, Sigma},
  volume  = {9},
  year    = {2021},
  pages   = {Paper No. e69, 58},
  doi     = {10.1017/fms.2021.62},
  url     = {https://doi.org/10.1017/fms.2021.62}
}

@article{Luo,
  author  = {Luo, Wenzhi},
  title   = {On {Hecke} {$L$}-series associated with cubic characters},
  journal = {Compositio Mathematica},
  volume  = {140},
  number  = {5},
  year    = {2004},
  pages   = {1191--1196},
  doi     = {10.1112/S0010437X0400051X},
  url     = {https://doi.org/10.1112/S0010437X0400051X}
}

@article{Gu,
  author  = {G{\"u}lo{\u{g}}lu, Ahmet M.},
  title   = {Non-vanishing of cubic {Dirichlet} {$L$}-functions over the
             {Eisenstein} field},
  journal = {Proceedings of the American Mathematical Society},
  volume  = {153},
  number  = {5},
  year    = {2025},
  pages   = {1947--1961}
}

@article{DG22,
  author  = {David, Chantal and G{\"u}lo{\u{g}}lu, Ahmet M.},
  title   = {One-level density and non-vanishing for cubic {$L$}-functions
             over the {Eisenstein} field},
  journal = {International Mathematics Research Notices},
  year    = {2022},
  number  = {23},
  pages   = {18833--18873},
  doi     = {10.1093/imrn/rnab240},
  url     = {https://doi.org/10.1093/imrn/rnab240}
}

@article{GY,
  author  = {G{\"u}lo{\u{g}}lu, Ahmet M. and Yesilyurt, Hamza},
  title   = {Mollified moments of cubic {Dirichlet} {$L$}-functions over the
             {Eisenstein} field},
  journal = {Journal of Mathematical Analysis and Applications},
  volume  = {533},
  number  = {2},
  year    = {2024},
  pages   = {Paper No. 128014, 49},
  doi     = {10.1016/j.jmaa.2023.128014},
  url     = {https://doi.org/10.1016/j.jmaa.2023.128014}
}

@article{DGH03,
  author  = {Diaconu, Adrian and Goldfeld, Dorian and Hoffstein, Jeffrey},
  title   = {Multiple {Dirichlet} series and moments of zeta and
             {$L$}-functions},
  journal = {Compositio Mathematica},
  volume  = {139},
  number  = {3},
  year    = {2003},
  pages   = {297--360},
  doi     = {10.1023/B:COMP.0000018137.38458.68},
  url     = {https://doi.org/10.1023/B:COMP.0000018137.38458.68}
}

@article{Su82,
  author  = {Suzuki, Toshiaki},
  title   = {Some results on the coefficients of the biquadratic theta series},
  journal = {Journal f{\"u}r die reine und angewandte Mathematik},
  volume  = {340},
  year    = {1983},
  pages   = {70--117}
}

@book{H,
  author    = {Hasse, Helmut},
  title     = {Vorlesungen {\"u}ber Zahlentheorie},
  series    = {Die Grundlehren der mathematischen Wissenschaften},
  volume    = {59},
  publisher = {Springer-Verlag},
  address   = {Berlin--G{\"o}ttingen--Heidelberg},
  year      = {1950}
}

@article{HB,
  author  = {Heath-Brown, D. R.},
  title   = {Kummer's conjecture for cubic {Gauss} sums},
  journal = {Israel Journal of Mathematics},
  volume  = {120},
  number  = {Part A},
  year    = {2000},
  pages   = {97--124},
  doi     = {10.1007/s11856-000-1273-y},
  url     = {https://doi.org/10.1007/s11856-000-1273-y}
}

@incollection{Hof93,
  author    = {Hoffstein, Jeffrey},
  title     = {Eisenstein series and theta functions on the metaplectic group},
  booktitle = {Theta Functions: From the Classical to the Modern},
  editor    = {Murty, M. Ram},
  series    = {CRM Proceedings \& Lecture Notes},
  volume    = {1},
  publisher = {American Mathematical Society},
  address   = {Providence, RI},
  year      = {1993},
  pages     = {65--104}
}

@article{Pat1,
  author  = {Patterson, S. J.},
  title   = {A cubic analogue of the theta series},
  journal = {Journal f{\"u}r die reine und angewandte Mathematik},
  volume  = {296},
  year    = {1977},
  pages   = {125--161},
  doi     = {10.1515/crll.1977.296.125},
  url     = {https://doi.org/10.1515/crll.1977.296.125}
}

@misc{DM23,
  author  = {David, Chantal and Meisner, Patrick},
  title   = {Expected values of cubic {Dirichlet} {$L$}-functions away from
             the central point},
  year    = {2023},
  eprint  = {2305.15120},
  archiveprefix = {arXiv},
  primaryclass  = {math.NT},
  note    = {arXiv: 2305.15120}
}

@article{AK14,
  author  = {Andrade, Julio C. and Keating, Jonathan P.},
  title   = {Conjectures for the integral moments and ratios of
             {$L$}-functions over function fields},
  journal = {Journal of Number Theory},
  volume  = {142},
  year    = {2014},
  pages   = {102--148}
}

@article{CF00,
  author  = {Conrey, J. Brian and Farmer, David W.},
  title   = {Mean values of {$L$}-functions and symmetry},
  journal = {International Mathematics Research Notices},
  year    = {2000},
  number  = {17},
  pages   = {883--908}
}

@article {DPP23,
    AUTHOR = {Diaconu, Adrian and Pa\c{s}ol, Vicen\c{t}iu and Popa,
              Alexandru A.},
     TITLE = {Quadratic {W}eyl group multiple {D}irichlet series of type
              {$D^{(1)}_4$}},
   JOURNAL = {Amer. J. Math.},
  FJOURNAL = {American Journal of Mathematics},
    VOLUME = {147},
      YEAR = {2025},
    NUMBER = {5},
     PAGES = {1159--1212},
      ISSN = {0002-9327,1080-6377},
   MRCLASS = {11F68 (11M06 11M32 11R58 17B67)},
  MRNUMBER = {4969289},
MRREVIEWER = {Qiao\ Zhang},
       DOI = {10.1353/ajm.2025.a971089},
       URL = {https://doi.org/10.1353/ajm.2025.a971089},
}

@article{FF04,
  author  = {Fisher, Benji and Friedberg, Solomon},
  title   = {Double {Dirichlet} series over function fields},
  journal = {Compositio Mathematica},
  volume  = {140},
  number  = {3},
  year    = {2004},
  pages   = {613--630}
}

@article{Ho-Ro92,
  author  = {Hoffstein, Jeffrey and Rosen, Michael},
  title   = {Average values of {$L$}-series in function fields},
  journal = {Journal f{\"u}r die reine und angewandte Mathematik},
  volume  = {426},
  year    = {1992},
  pages   = {117--150},
  doi     = {10.1515/crll.1992.426.117},
  url     = {https://doi.org/10.1515/crll.1992.426.117}
}

@article{KeSn00a,
  author  = {Keating, Jonathan P. and Snaith, Nina C.},
  title   = {Random matrix theory and {$\zeta(1/2+it)$}},
  journal = {Communications in Mathematical Physics},
  volume  = {214},
  number  = {1},
  year    = {2000},
  pages   = {57--89}
}

@article{KeSn00b,
  author  = {Keating, Jonathan P. and Snaith, Nina C.},
  title   = {Random matrix theory and {$L$}-functions at {$s=1/2$}},
  journal = {Communications in Mathematical Physics},
  volume  = {214},
  number  = {1},
  year    = {2000},
  pages   = {91--110}
}

@misc{MPPRW24,
  author        = {Miller, Jeremy and Patzt, Peter and Petersen, Dan and
                   Randal-Williams, Oscar},
  title         = {Uniform twisted homological stability},
  year          = {2024},
  eprint        = {2402.00354},
  archiveprefix = {arXiv},
  primaryclass  = {math.AT},
  note          = {arXiv: 2402.00354}
}

@article{Ru-Wu15,
  author  = {Rubinstein, Michael O. and Wu, Kaiyu},
  title   = {Moments of zeta functions associated to hyperelliptic curves
             over finite fields},
  journal = {Philosophical Transactions of the Royal Society A},
  volume  = {373},
  number  = {2040},
  year    = {2015},
  pages   = {20140307},
  note    = {37 pages}
}

@article{Young13,
  author  = {Young, Matthew P.},
  title   = {The third moment of quadratic {Dirichlet} {$L$}-functions},
  journal = {Selecta Mathematica},
  series  = {New Series},
  volume  = {19},
  number  = {2},
  year    = {2013},
  pages   = {509--543}
}

@misc{HZ25,
  author        = {Hong, Ziwei and Zheng, Zhiyong},
  title         = {Twisted second moment of primitive cubic {$L$}-functions},
  year          = {2025},
  eprint        = {2506.14656},
  archiveprefix = {arXiv},
  primaryclass  = {math.NT},
  doi           = {10.48550/arXiv.2506.14656},
  url           = {https://arxiv.org/abs/2506.14656},
  note          = {arXiv: 2506.14656}
}

@unpublished{DMPPW,
  author = {Diaconu, Adrian and Miller, Jeremy and Patzt, Peter and Petersen, Dan and Wang, Victor},
  title  = {Moments for higher degree extensions},
  note   = {In preparation}
}

@misc{FDH26,
  author      = {De Faveri, Alexandre and Dunn, Alexander and Hoffstein, Jeffrey},
  title           = {Non-orthogonality of the cubic and quartic large sieves via Rankin-Selberg},
  year          = {2026},
  eprint        = {2607.07911},
  archiveprefix = {arXiv},
  primaryclass  = {math.NT},
  url           = {https://arxiv.org/abs/2607.07911},
  note          = {arXiv: 2607.07911}
}
\Addresses

\end{document}